\documentclass[11pt]{article}
\newcommand{\sfrac}[2]{{\textstyle\frac{#1}{#2}}}
\usepackage{amsmath,amssymb}
\usepackage{graphicx}  
\usepackage[title]{appendix}
\usepackage{color}
\newcommand{\Ex}{\mathbb{E}}
 \renewcommand{\Pr}{{\mathbb{P}}}
 \renewcommand{\SS}{\mathcal S}
 \newcommand{\FF}{\mathcal F}

\newcommand{\Ints}{{\mathbb{Z}}}
 \newcommand{\eps}{\varepsilon}
 \newcommand{\bX}{\mathbf X}
 \newcommand{\bY}{\mathbf Y}
 \newcommand{\bZ}{\mathbf Z}
 \newcommand{\bt}{\mathbf t}
 \newcommand{\tree}{\mathbb{T}}
 \newcommand{\Reals}{{\mathbb{R}}}
 \newcommand\CTCS{\operatorname{CTCS}}
\newcommand\DTCS{\mathrm{DTCS}}
\newtheorem{Lemma}{Lemma}
\newtheorem{Theorem}[Lemma]{Theorem}
\newtheorem{Proposition}[Lemma]{Proposition}

\newtheorem{Remark}[Lemma]{Remark}
\newtheorem{OP}{Open Problem}
\newtheorem{Ansatz}[Lemma]{Ansatz}

\newenvironment{romenumerate}[1][-10pt]{
\addtolength{\leftmargini}{#1}\begin{enumerate}
 }{\end{enumerate}}

 \newcommand{\qed}{\ \ \rule{1ex}{1ex}}
 \newcommand{\var}{\mathrm{var}}
 \renewcommand{\dh}{\bar{h}}
 \newcommand{\dl}{\bar{l}}
 
\newcommand\nn{^{(n)}}

\newcommand\nnn{^{[n]}}
\newcommand\bbN{\mathbb N}
\newcommand\set[1]{\ensuremath{\{#1\}}}
\newcommand\setn{\set{1,\dots,n}}
\newcommand\bigabs[1]{\bigl\lvert#1\bigr\rvert}

\newcommand\KX{K}

\newcommand\XP{P}

\newcommand\ntoo{\ensuremath{{n\to\infty}}}

\newcommand\oi{\ensuremath{[0,1]}}

\newcommand\oio{\ensuremath{[0,1)}}
\newcommand\cXt{\XP_{t,1}}
\newcommand\gG{\Gamma}
\newcommand\cL{{\mathcal L}}
\newcommand\bigpar[1]{\bigl(#1\bigr)}
\newcommand\Bigpar[1]{\Bigl(#1\Bigr)}

\newcommand\intoo{\int_0^\infty}
\newcommand\dd{\,\mathrm{d}}
\newcommand\ddx{\mathrm{d}}

\newcommand\eqd{\overset{\mathrm{d}}{=}}
\newcommand\intoi{\int_0^1}
\newcommand\qw{^{-1}}

\newcommand\Exp{\operatorname{Exp}}
\newcommand\chpsi{\tilde\psi}
\newcommand\chnu{\tilde\nu}
\newcommand\chf{\tilde{f}}
\newcommand\chg{\tilde{g}}

\newcommand{\sumj}{\sum_{j=1}^\infty}
\newcommand{\sumk}{\sum_{k=1}^\infty}

\newcommand\gsx{\gs_*}
\newcommand\gLX{\gL^*}

\newcommand\gd{\delta}

\newcommand\gam{\gamma}

\newcommand\gl{\lambda}
\newcommand\gL{\Lambda}

\newcommand\gs{\sigma}

\newcommand\gU{\Upsilon}

\newcommand{\tend}{\longrightarrow}

\newcommand\pto{\overset{\mathrm{p}}{\tend}}
\newcommand\asto{\overset{\mathrm{a.s.}}{\tend}}

\newcommand\lrabs[1]{\left\lvert#1\right\rvert}
\newcommand\ttoo{\ensuremath{{t\to\infty}}}

\newcommand\marginal[1]{\marginpar[\raggedleft\tiny #1]{\raggedright\tiny#1}}
\newcommand\SJ{\marginal{SJ} }

\newcommand\xsigma{\bar\sigma}

\begin{document}

\title{The Critical Beta-splitting Random Tree II: Overview and Open Problems} 
 \author{David J. Aldous\thanks{Department of Statistics,
 367 Evans Hall \#\  3860,
 U.C. Berkeley CA 94720;  aldousdj@berkeley.edu;
  www.stat.berkeley.edu/users/aldous.}
 \and
 Svante Janson\thanks{Department of Mathematics, Uppsala University, P.O.Box 480, SE-751 06 Uppsala Sweden; 
   svante.janson@math.uu.se; www2.math.uu.se/{\tiny$\sim$}svantejs.
 }}

\date{April 17, 2025} 

\newpage
 \maketitle
 
 
\vspace*{-0.2in}
 
 \begin{abstract} 
 In the critical beta-splitting model of a random $n$-leaf rooted tree, clades are recursively (from the root) split into sub-clades, and a clade of $m$ leaves is split into sub-clades 
 containing  $i$ and $m-i$ leaves  with probabilities $\propto 1/(i(m-i))$. 
 Study of structure theory and explicit quantitative aspects of this model (in discrete or continuous versions) is an active research topic.
 For many results there are different proofs, probabilistic or analytic, so the model provides a testbed for  a ``compare and contrast" discussion of techniques. 
 This article provides an overview of results proved in the sequence of similarly-titled articles I, III, IV 
 \cite{beta1,beta3-arxiv,beta4-arxiv} and  related articles \cite{HDchain,iksanovHD,iksanovCLT,kolesnik}.
 We mostly do not repeat proofs given elsewhere: instead we seek to paint a ``Big Picture" via graphics and heuristics, and emphasize open problems.
 
Our discussion is centered around three categories of results. 

  \noindent
  (i) There is a CLT for leaf heights, and the analytic proofs can be extended to provide surprisingly precise analysis of other height-related aspects.
  
    \noindent
  (ii) There is an explicit description of the limit {\em fringe distribution} relative to a random leaf, whose graphical representation  is essentially the format of the cladogram representation of biological phylogenies.

  \noindent
 (iii) There is a canonical embedding of the discrete model into a continuous-time model, that is a random tree  $\CTCS(n)$ on $n$ leaves with real-valued edge lengths, and this model
 turns out more convenient to study. 
 The family $(\CTCS(n), n \ge 2)$ is consistent under a ``delete random leaf and prune" operation.
  That leads to an explicit inductive construction of $(\CTCS(n), n \ge 2)$  as $n$ increases, and then to a limit structure $\CTCS(\infty)$ formalized via exchangeable partitions, in some ways analogous to the Brownian continuum random tree.
  
   Many open problems remain, in particular to elucidate a relation between $\CTCS(\infty)$ and the $\beta(2,1)$ coalescent.

 \end{abstract}
 
\tableofcontents

 \newpage
\section{Introduction}
This article describes the current state of active research on a certain random tree model. 
The model arose as a toy model for phylogenetic trees, designed to mimic the uneven splits observed in real world examples  (see Section \ref{sec:motive}).
The model turns out to have a rich mathematical structure.
There are many questions one can ask (in addition to those suggested by the
phylogenetic context),  and many different proof techniques can be
exploited. 
This article is centered around 
three foundational results: 
Theorems \ref{TNorm} (CLT for leaf heights),
\ref{T:alimit} (occupation measure),
and \ref{T:consistent} (the consistency property),
and indeed each has several different proofs, 
probabilistic or analytic, so the model provides a testbed for  a ``compare and contrast" discussion of techniques. 

We will re-state most of the main results from the technical articles  \cite{beta1,beta3-arxiv,beta4-arxiv}, but mostly without proofs.
Instead we give several ``illustrative proofs" to illustrate a proof technique, and we give several ``alternative proofs" not  published elsewhere,
to reinforce the ``compare and contrast" theme.

Open problems\footnote{Some are ``open" merely because we have not thought about them.} are noted throughout, and enumerated separately as Open Problems 1 - \ref{OP:symmetry}.

 \section{The critical beta-splitting model of random trees}
  \label{sec:tree}
For $m \ge 2$, consider the distribution $(q(m,i),\ 1 \le i \le m-1)$ constructed to be proportional to $\frac{1}{i(m-i)}$.
 Explicitly 
\begin{equation}\label{01}
q(m,i)=\tfrac{m}{2h_{m-1}}\cdot\tfrac{1}{i(m-i)},\,\,1\le i\le m-1,
\end{equation}
where $h_{m-1}$ is the harmonic sum $\sum_{i=1}^{m-1}1/i$. 
Now fix $n \ge 2$.
Consider the process of constructing a random tree by recursively splitting the integer interval 
$[n] = \{1,2,\ldots,n\}$ of ``leaves" as follows.
First specify that there is a left edge and a right edge at the root,
 leading to a left subtree which will have the\footnote{$G$ for {\em gauche} (left) because later we use $L_n$ for leaf hop-height.}   $G_n$ leaves $\{1,\ldots,G_n\}$
 and a right subtree which will have the $R_n = n - G_n$ leaves $\{G_n + 1,\ldots, n\}$, where $G_n$ 
 (and also $R_n$, by symmetry) has distribution $q(n,\cdot)$. 
 Recursively, a subinterval with $m\ge 2$ leaves is split into two subintervals of random size
  from the distribution $q(m,\cdot)$. 
  Continue until reaching intervals of size $1$, which are the leaves.
  That is a discrete time construction, which we call\footnote{DTCS and CTCS are abbreviations for Discrete Time  Critical Splitting and Continuous Time Critical Splitting, for reasons explained in Section \ref{sec:motive}.} $\DTCS(n)$.
 Figure \ref{Fig:1d} (left) illustrates schematically the construction as interval-splitting, 
  with the natural tree structure shown in Figure \ref{Fig:1d} (center and right).
  
\begin{figure}[ht]
\setlength{\unitlength}{0.045in}
\begin{picture}(60,50)
\put(-5,0){\vector(0,1){50}}
\put(-9,-1){0}
\put(-9,24){5}
\put(-10,49){10}
\put(-16,29){time}
\multiput(-5,0)(0,5){11}{\line(-1,0){2}}
\multiput(0,0)(3,0){20}{\circle*{0.9}}
\put(-1.3,-0.5){[}
\put(57.7,-0.5){]}
\multiput(0,5)(3,0){20}{\circle*{0.9}}
\put(-1.3,4.5){[}
\put(6.7,4.5){]}
\put(7.7,4.5){[}
\put(57.7,4.5){]}
\multiput(0,10)(3,0){20}{\circle*{0.9}}
\put(-1.3,9.5){[}
\put(3.7,9.5){]}
\put(4.7,9.5){[}
\put(6.7,9.5){]}
\put(7.7,9.5){[}
\put(9.7,9.5){]}
\put(10.7,9.5){[}
\put(57.7,9.5){]}
\multiput(0,15)(3,0){2}{\circle*{0.9}}
\put(-1.3,14.5){[}
\put(0.7,14.5){]}
\put(1.7,14.5){[}
\put(3.7,14.5){]}
\multiput(12,15)(3,0){16}{\circle*{0.9}}
\put(10.7,14.5){[}
\put(21.7,14.5){]}
\put(22.7,14.5){[}
\put(57.7,14.5){]}
\multiput(12,20)(3,0){16}{\circle*{0.9}}
\put(10.7,19.5){[}
\put(18.7,19.5){]}
\put(19.7,19.5){[}
\put(21.7,19.5){]}
\put(22.7,19.5){[}
\put(54.7,19.5){]}
\put(55.7,19.5){[}
\put(57.7,19.5){]}
\multiput(12,25)(3,0){3}{\circle*{0.9}}
\put(10.7,24.5){[}
\put(15.7,24.5){]}
\put(16.7,24.5){[}
\put(18.7,24.5){]}
\multiput(24,25)(3,0){11}{\circle*{0.9}}
\put(22.7,24.5){[}
\put(48.7,24.5){]}
\put(49.7,24.5){[}
\put(54.7,24.5){]}
\multiput(12,30)(3,0){2}{\circle*{0.9}}
\put(10.7,29.5){[}
\put(12.7,29.5){]}
\put(13.7,29.5){[}
\put(15.7,29.5){]}
\multiput(24,30)(3,0){11}{\circle*{0.9}}
\put(22.7,29.5){[}
\put(36.7,29.5){]}
\put(37.7,29.5){[}
\put(48.7,29.5){]}
\put(49.7,29.5){[}
\put(51.7,29.5){]}
\put(52.7,29.5){[}
\put(54.7,29.5){]}
\multiput(24,35)(3,0){9}{\circle*{0.9}}
\put(22.7,34.5){[}
\put(30.7,34.5){]}
\put(31.7,34.5){[}
\put(36.7,34.5){]}
\put(37.7,34.5){[}
\put(39.7,34.5){]}
\put(40.7,34.5){[}
\put(48.7,34.5){]}
\multiput(24,40)(3,0){5}{\circle*{0.9}}
\put(22.7,39.5){[}
\put(27.7,39.5){]}
\put(28.7,39.5){[}
\put(30.7,39.5){]}
\put(31.7,39.5){[}
\put(33.7,39.5){]}
\put(34.7,39.5){[}
\put(36.7,39.5){]}
\multiput(42,40)(3,0){3}{\circle*{0.9}}
\put(40.7,39.5){[}
\put(42.7,39.5){]}
\put(43.7,39.5){[}
\put(48.7,39.5){]}
\multiput(24,45)(3,0){2}{\circle*{0.9}}
\put(22.7,44.5){[}
\put(24.7,44.5){]}
\put(25.7,44.5){[}
\put(27.7,44.5){]}
\multiput(45,45)(3,0){2}{\circle*{0.9}}
\put(43.7,44.5){[}
\put(45.7,44.5){]}
\put(46.7,44.5){[}
\put(48.7,44.5){]}
\put(66,15){\line(1,0){1.5}}
\put(66.75,15){\line(0,-1){5}}
\put(66.75,10){\line(1,0){2.25}}
\put(67.5,10){\line(0,-1){5}}
\put(82.5,10){\line(0,-1){5}}
\put(67.5,5){\line(1,0){15}}
\put(83.25,15){\line(0,-1){5}}
\put(70.5,10){\line(1,0){12.75}}
\put(74.25,20){\line(0,-1){5}}
\put(86.25,20){\line(0,-1){5}}
\put(74.25,15){\line(1,0){12}}
\put(73.5,25){\line(0,-1){5}}
\put(73.5,20){\line(1,0){3}}
\put(85.5,25){\line(0,-1){5}}
\put(85.5,20){\line(1,0){9}}
\put(72.75,30){\line(0,-1){5}}
\put(72.75,25){\line(1,0){2.5}}
\put(72,30){\line(1,0){1.5}}
\put(84,30){\line(0,-1){5}}
\put(91.5,30){\line(1,0){1.5}}
\put(92.25,30){\line(0,-1){5}}
\put(84,25){\line(1,0){8.25}}
\put(81,35){\line(0,-1){5}}
\put(87.75,35){\line(0,-1){5}}
\put(81,30){\line(1,0){6.75}}
\put(79.5,40){\line(0,-1){5}}
\put(83.25,40){\line(0,-1){5}}
\put(79.5,35){\line(1,0){3.75}}
\put(88.5,40){\line(0,-1){5}}
\put(85.5,35){\line(1,0){3}}
\put(78.75,45){\line(0,-1){5}}
\put(78.75,40){\line(1,0){2.25}}
\put(82.5,40){\line(1,0){1.5}}
\put(89.25,45){\line(0,-1){5}}
\put(87,40){\line(1,0){2.25}}
\put(78,45){\line(1,0){1.5}}
\put(88.5,45){\line(1,0){1.5}}
\put(66,15){\circle*{0.9}}
\put(67.5,15){\circle*{0.9}}
\put(69,10){\circle*{0.9}}
\put(70.5,10){\circle*{0.9}}
\put(72,30){\circle*{0.9}}
\put(73.5,30){\circle*{0.9}}
\put(75,25){\circle*{0.9}}
\put(76.5,20){\circle*{0.9}}
\put(78,45){\circle*{0.9}}
\put(79.5,45){\circle*{0.9}}
\put(81,40){\circle*{0.9}}
\put(82.5,40){\circle*{0.9}}
\put(84,40){\circle*{0.9}}
\put(85.5,35){\circle*{0.9}}
\put(87,40){\circle*{0.9}}
\put(88.5,45){\circle*{0.9}}
\put(90,45){\circle*{0.9}}
\put(91.5,30){\circle*{0.9}}
\put(93,30){\circle*{0.9}}
\put(94.5,20){\circle*{0.9}}
\put(80.25,5){\line(0,-1){5}}
\put(101,15){\circle{0.9}}
\put(102.5,15){\circle{0.9}}
\put(101.75,15){\line(0,-1){10}}
\put(102.5,10){\circle{0.9}}
\put(101.5,5){\line(1,0){14.5}}
\put(111,5){\line(0,-1){5}}
\put(116,5){\line(0,1){10}}
\put(115.5,10){\circle{0.9}}
\put(107,15){\line(1,0){13}}
\put(107,15){\line(0,1){15}}
\put(107.75,20){\circle{0.9}}
\put(107.75,25){\circle{0.9}}
\put(107.75,30){\circle{0.9}}
\put(106.25,30){\circle{0.9}}
\put(120,15){\line(0,1){10}}
\put(120.75,20){\circle{0.9}}
\put(116,25){\line(1,0){10}}
\put(126,25){\line(0,1){5}}
\put(126.75,30){\circle{0.9}}
\put(125.25,30){\circle{0.9}}
\put(116,25){\line(0,1){5}}
\put(111,30){\line(1,0){11}}
\put(122,30){\line(0,1){15}}
\put(121.25,35){\circle{0.9}}
\put(121.25,40){\circle{0.9}}
\put(121.25,45){\circle{0.9}}
\put(122.75,45){\circle{0.9}}
\put(111,30){\line(0,1){5}}
\put(107,35){\line(1,0){8}}
\put(107,35){\line(0,1){10}}
\put(107.75,40){\circle{0.9}}
\put(107.75,45){\circle{0.9}}
\put(106.25,45){\circle{0.9}}
\put(115,35){\line(0,1){5}}
\put(115.75,40){\circle{0.9}}
\put(114.25,40){\circle{0.9}}

\end{picture}
\caption{Equivalent representations of a realization of DTCS(20).}
\label{Fig:1d}
\end{figure}
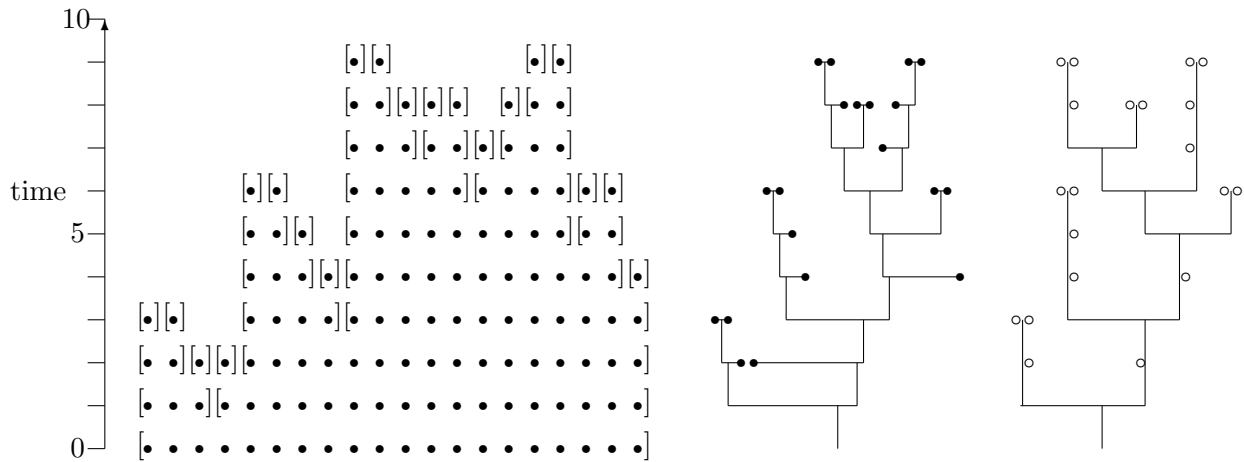

This discrete time model  was introduced and briefly studied many years ago in \cite{me_clad}.
A more recent observation was that an associated continuous time model has appealing structural properties, and that observation was major
motivation for the current project.  
We define the  associated continuous time model $\CTCS(n)$, by declaring that an interval with $m \ge 2$ leaves is split at rate $h_{m-1}$,
 that is after an Exponential($h_{m-1}$) random time. 
 Figure \ref{Fig:1c} shows a schematic realization of $\CTCS(20)$ as a ``continuization" of the realization of $\DTCS(20)$ in Figure \ref{Fig:1d}. 
 Figure \ref{Fig:400} shows an actual realization of $\CTCS(400)$.

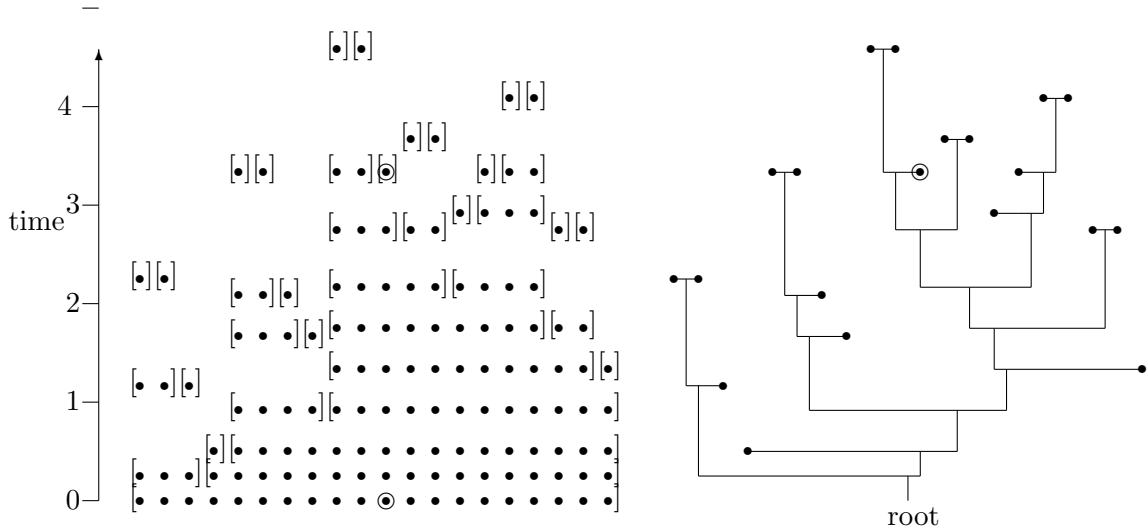
\begin{figure}
\setlength{\unitlength}{0.043in}
\begin{picture}(60,60)
\put(-5,0){\vector(0,1){55}}
\put(-16,33){time}
\multiput(-5,0)(0,12){6}{\line(-1,0){2}}
\put(-9,-1){0}
\put(-9,11){1}
\put(-9,23){2}
\put(-9,35){3}
\put(-10,47){4}
\multiput(0,0)(3,0){20}{\circle*{0.9}}
\put(30,0){\circle{1.8}}
\put(30,40){\circle{1.8}}
\put(-1.3,-0.5){[}
\put(57.7,-0.5){]}
\multiput(0,3)(3,0){20}{\circle*{0.9}}
\put(-1.3,2.5){[}
\put(6.7,2.5){]}
\put(7.7,2.5){[}
\put(57.7,2.5){]}

\multiput(0,14)(3,0){3}{\circle*{0.9}}
\multiput(0,27)(3,0){2}{\circle*{0.9}}

\put(-1.3,13.5){[}
\put(3.7,13.5){]}
\put(4.7,13.5){[}
\put(6.7,13.5){]}

\put(-1.3,26.5){[}
\put(1.7,26.5){[}
\put(0.7,26.5){]}
\put(3.7,26.5){]}

\multiput(9,6)(3,0){17}{\circle*{0.9}}

\put(7.7,5.5){[}
\put(9.7,5.5){]}
\put(10.7,5.5){[}
\put(57.7,5.5){]}

\multiput(12,11)(3,0){16}{\circle*{0.9}}

\put(10.7,10.5){[}
\put(21.7,10.5){]}
\put(22.7,10.5){[}
\put(57.7,10.5){]}

\multiput(12,20)(3,0){4}{\circle*{0.9}}
\multiput(24,16)(3,0){12}{\circle*{0.9}}

\put(10.7,19.5){[}
\put(18.7,19.5){]}
\put(19.7,19.5){[}
\put(21.7,19.5){]}
\put(22.7,15.5){[}
\put(54.7,15.5){]}
\put(55.7,15.5){[}
\put(57.7,15.5){]}

\multiput(12,25)(3,0){3}{\circle*{0.9}}
\put(10.7,24.5){[}
\put(15.7,24.5){]}
\put(16.7,24.5){[}
\put(18.7,24.5){]}

\multiput(24,21)(3,0){11}{\circle*{0.9}}
\put(22.7,20.5){[}
\put(48.7,20.5){]}
\put(49.7,20.5){[}
\put(54.7,20.5){]}

\multiput(12,40)(3,0){2}{\circle*{0.9}}
\put(10.7,39.5){[}
\put(12.7,39.5){]}
\put(13.7,39.5){[}
\put(15.7,39.5){]}

\multiput(24,26)(3,0){9}{\circle*{0.9}}
\put(22.7,25.5){[}
\put(36.7,25.5){]}
\put(37.7,25.5){[}
\put(48.7,25.5){]}

\multiput(51,33)(3,0){2}{\circle*{0.9}}
\put(49.7,32.5){[}
\put(51.7,32.5){]}
\put(52.7,32.5){[}
\put(54.7,32,5){]}

\multiput(24,33)(3,0){5}{\circle*{0.9}}
\put(22.7,32.5){[}
\put(30.7,32.5){]}
\put(31.7,32.5){[}
\put(36.7,32.5){]}

\multiput(39,35)(3,0){4}{\circle*{0.9}}
\put(37.7,34.5){[}
\put(39.7,34.5){]}
\put(40.7,34.5){[}
\put(48.7,34.5){]}

\multiput(24,40)(3,0){3}{\circle*{0.9}}
\put(22.7,39.5){[}
\put(27.7,39.5){]}
\put(28.7,39.5){[}
\put(30.7,39.5){]}

\multiput(33,44)(3,0){2}{\circle*{0.9}}
\put(31.7,43.5){[}
\put(33.7,43.5){]}
\put(34.7,43.5){[}
\put(36.7,43.5){]}

\multiput(42,40)(3,0){3}{\circle*{0.9}}
\put(40.7,39.5){[}
\put(42.7,39.5){]}
\put(43.7,39.5){[}
\put(48.7,39.5){]}

\multiput(24,55)(3,0){2}{\circle*{0.9}}
\put(22.7,54.5){[}
\put(24.7,54.5){]}
\put(25.7,54.5){[}
\put(27.7,54.5){]}

\multiput(45,49)(3,0){2}{\circle*{0.9}}
\put(43.7,48.5){[}
\put(45.7,48.5){]}
\put(46.7,48.5){[}
\put(48.7,48.5){]}


\put(68,3){\line(1,0){27}}
\put(93.5,3){\line(0,-1){3}}

\put(66.5,14){\line(1,0){4.5}}
\put(68,3){\line(0,1){11}}
\put(71,14){\circle*{0.9}}

\put(65,27){\line(1,0){3}}
\put(66.5,14){\line(0,1){13}}
\put(65,27){\circle*{0.9}}
\put(68,27){\circle*{0.9}}

\put(95,3){\line(0,1){3}}
\put(74,6){\line(1,0){25.5}}
\put(74,6){\circle*{0.9}}

\put(99.5,6){\line(0,1){5}}
\put(81.5,11){\line(1,0){24}}
\put(81.5,11){\line(0,1){9}}

\put(81.5,20){\line(1,0){4.5}}
\put(86,20){\circle*{0.9}}
\put(81.5,20){\line(-1,0){1.5}}

\put(80,20){\line(0,1){5}}
\put(80,25){\line(-1,0){1.5}}
\put(80,25){\line(1,0){3}}
\put(83,25){\circle*{0.9}}

\put(78.5,25){\line(0,1){15}}
\put(77,40){\line(1,0){3}}
\put(77,40){\circle*{0.9}}
\put(80,40){\circle*{0.9}}

\put(105.5,11){\line(0,1){5}}
\put(105.5,16){\line(1,0){16.5}}
\put(122,16){\circle*{0.9}}
\put(105.5,16){\line(-1,0){1.5}}

\put(104,16){\line(0,1){5}}
\put(104,21){\line(1,0){13.5}}
\put(104,21){\line(-1,0){3}}

\put(101,21){\line(0,1){5}}
\put(101,26){\line(1,0){7.5}}
\put(101,26){\line(-1,0){6}}

\put(95,26){\line(0,1){7}}
\put(95,33){\line(1,0){4.5}}
\put(95,33){\line(-1,0){3}}

\put(92,33){\line(0,1){7}}
\put(92,40){\line(1,0){3}}
\put(92,40){\line(-1,0){1.5}}
\put(95,40){\circle*{0.9}}
\put(95,40){\circle{1.8}}

\put(90.5,40){\line(0,1){15}}
\put(89,55){\line(1,0){3}}
\multiput(89,55)(3,0){2}{\circle*{0.9}}

\put(99.5,33){\line(0,1){11}}
\put(98,44){\line(1,0){3}}
\multiput(98,44)(3,0){2}{\circle*{0.9}}

\put(108.5,26){\line(0,1){9}}
\put(107,35){\line(1,0){3}}
\put(110,35){\line(0,1){5}}

\put(107,40){\line(1,0){4.5}}
\put(107,40){\circle*{0.9}}
\put(111.5,40){\line(0,1){9}}
\multiput(110,49)(3,0){2}{\circle*{0.9}}
\put(110,49){\line(1,0){3}}

\put(117.5,21){\line(0,1){12}}
\multiput(116,33)(3,0){2}{\circle*{0.9}}
\put(116,33){\line(1,0){3}}

\put(104,35){\circle*{0.9}}
\put(107,35){\line(-1,0){3}}

\put(91,-3){root}
 \end{picture}
\caption{Equivalent representations of a realization of $\CTCS(20)$.
One distinguished leaf is marked.
}
\label{Fig:1c}
\end{figure}

\begin{figure}
\hspace*{-1.9in}
\includegraphics[width=8.7in]{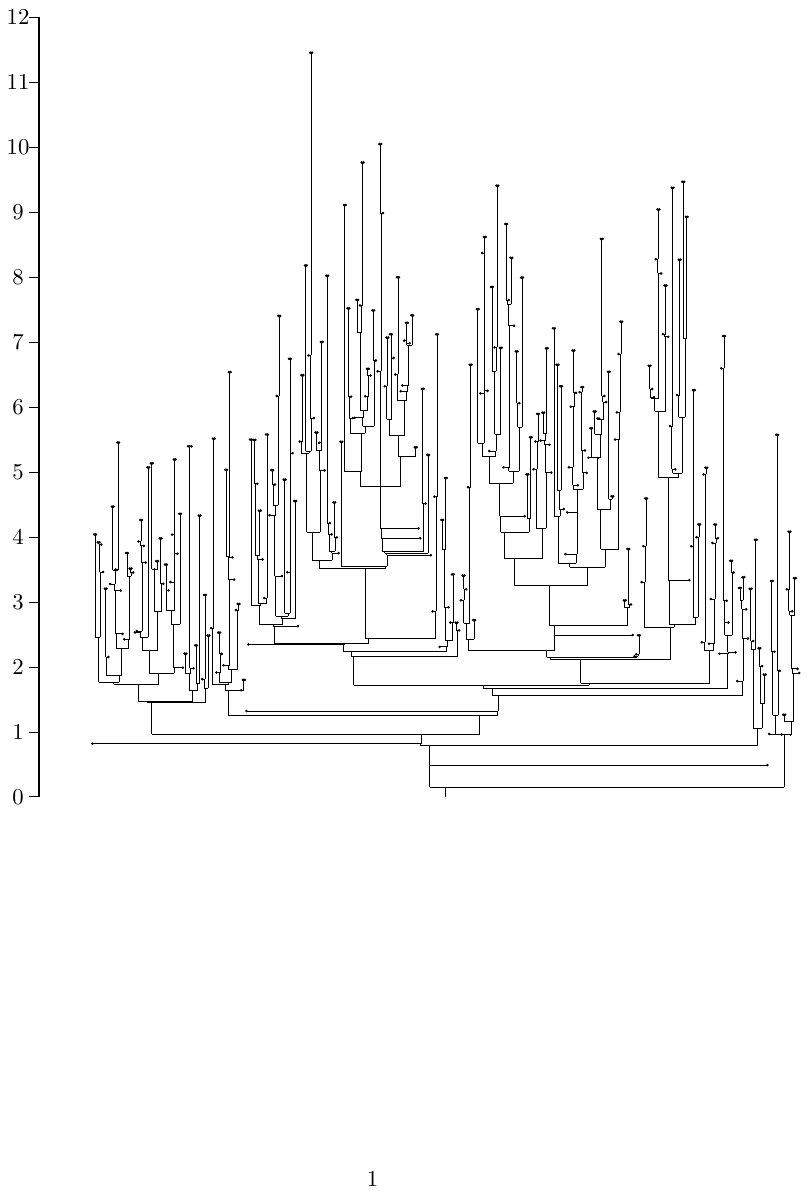}
\vspace*{-4.1in}
\caption{A realization of the tree-representation of the $\CTCS(n)$ model with $n = 400$. 
Drawn as in the previous Figure, so the width of subtrees above a given time level are the sizes of clades at that time.
}
\label{Fig:400}
\end{figure}
  
 Observe that there is no direct connection between the model (discrete or continuous) for $n$ and the model for $n+1$.
 Because a binary tree on $n$ leaves has $n-1$ splits, one imagines that as $n$ increases the trees will tend to get taller.
 However in the continuous model there is an offsetting feature, that the initial splitting rate $h_{n-1}$ is increasing with $n$.
 This turns out to have the following remarkable effect.
\begin{Proposition}
\label{P:split}
Let $B_n$ denote the height of the branchpoint between the paths to two uniform random distinct leaves of $\CTCS(n)$.
Then, for each $n \geq 2$, $B_n$ has exactly Exponential$(1)$ distribution.
\end{Proposition}
 The short stochastic calculus proof will be given in Appendix \ref{sec:Exp1}.
 This result hints at the ``consistency" result (Theorem \ref{T:consistent})
 and suggests that $h_{n-1}$ is the canonical choice of splitting rates for the continuization.
 In general the continuous model is more tractable, so we focus more on it rather than the discrete model.

We do find it convenient to adopt the biological term  {\em clade} for the set of leaves in a subtree, 
that is the elements in a subinterval somewhere in the interval-splitting process.
So the path from the root to the distinguished leaf 11 in Figure \ref{Fig:1c} passes through successive clades

$[[1,20]], \ [[4,20]], \ [[5,20]], \ [[9,20]], \ [[9,19]], \ [[9,17]], [[9,13]], \ [[9,11]], \ [[11]]$

\noindent
which have successive sizes (number of leaves) 
$20,17,16,12,11,9,5,3,1$.

Regarding terminology, remember that ``time" and ``height" are the same, 
within the construction of $\CTCS(n)$ for fixed $n$.
The height\footnote{Or depth, if one draws trees upside-down.}  of a leaf is the time at which its clade becomes a singleton, 
and the height of a split between clades is the time at which the split occurs.
Within the mathematical analysis of random processes we generally follow the  usual ``time" convention,
while in stating results we generally use the tree-related terminology of ``height".

\subsection{The unordered version}\label{sec:unordered}
In the definition above (in both discrete and continuous time), 
the leaves are labelled $1,\dots,n$ from left to right.
This is sometimes convenient, and it is important in
Section \ref{sec:depends}. 

Sometimes it is more convenient  to consider the \emph{unordered} version,
which is obtained by relabelling the leaves by a uniform random permutation.
Equivalently, the unordered version is obtained if we modify the
construction so that each time a clade of size $m$ splits into two of sizes
$i$ and $m-i$, we let the left subclade consist of $i$ leaves that are
chosen uniformly at random from the leaves in the clade.

In most cases it does not matter which version we use, for example when
considering properties of a random leaf.
The unordered version has the feature of (by definition) being exchangeable,
i.e., its distribution is invariant under permutations of the leaves.
This will be important in 
Section \ref{sec:consistent-exch},
where we therefore will use this version.

 \subsection{The three foundational results}
 We regard the following three results as ``foundational", in that they open the way to further developments.
 \begin{itemize}
 \item In Section \ref{sec:height} we describe the CLT for leaf heights, Theorem \ref{TNorm}, 
leading to results and conjectures for other height-related statistics.
 \item In Section \ref{sec:OPfringe} we describe the ``occupation measure" Theorem \ref{T:alimit}, leading to an explicit description 
 of the asymptotic {\em fringe tree},  many of whose properties have yet to be investigated.
 The fringe tree is essentially the way that real-world phylogenies are drawn as {\em cladograms}, and we illustrate a real example alongside a realization
of our model. 
 
 \item In Section \ref{sec:consistent-exch} we describe the consistency property (Theorem \ref{T:consistent}) and the resulting representation of a 
 limit tree  $\CTCS(\infty)$ via an exchangeable random partition of $\bbN$.
 This leads to a description of the ``number of subclades along a path to a uniform random leaf on the infinite boundary" process within $\CTCS(\infty)$ 
 in terms of a certain subordinator (Theorem \ref{T:exact}), and the possibility for further analysis of $\CTCS(\infty)$ itself.
 \item Section \ref{sec:misc} describes some less-studied aspects of the model, and provides more conceptual discussion.
 \item The Appendix contains alternative proofs of results proved in other papers.
 \end{itemize}

 \section{Heights and related statistics}
\label{sec:height}
To start our study of quantitative aspects of $\CTCS(n)$, let us consider heights of leaves. 
What can we say about the  height $D_n$ of a uniform random leaf $\ell$?
Figure \ref{Fig:400} suggests that $D_n$ increases slowly with $n$.

\subsection{The harmonic descent chain}
\label{sec:HD}
We can characterize $D_n$ in an alternate way, as follows.  In the discrete construction, 
the sequence of clade sizes along the path from the root to $\ell$ is the discrete-time Markov chain, starting in state $n$,  whose transition 
($m \to i)$ probabilities $q^*(m,i)$ 
are obtained by size-biasing the $q(m,\cdot)$ distribution; so
\begin{equation}
 q^*(m,i) :=   \sfrac{2i}{m} q(m,i) = \sfrac{1}{h_{m-1}} \cdot \sfrac{1}{m-i}, 
 \qquad 1 \le i \le m-1 , \ m \ge 2 
 \label{qstar}
 \end{equation}
from \eqref{01}.
Because the continuous-time CTCS process exits $m$ at rate $h_{m-1}$, the continuous-time process of clade sizes 
as one moves at speed $1$ along the path is the 
continuous-time Markov process on states $\{1,2,3,\ldots\}$ with transition rates
\begin{equation}
\lambda_{m,i} := \sfrac{1}{m-i}, \qquad 1 \le i \le m-1, \ m \ge 2
\label{lambda-rates}
\end{equation}
with state $1$ absorbing. 
So $D_n$ is the absorption time for this chain, started at state $n$.
Let us call this the 
(continuous-time)
{\em harmonic descent} (HD) chain.\footnote{{\em Descent} is a reminder that the chain is decreasing.  Despite its simple form, the HD chain has apparently never been studied until recently.} 
In parallel there is the discrete-time HD chain with transition probabilities \eqref{qstar}, and we write $L_n$ for the absorption time in discrete time.
So $L_n$ is the height, in the sense of number of edges, of a uniform random leaf in $\DTCS(n)$.

The next few sections study $D_n$ and $L_n$ as defined by the Markov chains, without using any extra properties of the tree model.

\subsection{The subordinator approximation heuristic}
\label{sec:subo}


The HD chain is directly relevant to our tree model in two ways (the second way involves the {\em fringe distribution}, Section \ref{sec:fringe}).
First,
there is a simple probabilistic heuristic for the behavior of the continuous time HD chain, 
leading to the approximation \eqref{approx} below.
Write $\bX = (X_t, t \ge 0)$ for the HD chain with rates \eqref{lambda-rates}, 
or $\bX^{(n)} = (X^{(n)}_t, t \ge 0)$ for this chain starting with $X^{(n)}_0 = n$.
The key idea is to study the process  $\log \bX = (\log X_t, \ t \ge 0)$.
By considering its transitions, 
one quickly sees that, for large $n$,  there should be a good approximation (the calculation is given in Section \ref{sec:approxc} below)
\begin{equation}
 \log X^{(n)}_t \approx \log n - Y_t    \mbox{ while } Y_t < \log n 
 \label{approx}
 \end{equation}
 where  $(Y_t, 0 \le t < \infty)$ is the subordinator with  {\em L\'{e}vy measure} 
  $\psi_\infty$ and corresponding $\sigma$-finite density $f_\infty$ on $(0,\infty)$ defined as
\begin{equation}
 \psi_\infty[a, \infty) :=  - \log (1 - e^{-a}); 
\quad  f_\infty(a) :=  \sfrac{e^{-a}}{1 - e^{-a}}, \quad
  \ 0 < a < \infty .
  \label{muinf}
  \end{equation} 
  Recall that a {\em subordinator} \cite{bertoin}
 is the continuous-time analog of the discrete-time process of partial sums of i.i.d.\ positive summands: informally
 \begin{align}
\Pr(Y_{t+dt}  - Y_{t} \in \ddx a) = f_\infty(a) \dd a\dd t . 
\end{align}
We call \eqref{approx}  the  {\em subordinator approximation heuristic}.
It often enables quick ``back of an envelope" calculations which can then be
formalized in different ways.
One formalization of this approximation is the limit theorem 
Theorem \ref{T:exact}.

It is well known that the subordinator $(Y_t, 0 \le t < \infty)$ satisfies the strong law of large numbers
\begin{equation}
 t^{-1} Y_t \to \rho \qquad \mbox{ a.s.\ as }  t \to \infty
 \label{LLN}
 \end{equation}
where the limit is the mean
\begin{eqnarray} 
\rho & =& \int_0^\infty \psi_\infty[a, \infty) \dd a 
=  \int_0^\infty -  \log (1 - e^{-a}) \dd a \nonumber \\
&=& \int_0^\infty \sum_{i=1}^\infty \frac{e^{-ia}}{i} \dd a
=\sum_{i=1}^\infty \frac{1}{i^2} = \zeta(2) =  \pi^2/6   \label{psirho}
\end{eqnarray}
by a classical calculation\footnote{Euler's formula $\zeta(2) := \sum_{i \ge
    1}  i^{-2} = \pi^2/6$ is used frequently in proofs.}.
Now $D_n$ is the time at which $\log X^{(n)}_t$ reaches $0$, 
so the approximation heuristic and \eqref{LLN} suggest the law of large numbers\footnote{Being pedantic, we do not yet have a joint distribution for $(D_n, n \ge 1)$ so we cannot write a.s. convergence.}
\begin{equation}
\label{WLLN}
 D_n/\log n \to_p \mu:= 1/\rho = 6/\pi^2 \mbox{ as } n \to \infty . 
 \end{equation}
Moreover the subordinator satisfies a central limit theorem,
because the central limit theorem for 
sums of i.i.d.\ variables extends immediately to subordinators by considering integer times.
The variance of the subordinator is $\var(Y_t) = \xsigma^2 t$ where $\xsigma^2$ is calculated as above:
\begin{eqnarray} 
\xsigma^2 &=& 2 \int_0^\infty a \psi_\infty[a, \infty) \dd a  =  2 \int_0^\infty - a \log (1 - e^{-a}) \dd a \nonumber \\
&=& 2\int_0^\infty \sum_{i=1}^\infty a\frac{e^{-ia}}{i} \dd a
=  2 \sum_{i=1}^\infty \frac{1}{i^3} = 2 \zeta(3)  
\end{eqnarray}
by another classical calculation.
So the CLT for the subordinator is
\[ \frac{Y_t - \rho t}{t^{1/2}} \to_d  \mathrm{Normal}(0,\xsigma^2) .\]
As with the sums of i.i.d.\ variables\footnote{Commonly seen as a textbook exercise, e.g.\ \cite[10.6.3]{grimmett}  or \cite[3.4.7]{durrett}.},
this extends to the ``renewal CLT" for 
$Q_s :=  \inf\{t: Y_t \ge s\}$:
\[ \frac{Q_s - \mu s}{s^{1/2}} \to_d  \mathrm{Normal}(0,\mu^3 \xsigma^2) .
\]
Recalling again that $D_n$ is the time at which $\log X^{(n)}_t$ reaches $0$, 
so the approximation heuristic and the renewal CLT above with $s = \log n$ 
suggest the following (true) Theorem, whose proof methods will be discussed in Section \ref{sec:proofCLT}.
\begin{Theorem}
\label{TNorm}
\[ \frac{\Ex[D_n]}{\log n} \to \mu \quad\mbox{and\/}\quad
\frac{D_n - \mu \log n}{\sqrt{\log n}} \to_d  \mathrm{Normal}(0,\sigma^2) 
\quad \mbox{ as } n \to \infty 
\]
where
\[ \mu := 1/\zeta(2) = 6/\pi^2 = 0.6079... ; \quad 
\sigma^2 := 2 \zeta(3)/\zeta^3 (2) = 0.5401... .\]
\end{Theorem}

\noindent
Figure \ref{Fig:6} shows the Normal distribution emerging.
\begin{figure}[ht]
\begin{center}
\includegraphics[width = 3.0in]{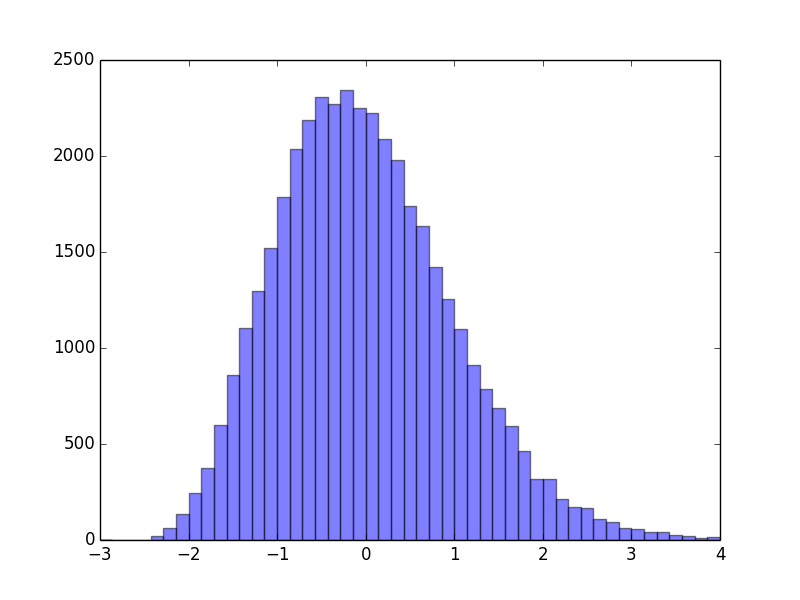}
\caption{Histogram of leaf heights $D_n$ , relative to mean and s.d.; multiple simulations with n = 3,200.
}
\label{Fig:6}
\end{center}
\end{figure}

\subsection{The approximation calculation}
\label{sec:approxc}


Here is the calculation for the approximation \eqref{approx}.
The process  $\log \bX $ is itself Markov with transition rates described below.
A jump\footnote{Note these are {\em downward}  jumps, so take negative values.} of $\bX$ from $j$ to $j - i$ has height $-i$, which corresponds to a
jump of $\log \bX$ from $\log j$ having height  $\log (j-i)  - \log j =  \log (1 - i/j)$.
Define the measure $\chpsi_j$ on $(-\infty, 0)$ as the measure assigning weight $1/i$ to point $\log (1 - i/j)$, for each $1 \le i \le j-1$.
So this measure $\chpsi_j$ specifies the  heights and rates of the downward jumps of $\log \bX$ from $\log j$. 
Writing
\begin{align}
 \chpsi_j(-\infty, a] = \sum_{i = j (1-e^{a})}^{j-1} 1/i  
\end{align}
shows that there is a $j \to \infty$ limit measure in the sense
\begin{equation}
  \chpsi_j(-\infty, a]  \to  \chpsi_\infty(-\infty, a] \mbox{ as } j \to \infty, \  -\infty < a < 0  
  \label{mumu}
  \end{equation}
where the limit $\sigma$-finite measure $\chpsi_\infty$ on $(-\infty,0)$  is the ``reflected" version of the measure $\psi_\infty$ on $(0,\infty)$ at (\ref{muinf}):
\begin{equation}
 \chpsi_\infty(-\infty, a] := - \log (1 - e^a), \quad
\chf_\infty(a) := \sfrac{e^a}{1 - e^a}, \ -\infty < a < 0  .
\label{phiphi}
\end{equation}
So this is the origin of the subordinator approximation heuristic. 
But a more striking fact is that, as well as providing an approximation within  $\CTCS(n)$, we shall show in Section \ref{sec:sub} that the subordinator arises {\em exactly} within a 
limit structure  $\CTCS(\infty)$.

\subsection{The discrete time setting}
\label{sec:DTsetting}
The arguments above suggest that the continuous-time model is somewhat more tractable than the discrete-time model -- in a sense, in formulating the
continuous-time model we have already done the scaling that leads to the subordinator approximation for $\log \bX$.
However there are parallel results in discrete time.
Here is the  analog of Theorem \ref{TNorm} 
(proof methods discussed below).

\begin{Theorem}
\label{Lnorm}
\[ \frac{\Ex[L_n]}{\log^2 n} \to \frac{1}{2\zeta(2)} \mbox{ and } 
 \frac{L_n- \frac{1}{2\zeta(2)}  \log^2 n}{\log^{3/2} n} \to_d 
 \mathrm{Normal} \left(0,\tfrac{2\zeta(3)}{3\zeta^3(2)}\right) .
\]
\end{Theorem}

\subsubsection{Heuristics for $\log^2 n$}
\label{sec:log2}


If one starts discussing the beta-splitting family
by saying that the mean leaf height in discrete time is order $\log^2 n$,
that order seems surprising.
Here is a heuristic explanation.

Start with the subordinator approximation heuristic 
\[ \log X^{(n)}_t \approx \log n - \frac{\pi^2}{6} t \]
which explains the continuous-time mean $\frac{6}{\pi^2} \log n$.
Then observe that 
the mean number of steps of the discrete chain associated with a continuous-model time increment $\delta t$ 
is $h(X^{(n)}_t) \cdot  \delta t \approx (\log n - \frac{\pi^2}{6} t) \delta t .$
So the approximate discrete height is the integral
\[ \int_0^{6/\pi^2 \cdot \log n} \Bigpar{\log n - \frac{\pi^2}{6} t} \ dt 
= \frac{3}{\pi^2} \log^2 n .\]

See also Section \ref{sec:OP}.

\subsection{Proof methods for the CLTs for leaf height}
\label{sec:proofCLT}


We currently know 5 proofs of the asymptotic normality in
Theorems \ref{TNorm} and/or  \ref{Lnorm}, described below in chronological order of discovery.

{\em Proof 1.} The first proof we found, given for the record in Section \ref{sec:Proof1}, is a direct attempt to justify the approximation
\eqref{approx}  so that one can apply a martingale CLT to prove Theorem \ref{TNorm}.
  This is in principle straightforward but seems quite tedious and lengthy in detail.

{\em Proof 2.} The second proof we found \cite[Theorem 1.7]{beta1} is via an analysis of recursions for the Laplace transforms of $D_n$.
The full proof of both theorems (and many other results indicated later) by this methodology appears in the article 
\cite{beta1}.  To illustrate that methodology, in Section \ref{sec:recurse} we show (following the first steps of the proof 
of \cite[Proposition 2.1]{beta1}) 
how to prove
\begin{equation}
\Ex[D_n]= \sfrac{6}{\pi^2}\log n+O(\log\log n) \mbox{ as } n \to \infty .
\label{loglog}
\end{equation}

{\em Proof 3.}  
There is a general {\em contraction method}  \cite{NR04} which has been used to prove convergence in distribution for other recursively-defined  structures.
Kolesnik uses that method to prove  Theorems \ref{TNorm} and \ref{Lnorm}:
see  \cite[Theorem 3.1 and Theorem 4.1]{kolesnik}.
That method also gives rates of convergence in the Zolotarev metric $\zeta_3$.

{\em Proof 4.}
Iksanov \cite{iksanovCLT} shows that 
Theorems \ref{TNorm} and \ref{Lnorm}, including a joint normal limit,
can be derived from known results
\cite{GI12} in the theory of {\em regenerative composition structures}.  To quote \cite{iksanovCLT}, they 
``exploit a connection with an infinite ``balls-in-boxes’' scheme, a.k.a.\ Karlin’s occupancy scheme in random environment".
This is an exact relationship for finite $n$, unlike in previous proofs,
thus allowing a shorter derivation of 
Theorems \ref{TNorm} and \ref{Lnorm} 
from known results. 
 See Section \ref{sec:sub} for a brief discussion.
 
 {\em Proof 5.} 
As a corollary of  Theorem 1.4 of \cite{beta4-arxiv}, quoted below as
Theorem \ref{Tmgf1}, which gives sharp estimates of $\Ex[e^{zD_n}]$.

Why are we mentioning 5 proofs? 
As discussed in Sections \ref{sec:depends} and \ref{sec:leaf-ht} below,
a more refined analysis of correlations between leaf-heights is needed for analysis of the tree height.
It is not clear which of these techniques might be most applicable for tackling such possible extensions.

\subsection{Summary of sharper results from \cite{beta1} and \cite{beta4-arxiv}}
\label{sec:summary}
A variety of sharper and additional results in the spirit of  Theorems \ref{TNorm} and  \ref{Lnorm} 
have been established, first in \cite{beta1}  by 
``analysis of recursions" 
and second in \cite{beta4-arxiv} by Mellin transforms. 
The proofs are technically intricate, and here we will merely list the results.



Here are results from  \cite{beta4-arxiv} proved by Mellin transforms. 
Let $\psi$ be the digamma function $\psi(z) := \Gamma^\prime(z)/\Gamma(z)$.
Let $0>s_1>s_2>\dots$ be the negative roots of $\psi(s)=\psi(1)$.
Recall that $\zeta(2) = \pi^2/6$ and $\zeta(3) \doteq 1.202$,
and note that $\sim$ in the results below denotes asymptotic expansion%
\footnote{In the sense that  the error when approximating
with a partial sum is of the order of the largest (non-zero) omitted term.}.

\begin{Theorem}[\cite{beta4-arxiv} Theorem 1.1]
\label{TD1}
As $n \to \infty$
  \begin{align}\label{aw21}
  \Ex[D_n] \sim 
\frac{6}{\pi^2}\log n
+\sum_{i=0}^\infty c_{i} n^{-i}
+\sum_{j=1}^\infty\sum_{k=1}^\infty c_{j,k}\,n^{-|s_j|-k}
\end{align}
for some coefficients $c_i$ and $c_{j,k}$ that can be found explicitly;
in particular, $c_1=-3/\pi^2$ and 
\begin{equation}
c_0 =  {\frac{\zeta(3)}{\zeta^2(2)}+\frac{\gamma}{\zeta(2)}} \doteq  0.795155660439.
\label{c01}
\end{equation}
\end{Theorem}
This improves on \SJ\cite[Theorem 1.1 and Proposition 2.3]{beta1} which gave
the initial terms 
$\frac{6}{\pi^2}\log n + c_0 + c_1n^{-1}$ 
with the explicit formula for $c_1$ but not the formula\footnote{Before knowing the exact value of $c_0$, numerics gave an estimate that agrees with \eqref{c01} to 10 places.} 
for $c_0$.
The discussion of the  \emph{h-ansatz} in \cite{beta1} assumes that only integer  powers of $1/n$ should appear in the expansion \eqref{aw21}, 
but in fact (surprisingly?)\ the spectrum of powers of $n$ appearing is $\set{-i:i\ge0}\cup\set{-(|s_j|+k):j\ge1,k\ge1}$.

\begin{Theorem}[\cite{beta4-arxiv} Theorem 1.2]
\label{TL1}
As $n \to \infty$
\begin{multline}\label{sw2x1}
  \Ex[L_n]\sim
\frac{3}{\pi^2}\log^2n 
+ \Bigpar{\frac{\zeta(3)}{\zeta^2(2)}+\frac{\gamma}{\zeta(2)}}\log n
+b_0
\\
+\sumk a_k n^{-k}\log n
+\sumk b_k n^{-k}
+ \sumj \sumk c_{j,k}n^{-|s_j|-k}
\end{multline}
for some computable constants $a_k$, $b_k$, $c_{j,k}$;
in particular,
\begin{align}\label{hw211}
  b_0 = 
  \frac{3\gam^2}{\pi^2}
+\frac{\zeta(3)}{\zeta^2(2)}\gamma+
\frac{\zeta^2(3)}{\zeta^3(2)}
+\frac{1}{10}
\doteq
0.78234
.\end{align}
\end{Theorem}
The first term $\frac{3}{\pi^2}\log^2n$ was observed long ago in \cite{me_clad}.
Using the recurrence method, the coefficient for $\log n$ 
was found in  \cite[Theorem 1.2]{beta1}; that coefficient 
equals the constant term $c_0$ in the asymptotic expansion \eqref{aw21} of
$\Ex[D_n]$.

\begin{Theorem}[\cite{beta4-arxiv} Theorem 1.3]
\label{TvarD1}
As $n \to \infty$
  \begin{align}\label{kb91}
\var[D_n]
=
\frac{2\zeta(3)}{\zeta^3(2)}\log n
+ \frac{2\zeta(3)}{\zeta^3(2)}\gamma
+\frac{5\zeta^2(3)}{\zeta^4(2)}
-\frac{18}{5\pi^2}
+ O\Bigpar{\frac{\log n}{n}}
.\end{align}
\end{Theorem}
The leading term $\frac{2\zeta(3)}{\zeta^3(2)}\log n$ was found in \cite[Theorem 1.1]{beta1} by the recursion method.
Higher moments of $D_n$ are also discussed in  \cite{beta4-arxiv}.

\begin{Theorem}[\cite{beta4-arxiv} Theorem 1.4]
\label{Tmgf1}
For $-\infty < z < 1$ there is a 
unique real number $\rho(z)$ in $(-1,\infty)$ satisfying
 $ \psi\bigpar{1+\rho(z)}-\psi(1)=z$.
Then
\begin{align}\label{qz31}
\Ex[e^{zD_n}] &
=
\frac{-z\gG(-\rho(z))}{\psi'(1+\rho(z))}\frac{\gG(n)}{\gG(n-\rho(z))}
+O\bigpar{n^{-\gsx}}
\end{align}
and
\begin{align}\label{qz331}
\Ex[e^{zD_n}] &
=
\frac{-z\gG(-\rho(z))}{\psi'(1+\rho(z))}n^{\rho(z)}
\cdot\bigpar{1+O\bigpar{n^{-\min(1,\gsx+\rho(z))}}}
\end{align}
where 
$\gsx   =1-s_1(1+\psi(1))=1+|s_1(\psi(2))|  \doteq 1.457$.

Furthermore, \eqref{qz31} holds uniformly for $z<1-\gd$ for any $\gd>0$,
and \eqref{qz331} holds uniformly for $z$ in a compact subset of $(-\infty,1)$.
\end{Theorem}
This improves on bounds in  \cite[Section 2.7]{beta1}.

As a corollary of Theorem \ref{Tmgf1}, a new proof of the CLT for $D_n$, our Theorem \ref{TNorm}, is given in \cite{beta4-arxiv}.
As mentioned before, several previous proofs have been given.

Another corollary of Theorem \ref{Tmgf1} is the following large deviation result.
\begin{Theorem}[\cite{beta4-arxiv} Theorem 1.6]
\label{Tdev1}
As \ntoo, we have:
  \begin{align}\label{tdev11}
  \Pr(D_n<x\log n)& = n^{-\gLX(x)+o(1)}, \qquad\text{if}\quad 0<x\le x_0,
\\ \label{tdev21}
\Pr(D_n>x\log n)& =n^{-\gLX(x)+o(1)}, \qquad\text{if}\quad x_0\le x< x_1,
\\   \label{tdev31} 
  \Pr(D_n>x\log n)& \le n^{-\gLX(x)+o(1)}, \qquad\text{if}\quad x\ge x_1,
  \end{align}
  where $x_0 = 1/\zeta(2), \ x_1 = 1/(\zeta(2)-1)$ and $\Lambda^*$ (defined
  at \cite[(12.29)]{beta4-arxiv})
is such that
$\Lambda^*(x) > 0$ for $x \neq x_0$, 
  and $\Lambda^*(x) = x-1$ for $x \ge x_1$.
\end{Theorem}

Theorem \ref{Tdev1} improves estimates for the upper tail in \cite[Theorem 1.4]{beta1}. 

\paragraph{The recurrence method.}
The theme of \cite{beta1} was to exploit ``the recurrence method", that is to take a sequence defined by a recurrence
and then upper and lower bound the unknown sequence by known sequences.
Note that there is indeed a simple recurrence \eqref{tn2} for $ \Ex[D_n]$, 
and we give an illustration of the use of this methodology in Section \ref{sec:recurse}.
This method was used in \cite{beta1} for many of the problems in this paper, as indicated in the references above, though (where applicable) the Mellin transform method seems to yield sharper results.
However we were unable to use the methodology in \cite{beta4-arxiv} for analysis of $L_n$, the discrete absorption time, beyond Theorem \ref{TL1}. 
So we now quote parallel results for $L_n$, taken from  \cite{beta1}
and \cite{kolesnik}.
One is the CLT (our Theorem \ref{Lnorm}), which is  \cite[Theorem 1.7]{beta1}.
Also
\begin{Theorem}[{\cite{kolesnik} Lemma 3.3}] 
\label{T:varL}
\[
\var(L_n)=\frac{2\zeta(3)}{3\zeta^3(2)}\log^3 n
+\Bigpar{\frac{4\zeta(3)^2}{\zeta(2)^4}
-\frac{3\zeta(4)}{\zeta(2)^3}
+\frac{2\gamma\zeta(3)}{\zeta(2)^2}
-\frac{1}{2\zeta(2)}
}\log^2 n
+O(\log n). 
\]
\end{Theorem}
This result from \cite{kolesnik} improves \cite[Theorem 1.2]{beta1} 
where $\var(L_n)$ was found up to an error $O(\log^2 n)$;
\cite{kolesnik} uses also the recurrence method. 

Finally,  a large deviation estimate.
\begin{Theorem}[\cite{beta1} Proposition 2.12]
\label{T:10}
 For $\eps>0$,
\[
\Pr \bigl(L_n\ge\tfrac{3}{\pi^2}(1+\eps)\log^2n\bigr)=O\bigl(n^{-\Theta(\eps)}\bigr).
\]
\end{Theorem}

\subsection{An illustration of analysis of recursions}
\label{sec:recurse}


Here we  illustrate the ``analysis of recursions" methodology used in \cite{beta1} for many results similar to those above.
We will copy the first steps of the proof in \cite{beta1} of a slightly weaker form of Theorem \ref{TD1};
these first steps are enough to reach the weaker result stated  in
Proposition \ref{thm1} below.
The proof uses only the elementary recurrence for $t_n := \Ex D_n$:
\begin{equation}
t_n =    \sfrac{1}{h_{n-1}} (1 + \sum_{i=1}^{n-1} \sfrac{t_i}{n-i} )
\label{tn2}
\end{equation}
with $t_1 = 0$.
One can see the first order result $\Ex[D_n] \sim  \sfrac{6}{\pi^2}\log n$ heuristically by plugging $c \log n$ into the recursion and taking the natural first-order
approximation to the right side; the constant $c$ would emerge as the inverse of the constant
\begin{equation}
\int_0^1\sfrac{\log(1/x)}{1-x}\,dx=\zeta(2)=\sfrac{\pi^2}{6}
\label{zeta2}
\end{equation}
and indeed this is how it emerges in the proof below.
\begin{Proposition}\label{thm1}
\[
\Ex[D_n]= \sfrac{6}{\pi^2}\log n+O(\log\log n) \mbox{ as } n \to \infty .
\]
\end{Proposition}
\smallskip
\noindent
\begin{proof} 
The proof involves three steps.

\noindent
{\em Step 1.} We shall prove
\begin{equation}
\Ex[D_n]\ge \sfrac{6}{\pi^2}\log n, \quad n \ge 1 .
\label{step1}
\end{equation}
Setting $\tau_n=A\log n$ for $A = 6/\pi^2$, it suffices to show 
\begin{equation}
\label{new2}
  \frac{1}{h_{n-1}}\biggl(1+\sum_{k=1}^{n-1}\sfrac{\tau_k}{n-k}\biggr) \ge \tau_n, \quad n\ge 2,
\end{equation}
because then, by (\ref{tn2}) and induction on $n$, $\Ex[D_n]\ge \tau_n$ for all $n\ge 1$, establishing (\ref{step1}).
We compute
\begin{eqnarray*}
 \frac{1}{h_{n-1}}\biggl(1+\sum_{k=1}^{n-1}\tfrac{\tau_k}{n-k}\biggr) &=& \frac{1}{h_{n-1}}\biggl(1+\sum_{k=1}^{n-1}\tfrac{A\log k}{n-k}\biggr)\\
 &=& \frac{1}{h_{n-1}}\biggl(1+A(\log n) h_{n-1}+A\sum_{k=1}^{n-1}\tfrac{\log(k/n)}{n(1-k/n)}\biggr)\\
&=& \tau_n+\tfrac{1}{h_{n-1}}\biggl(1+A\sum_{k=1}^{n-1}\tfrac{\log(k/n)}{n(1-k/n)}\biggr)\\
&\ge&  \tau_n+\tfrac{1}{h_{n-1}}\biggl(1-A\int_0^1\tfrac{\log(1/x)}{1-x}\,dx\biggr).
\end{eqnarray*}
The inequality holds because the integrand is positive and decreasing. 
So by (\ref{zeta2}), the choice $A = 6/\pi^2$ establishes (\ref{new2}).

\medskip
\noindent
{\em Step 2.} 
 Let us prove 
 \begin{equation}
  \mathbb E[D_n]\le f(n)
  \label{step2}
  \end{equation}
where
 \begin{equation}
f(x):= 
\begin{cases}
  x-1, &  x\le 2,
\\
1+\log(x-1), & x\ge2.
\end{cases}
  \label{step2b}
  \end{equation}
  This is true for $n=1$ since $\mathbb E[D_1]=0$.
So, similarly to \eqref{new2}, it is enough to show that  $f(n)$ satisfies  
\begin{equation}\label{new2.01}
f(n)\ge \frac{1}{h_{n-1}}\biggl(1+\sum_{i=1}^{n-1}\tfrac{f(i)}{n-i}\biggr),\quad n\ge 2.
\end{equation}
Since $f(x)$ is {\it concave\/}, we have
\begin{multline*}
 \frac{1}{h_{n-1}}\biggl(\,1+\sum_{i=1}^{n-1}\tfrac{f(i)}{n-i}\biggr)\le \tfrac{1}{h_{n-1}}+f\biggl(\tfrac{1}{h_{n-1}}\sum_{i=1}^{n-1}\tfrac{i}{n-i}\biggr)\\
 =\tfrac{1}{h_{n-1}}+f\bigl(n-\tfrac{n-1}{h_{n-1}}\bigr)\le \tfrac{1}{h_{n-1}}+f(n)-f'(n)\bigl(\tfrac{n-1}{h_{n-1}}\bigr),\\
 \end{multline*}
  which is exactly $f(n)$, since $f'(x)=\tfrac{1}{x-1}$ for $x\ge2$.

\medskip
\noindent
{\em Step 3.}
Let $n_0\ge 2$, and
\[
A=A(n_0):=\biggl(\int_{1/n_0}^1 \sfrac{\log(1/x)}{1-x}\,dx\biggr)^{-1},\quad B=B(n_0):=n_0^{\sfrac{2}{A\log 2}-1}.
\]
We shall prove 
\begin{equation}
\Ex [D_n]\le A \log(nB), \ n\ge 2. 
\label{step3}
\end{equation}
This inequality certainly holds for $n\le n_{\scriptscriptstyle 0}$, because, by (\ref{step2}), for those $n$
\[
\Ex [D_n]\le  \sfrac{2}{\log 2}\log n =A\log n\cdot\sfrac{2}{A\log 2} =A\log \Bigl(n^{\sfrac{2}{A\log 2}}\Bigr)\le A\log (nB).
\]
Therefore it suffices to show that $\tau_n:=A\log (nB)$ satisfies
\begin{align}
   \frac{1}{h_{n-1}}\biggl(1+\sum_{i=1}^{n-1}\sfrac{\tau_i}{n-i}\biggr) \le
  \tau_n, \quad n>n_0.
\label{jesper}
\end{align}
Plugging $\tau_i = \tau_n - A \log (n/i)$  into the  left side of 
\eqref{jesper}, we rewrite it as follows, cf.\ Step 1:
\begin{eqnarray*}
   \frac{1}{h_{n-1}}\biggl(1+\sum_{i=1}^{n-1}\sfrac{\tau_i}{n-i}\biggr)       &=& 
   \frac{1}{h_{n-1}}\biggl(1+\tau_n\cdot h_{n-1}-A\sum_{i=1}^{n-1}\sfrac{\log(n/i)}{n(1-i/n)}\biggr) \\
&\le &\tau_n+\sfrac{1}{h_{n-1}}\biggl(1-A\int_{1/n}^1\sfrac{\log(1/x)}{1-x}\,dx\biggr)\\
&\le & \tau_n+\sfrac{1}{h_{n-1}}\biggl(1-A\int_{1/n_0}^1\sfrac{\log(1/x)}{1-x}\,dx\biggr)\\
&=& 
\tau_n.
\end{eqnarray*}
This establishes \eqref{jesper} and thus (\ref{step3}).

\medskip
\noindent
{\em Step 4.}
Note that $A(n_0) \le \frac{6}{\pi^2} + O(\sfrac{\log n_0}{n_0})$ and $\log B(n_0) = O(\log n_0)$.
So choosing $n_0=\lceil\log n\rceil$ we have 
\[ A(n_0) \log (nB(n_0)) 
\le 
\left( \sfrac{6}{\pi^2} + O( \sfrac{\log \log n}{\log n}) \right) \ \ 
\left( \log n + O(\log \log n) \right) .
\]
So (\ref{step3}) establishes the upper bound in Theorem \ref{thm1}, and (\ref{step1}) establishes the lower bound.
\qed
\end{proof}

\begin{Remark}
The simple idea is to replace a recurrence equality by a recurrence {\it inequality\/} for which an explicit solution can be found and then to use it to upper (or lower) bound the  solution of the recurrence {\it equality\/}. 
And one can use  probabilistic heuristics to guess the asymptotic behavior, and then check that a slightly larger (or smaller) function satisfies the recurrence inequality.
\end{Remark}

 \newcommand{\ent}{\mathrm{ent}}

\newcommand{\TT}{\mathcal T}
 \newcommand{\tre}{\mathbf{t}}
\subsection{Another illustration:  The $B_2$ (entropy) index}
\label{sec:balance}

In section \ref{sec:tree_balance} we will study, mostly heuristically, a variety of  {\em tree balance indices}.
Here we study one, as another illustration of the method of analysis of recursions.


On a realization $\tre$ of a binary tree, 
by moving away from the root and at each branchpoint 
taking the right or left branch with equal probability, 
one ends with a
(typically non-uniform) probability distribution $\nu_\tre$ on the  leaves $\ell$ of $\tre$.
Now consider the entropy of this distribution\footnote{Convenient here to take {\em logs} in base 2.}
\[ \ent(\nu_\tre) : = - \sum_\ell \nu_\tre(\ell) \log_2 \nu_\tre(\ell)\]
and abuse notation by writing $\ent(\tre) := \ent(\nu_\tre)$. 
This is called the $B_2$ index in the phylogenetics literature.  
A comprehensive account of this index $B_2$ appears in  \cite{franccois2021revisiting}.
Recall the intuitive interpretation of the entropy of a distribution $\nu$ as indicating that the distribution is 
``as random as" the uniform distribution on $2^{\ent(\nu)}$ elements.

What is the distribution of $\ent(\TT_n)$ for our random tree model $\TT_n$?
The expectation $B_2(n) := \Ex [\ent(\TT_n)]$ can in principle be studied via the recursion \eqref{Bnrec}  below, 
a method used for other aspects of $\TT_n$ in \cite{beta1}.
It is easy to check  \cite{franccois2021revisiting} the equivalent definition 
\[ \ent(\tre) = \sum_\ell d(\ell) 2^{- d(\ell)} \]
where $d(\ell)$ is the height of leaf $\ell$.
From that representation one obtains a very simple recursion for the 
expectation 
$B_2(n) := \Ex [\ent(\TT_n)]$ in our model
\begin{equation}
B_2(n) = 1 + \tfrac{1}{2} \sum_{i=1}^{n-1} q(n,i)(B_2(i) + B_2(n-i)) 
\label{Bnrec} 
\end{equation}
with $B_2(1) = 0, B_2(2) = 1$.


It is straightforward to calculate $B_2(n)$ numerically from the recursion.
It is convenient to record the ``equivalent uniform distribution size" $2^{B_2(n)}$, shown in the table below.
These numerical values suggest that $2^{B_2(n)}$ grows slightly faster than $\log n$ but slower than $\log^2 n$.

\begin{table}[h!]
 \begin{tabular}{c|ccccccccc}
 \hline
$n$ &16 & 32 & 64 & 128 & 256 & 512 & 1024  & 2048 & 4096\\
$2^{B_2(n)}$ & 7.99 & 10.75 & 13.74 & 16.92 & 20.26 & 23.72 & 27.32 & 31.03 & 34.85\\
\hline
\end{tabular}
\caption{Numerical values of $B_2(n)$.}
\label{table:1}
\end{table}

\subsubsection{Easy bounds on $B_2$ via recursion}
\label{sec:BB2}
In our context there is a simple sequence (\ref{bn-def}) that grows as order $\log \log n$.
Using the recurrence method we luckily get this sequence as a lower bound for $B_2(n)$, together with a cruder upper bound.
\begin{Proposition}
\label{P:Bn}
Let
\begin{equation}
 b(n) := \sum_{i=1}^{n-1} \sfrac{1}{ih_i} . 
 \label{bn-def}
 \end{equation}
Then
\begin{align}
  \label{b2b}
b(n) \le B_2(n) \le h_{n-1}  
\end{align}
for all $n \ge 1$.
 \end{Proposition}
We have $b(1) = 0$ and $b(2) = 1$, and $b(n) \sim \log \log n$.

\begin{Remark}
Boris Pittel (private communication) improves these bounds enough to prove $B_2(n) \sim 2 \log \log n$.
\end{Remark}


\noindent
 {\bf Proof of Proposition \ref{P:Bn}.}
 We prove first the lower bound.
 Because $B_2(n)$ satisfies the equality (\ref{Bnrec}), 
and $B_2(1)=b(1)$,
it suffices to show that $b(n)$ satisfies the corresponding inequality
\begin{align}
  \label{karin}
 b(n) \le 1 + \frac{1}{2} \sum_{m=1}^{n-1} q(n,m) (b(m) + b(n-m)) := b^*(n),
  \mbox{ say}.
\end{align}
We calculate, using $h_{n-i-1}\ge h_{n-1}-h_i$,
\begin{eqnarray*} 
b^*(n) - 1 &=&  \sum_{m=1}^{n-1} q(n,m) b(m) \quad\mbox{ by symmetry } \\
 &=&  \  \frac{1}{2h_{n-1}} \ \sum_{m=1}^{n-1}  (\frac{1}{m} + \frac{1}{n-m} ) \    \sum_{i=1}^{m-1} \frac{1}{ih_i}   \\
&=&  \  \frac{1}{2h_{n-1}} \ \sum_{i=1}^{n-2} \frac{1}{ih_i}  \sum_{m=i+1}^{n-1} (\frac{1}{m} + \frac{1}{n-m} ) \\
&=& \  \frac{1}{2h_{n-1}} \ \sum_{i=1}^{n-1} \frac{1}{ih_i}  (h_{n-1} - h_i
    +  h_{n-i-1} ) \\
&\ge& \  \frac{1}{h_{n-1}} \ \sum_{i=1}^{n-1} \frac{1}{ih_i}  (h_{n-1} - h_i) \\
 &=&   b(n) -   \frac{1}{h_{n-1}} \ \sum_{i=1}^{n-1}  \frac{ h_i}{ih_i} \\
&=&   b(n)  - 1,
 \end{eqnarray*}
which proves \eqref{karin}.

For the upper bound, we similarly consider
\begin{align}
\frac{1}{2} \sum_{m=1}^{n-1} q(n,m) (h_{m-1} + h_{n-m-1})
  &=  \sum_{m=1}^{n-1} q(n,m) h_{m-1}\notag\\
 &=  \  \frac{1}{2h_{n-1}} \ \sum_{m=1}^{n-1}  (\frac{1}{m} + \frac{1}{n-m} ) \    \sum_{i=1}^{m-1} \frac{1}{i}   \notag\\
&=  \  \frac{1}{2h_{n-1}} \ \sum_{i=1}^{n-2} \frac{1}{i}  \sum_{m=i+1}^{n-1} (\frac{1}{m} + \frac{1}{n-m} ) \notag\\
&= \  \frac{1}{2h_{n-1}} \ \sum_{i=1}^{n-1} \frac{1}{i}  (h_{n-1} - h_i
    +  h_{n-i-1} ) \notag\\
&=  \ \sum_{i=1}^{n-1} \frac{1}{i}
-
\frac{1}{2h_{n-1}} \ \sum_{i=1}^{n-1} \frac{1}{i}  (h_i+h_{n-1} - h_{n-1-1})
\notag\\&
\le h_{n-1}-1,
\end{align}
where the final inequality is checked numerically for $n\le 20$; 
for larger $n$ it follows because
\begin{align}
  \sum_{i=1}^{n-1} \frac{1}{i}  h_i \ge 2h_{n-1}
\end{align}
for $n=21$, and thus by induction for all $n\ge21$.
\qed

\subsection{Heuristics: correlation between leaf heights and the original interval-splitting model}
\label{sec:depends}

In studying the height of a uniform random leaf via the HD chain, we are of course dealing with two levels of randomness, the realization of the random tree and then the random choice of leaf.
To study the interaction between the two levels of randomness, it is natural to consider the correlation between leaf heights.
At a heuristic level (we leave the proof as a presumably easy open problem, stated below) this is straightforward, as follows.
Consider the interval-splitting representation of $\CTCS(n)$, but (for simplicity) imagine the interval as a cycle.
Consider the heights $D_n^{(1)}$ and $D_n^{(2)}$ 
of two leaves a distance $r_n$ apart in the cycle.
Each is distributed as $D_n$,
 the time height of the uniformly random leaf, but are not independent.
 We study the correlation coefficient defined by
\[
\rho(n,r_n) :=\tfrac{\Ex[D_n^{(1)} D_n^{(2)}]-\Ex^2[D_n]}{\text{Var}(D_n)}.
\]
From the Brownian motion formalization of the subordinator approximation (Proposition \ref{PB}),
one sees that the correlation should be asymptotically the same as the correlation between
(for independent Brownian motions)
$B_{s_0} + (B^{(1)}_1 -  B^{(1)}_{s_0})$ 
and
$B_{s_0} + (B^{(2)}_1 -  B^{(2)}_{s_0})$ 
where $s_0 := s_0(n,r_n)  := \frac{\log n - \log r_n}{\log n}$.
But that correlation is simply $s_0$ itself.
So it should be quite straightforward to prove
\begin{OP}
\label{OP:alpha}
Prove that, for  $\frac{\log r_n}{\log n} \to \alpha \in [0,1]$, we have
\begin{equation}
\rho(n,r_n) \to 1 - \alpha .\label{rhon}
\end{equation}
\end{OP}
For $\alpha = 1$ this is essentially \cite[Theorem 2.6]{beta1}.\footnote{The arXiv version 3 preprint of  \cite[Theorem 2.6]{beta1} is incorrect: a correct argument appears in the published version \cite{beta1}.}

Here is a related issue.
In the original interval-splitting description of the model (without randomizing leaf labels),
one could study (for example) the height of leaf $1$.
Indeed, consider the subordinator approximation calculation in section \ref{sec:approxc}, 
but work with the discrete time process $(Z_t, t = 1,2,\ldots) $ giving the
sequence of clade sizes along the route from the root, taking always the
left-side split.
This gives a Markov chain with transition probabilities $q(m,i)$, in
contrast to the size-biased probabilities \eqref{qstar} for a random leaf.
A quick back-of-an-envelope calculation shows that conditioned on the left
subclade being the smallest,
$\log Z_{t+1} \approx U\log Z_t$ where $U$ is Uniform$(0,1)$ and
independent of $Z_t$, while $\log Z_{t+1}\approx\log Z_t$ if the left
subclade is the largest. So taking logarithms again and approximating with a
renewal process suggests that the height of leaf $1$ in $\DTCS(n)$ should be
around $2 \log\log n$,
in contrast to the order $\log^2 n$ height of the uniform random leaf. 
\begin{OP}
\label{OP:original}
In the original interval-splitting model, analyze the distribution of the height of the 
leaf $i(n)$ in $\DTCS(n)$ and $\CTCS(n)$.
\end{OP}

\subsection{The tree height}
\label{sec:height?}


Write $D^*_n$ for the height of the random tree  $\CTCS(n)$ itself, that is the maximum leaf height.
\begin{Proposition}
\label{P:height}
\[ \Pr(D^*_n > (2 + \eps) \log n) \to 0 \mbox{ for all } \eps > 0  \]
\[ \Ex[D^*_n] \le 1 + 2 \log n .\]
\end{Proposition}
This has a direct ``stochastic calculus" proof, which we give in the next section as an illustration of that methodology. 
The first assertion  also follows from Theorem \ref{Tdev1}, because 
\[ \Pr(D^*_n \ge (2 +\eps) \log n) 
\le n \Pr(D_n \ge (2 +\eps) \log n) \le
n \cdot n^{-1 - \eps + o(1)} = n^{- \eps + o(1)} .
\]
A similar argument is in \cite[Theorem 1.4]{beta1}.

\subsection{An illustration of stochastic calculus: a bound for the tree height}
\label{sec:bdheight}


We can replace the upper tail bound in Theorem \ref{Tdev1} by the following ``clean" bound, essentially similar to the $x > x_1$ case therein.
\begin{Lemma}
\label{L:ht}
$\Pr(D_n > t) \le (n-1)e^{-t} , \ 0 \le t < \infty . $
\end{Lemma}
\begin{proof}
Write $(X_t \equiv X^{(n)}_t, 0 \le t < \infty)$ for the HD chain started at $X_0 = n$, so 
$D_n = \inf\{t: X_t = 1\}$.
From the transition rates,
\[
\Ex [dX_t \mid X_t = j] = - \sum_{i=1}^{j-1}\frac{j- i}{j-i} \, dt = - (j-1) dt \mbox{ on } \{X_t \ge 2\} .
\]
So setting $Y_t := X_t - 1$ we have $Y_0 = n-1$ and
\[ 
\Ex [dY_t \mid \FF_t]  = - Y_t  \, dt,\ 0 \le t < \infty .
\]
So 
\[ \Ex [Y_t] = (n-1)e^{-t} 
\]
and then
\[
\Pr(D_n > t) = \Pr(Y \geq 1) \le (n-1)e^{-t} .\]
\qed
\end{proof}

Now from Boole's inequality and Lemma \ref{L:ht}
\[ \Pr(D^*_n > t) \le n \Pr(D_n > t) \le n(n-1)e^{-t} \]
and so 
\[ \Pr(D^*_n > (2 + \eps) \log n) \to 0 \mbox{ for all } \eps > 0  \]
\[ \Ex[D^*_n] \le \int_0^\infty \min(1,n(n-1)e^{-t}) \ dt \le 1 + 2 \log n \]
establishing Proposition \ref{P:height}.

\subsection{Heuristics for maximum leaf height}
\label{sec:leaf-ht}


One aspect where there seems to be a substantial qualitative difference between the  discrete and continuous time models concerns the tree height.
We will discuss the discrete case in Section \ref{sec:htDT}, and continue with the continuous case $D^*_n$ here.

From Theorem \ref{TD1},  Proposition \ref{P:height} and the obvious relation 
$\Ex [D^*_n] \ge \Ex [D_n]$ we know rigorously
\[ (6/\pi^2 + o(1)) \log n \le \Ex[D^*_n] \le (2 + o(1)) \log n . \]

\begin{OP}
\label{OPmax}
Show that $D^*_n \sim c \log n$ in probability, and identify the constant $c$.
\end{OP}
We conjecture (not very confidently) that in fact
this holds for
\begin{equation}
 c :=  1 + \mu + \sigma^2/2 =   1.878...  .
 \label{Con:ht}
 \end{equation}
Here $ \mu := 1/\zeta(2) = 6/\pi^2 = 0.6079... ; \quad 
\sigma^2 := 2 \zeta(3)/\zeta^3(2)  = 0.5401...$
as in the CLT (Theorem \ref{TNorm}):
\[
\frac{D_n - \mu \log n}{\sqrt{\log n}} \to_d  \mathrm{Normal}(0, \sigma^2) \ \mbox{ as } n \to \infty .
\]
We give the heuristic argument for \eqref{Con:ht} below.
This contains an essentially rigorous argument for 
\begin{equation}
 \Ex[D^*_n] \ge (1 + \mu + o(1)) \log n .
 \label{Con:ht2}
 \end{equation}
 An alternative (but weaker) lower bound is indicated in Section \ref{sec:LBgreedy} below.

A naive starting argument would be to believe that $D^*_n$ behaves as the maximum of $n$ i.i.d.\ samples from the approximating Normal distribution, which would give
\begin{equation}
D^*_n \approx \mu \log n + \sqrt{2 \log n} \times \sqrt{ \sigma^2 \log  n} 
= \bigpar{\mu+2^{1/2} \sigma } \log n = 1.65...  \log n .
\label{eq:normtail2}
\end{equation}
But (\ref{eq:normtail2}) is in fact not the right way to study $D^*_n$, because of the ``fringe" behavior in the continuous model.
Figure \ref{Fig:400} gives a hint about the issue, which is that there are some unusually long  terminal edges to a pair of leaves.
The $\CTCS(n)$ tree has order $n$ terminal edges to a pair of leaves; in the heuristics below 
we take as this as $n$ for simplicity (this should only affect the estimate of $D^*_n$ by $\pm O(1)$).
These $n$ edges have i.i.d.\ Exponential(1) distribution, and the (asymptotic) structure of the largest of these $n$ lengths is well-known: 
the lengths in decreasing order are
\[ (\log n + \xi_1, \log n + \xi_2, \log n + \xi_3, \ldots) \]
where $\infty > \xi_1 > \xi_2 > \xi_3 > \ldots > - \infty$
are the largest points of the Poisson point process on $\Reals$ with rate $e^{-x}$, so that $\xi_1$ has the standard Gumbel distribution
\[ \Pr(\xi_1 \le x) = \exp(-e^{-x}), \ - \infty < x < \infty . \]
By imagining that the longest such edge is  attached to the tree at the typical leaf depth $D_n$,
and using the Normal limit for the random leaf heights $D_n$, we assert a lower bound
\[ \Pr( D^*_n \le (\mu + 1) \log n - \omega_n \sqrt{ \log n} ) \to 0 \mbox{ for any } \omega_n \to \infty . \]
This construction could certainly be made rigorous to prove \eqref{Con:ht2}.
However, we conjecture that we get the correct behavior for $D^*_n$ by maximizing over all the $o(n)$ longest fringe edges.
Imagine that each of these longest fringe edges is attached to  the tree at independent depths $D_n$.
So
\[ D^*_n \approx  (\mu + 1) \log n + H_n \]
\[ H_n := \max_i (\xi_i + \nu_i) \]
for $(\xi_i)$ as above and $(\nu_i)$ i.i.d.\ Normal $(0, \alpha_n^2)$, with
 $\alpha^2_n =   \sigma^2 \log n$ in the notation of the Normal limit for $D_n$.

To analyze $H_n$,
write $\bar{\Phi}_n(\cdot)$ for the tail distribution function of Normal $(0, \alpha_n^2)$.
Because the pairs $(\xi_i, \nu_i)$ form a Poisson process we have
\begin{eqnarray}
- \log \Pr(H_n \le y) &=& \int_{- \infty}^\infty e^{-x}  \ \bar{\Phi}_n(y-x)\ dx \label{Hni}\\
&=& e^{-y} \int_{- \infty}^\infty e^{y-x}  \ \bar{\Phi}(\sfrac{y-x}{\alpha_n}) \ dx \nonumber \\
&=& e^{-y}  \alpha_n \int_{- \infty}^\infty e^{\alpha_n u}  \ \bar{\Phi}(u) \ du \nonumber
\end{eqnarray}
where $\bar{\Phi}$ refers to the standard Normal distribution, and $\phi(\cdot)$  below is its density.
The integrand above is maximized for $u$ around $\alpha_n$, so setting $v = u - \alpha_n$
and using  $\bar{\Phi}(z) \sim \phi(z)/z$ as $z \to \infty$,
\begin{eqnarray*}
&\approx &e^{-y} \alpha_n  \sfrac{1}{\sqrt{2\pi}}  \int_{- \infty}^\infty \exp(\alpha_n(v+\alpha_n)) \exp(-(v+\alpha_n)^2/2)\  \sfrac{1}{v+\alpha_n}  \ dv \\
&\approx& e^{-y}  \sfrac{1}{\sqrt{2\pi}}  \int_{- \infty}^\infty \exp(-v^2/2 + \alpha_n^2/2) \ dv \\
&=& e^{-y} e^{\alpha_n^2/2} .
\end{eqnarray*}
Putting all this together
\[ - \log \Pr(H_n \le y + \alpha_n^2/2) \approx e^{-y} 
\]
and the final conclusion is
\[ D^*_n \approx c \log n  + \xi; \quad c :=  1 + \mu + \sigma^2/2 =   1.878...     \]
where $\xi$ has standard Gumbel distribution.

Now this outline is too crude to believe that the $+ \xi$ term above is correct, but this value of $c$ seems plausible.
Here is one ``reality check" for the argument/calculation above.
Look back at the first integral \eqref{Hni}: for given $y$, the relevant values of $x$ are around $y - \alpha_n^2$.
The relevant values of $y$ are around $\alpha_n^2/2$, so overall the relevant values of $x$ are around $- \alpha_n^2/2$.
This corresponds to the $\nu_i$ around position $- \alpha_n^2/2$, and the number of such edges is around 
$\exp(\alpha_n^2/2) \approx n^{0.27}$.
So one implicit assumption was
\begin{quote}
If we pick $n^{0.27}$ random leaves from $\CTCS(n)$, then the distribution of their maximum height is essentially the same as $n^{0.27}$
picks from the corresponding Normal distribution.
\end{quote}
Is this plausible?
We have the correlation \eqref{rhon} between heights of leaves situated $n^{0.73}$ apart on the interval, but what is relevant here is 
the dependence between the {\em tails} of these leaf-height distributions,  which has not been studied.

\subsubsection{The greedy lower bound}
\label{sec:LBgreedy}


One can also
consider the length $D^+_n$ of the path from the root that is chosen via a natural greedy algorithm,
taking the larger sub-clade at each split.
This is the absorption time for the modification of the HD chain
\eqref{lambda-rates} with
\begin{equation}
\lambda^+_{m,m-i} := 
\begin{cases}
\sfrac{1}{m-i} + \sfrac{1}{i}, & 1 \le i < m/2, 
\\
\frac{1}{m/2} , &  i =m/2.
\end{cases}
\label{lambda+rates}
\end{equation}
Following
\eqref{approx} and Section \ref{sec:approxc},
this chain has an approximation
\begin{equation}
 \log n - Y^+_t    \mbox{ while } Y^+_t < \log n 
 \label{approx+}
 \end{equation}
  where  $(Y^+_t, 0 \le t < \infty)$ is the subordinator with  {\em L\'{e}vy measure} 
  $\psi^+_\infty$ and corresponding $\sigma$-finite density $f^+_\infty$ on
  $(0,\infty)$ defined as
\begin{equation}
 \psi^+_\infty[a, \infty) :=  - \log (1 - e^{-a}) - a,
\quad  f_\infty(a) :=  \sfrac{1}{1 - e^{-a}}, \quad
  \ 0 < a < \log 2,
  \label{muinf+}
  \end{equation} 
and supported on $[0,\log 2]$.

So the mean drift of this subordinator is
\begin{eqnarray} 
\rho^+ &: =& \int_0^{\log 2} \psi^+_\infty[a, \infty) \dd a 
=  \int_0^{\log 2} -  \log (1 - e^{-a}) \dd a - \frac{(\log 2)^2}2 
\nonumber \\
&=& \int_{0}^{\log2}  \sum_{i=1}^\infty \frac{e^{-ia}}{i} \dd a
- \frac{(\log 2)^2}2 
= \sum_{i=1}^\infty \frac{1-2^{-i}}{i^2} - \frac{(\log 2)^2}{2} = \frac{\pi^2}{12}  \label{psirho+}
\end{eqnarray}
where the final equality uses the dilogarithm \cite[25.12.6]{NIST}
\[ \sum_{i=1}^\infty (2^{-i})/i^2 := \mathrm{Li}_2 (1/2) = \pi^2/12 - (\log 2)^2/2 . \]
This suggests the asymptotics 
\[ \Ex[D^+_n] \sim c^+ \log n  \mbox{ for } c^+ = 1/\rho^+ = 12/\pi^2 . \]
Comparing with \eqref{Con:ht} we see that the constant $c^+ = 2\mu$ here is smaller than the constant $(1 + \mu)$ in the lower bound for $\Ex[D^*_n]$, so 
we do not improve on the latter bound by observing that $D^*_n \ge D^+_n$.
However, one might be able to combine this ``greedy" procedure with the ``longest terminal edge" construction above.

\subsection{The height of $\DTCS(n)$}
\label{sec:htDT}


We now turn to the discrete context.
 We remark that  Section \ref{sec:drawn} contains a specific context where the discrete tree height is relevant. 
 
 Write $L_n^*$ and $L_n$ for the tree height and the height of a random leaf in  $\DTCS(n)$. 
 Recall that $\Ex[L_n] \sim  3 \pi^{-2} \log^2 n$ (Theorem \ref{Lnorm}). 
 Although the CLTs for $D_n$ and $L_n$ seem analogous,
 there is a qualitative difference regarding tree-height, because
  (as noted in Section \ref{sec:leaf-ht}) in  $\CTCS(n)$ the tree height
 is affected by the extremes of the terminal edge lengths, which cannot happen for  $\DTCS(n)$.

Note that the tail bound for $L_n$ in Theorem \ref{T:10} does not help directly  for studying $L^*_n$ because the bound is $O(n^{-\delta})$ rather than $o(n^{-1})$.

\begin{Theorem}[\cite{beta1} Theorem 1.5]
\label{thmE} 
Let $\beta=\min_{\alpha>1/\log 2}\bigl[\alpha+\tfrac{4\alpha^2\zeta(3)}{\alpha\log 2-1}\bigr]\approx 42.9$. For $\eps\in (0,1)$, 
\[
\Pr\Bigl( L^*_n\ge (1+\eps)\beta\log^2 n\Bigr)\le \exp\bigl(-\Theta(\eps\log n)\bigr).
\]
\end{Theorem}

\noindent
 Analogous to Open Problem \ref{OPmax} we conjecture
 \begin{OP}
\label{OPmax2}
Show that $L^*_n \sim c \log^2 n$ in probability, and identify the constant $c$.
\end{OP}
As in Section \ref{sec:LBgreedy},
one could also\footnote{This idea is mentioned in \cite{me_clad} but there is a foolish calculus error leading to an incorrect conclusion.} 
consider the length $L^+_n$ of the path from the root that is chosen via the natural greedy algorithm,
taking the larger sub-clade at each split. 
Numerics suggest that $\Ex[L^+_n - L_n]$ grows slightly faster than $(\log n)\cdot (\log \log n)$.
This  suggests that the limit constant $c$ may in fact be the lower bound $3 \pi^{-2}$ arising from $\Ex[D_n]$ itself.

 \section{The occupation measure and the fringe tree}
 \label{sec:OPfringe}

 \subsection{The occupation measure}
 \label{sec:OP}
 
 Here is the second way in which the HD chain is relevant to this article.
 The chain describes the number of descendant leaves of a node, as one moves at speed $1$
along the path from the root to a uniform random leaf.
We study the ``occupation measure", that is
 \begin{equation}
 \mbox{
 $a(n,i) := $ probability that the chain started at state $n$ is ever in state $i$.
 }
 \label{def:ani}
 \end{equation}
 So $a(n,n) = a(n,1) = 1$.
To see  the relevance of $a(n,i)$ to the tree model,
we 
let $N_n(j)$ be the number of subtrees of $\CTCS(n)$ that have $j$ leaves;
thus, 
for $j\ge2$,  $N_n(j)$ is the number of internal nodes of $\CTCS(n)$ that
have exactly $j$ leaves as descendants. 
Then, conditioned on $\CTCS(n)$, the number of leaves that are in some subtree
with $i$ leaves is $iN_n(i)$, and thus the (conditional) probability that a
random leaf is in such a subtree is $iN_n(i)/n$. Taking the expectation we
find
\begin{align}\label{jup}
  a(n,i) = \frac{i\Ex[N_n(i)]}{n}
\end{align}
and, conversely,
 \begin{equation}\label{jun12}
\Ex[N_n(i)]
= n a(n,i)/i .
 \end{equation}
 It seems very intuitive (but not obvious at a rigorous level) that the limits $a(i) = \lim_{n \to \infty} a(n,i)$ exist.
Note that   $\sum_{i=2}^n a(n,i)/h_{i-1} $ is just the mean absorption time $\Ex[ D_n]$, 
 so (from Theorem \ref{TNorm}) we anticipate that,
 assuming the limits exist, 
 \begin{align}
 \sum_{i=2}^n \sfrac{a(i)}{\log i }\sim \Ex [D_n] \sim (6/\pi^2) \log n \mbox{ as } n \to \infty.
\end{align}
 This in turn suggests
 \footnote{And this argument explains why the constant $6/\pi^2$ must be the same in Theorems \ref{TNorm} and \ref{T:alimit}.  Similarly one sees heuristically that $\Ex[L_n] = \sum_{i=2}^n a(n,i) \sim  \sum_{i=2}^n a(i) \sim 3\pi^{-2} \log^2 n$, as stated in Theorem \ref{Lnorm}.}
 \begin{align}
a(i) \sim \sfrac{6}{\pi^2} \sfrac{\log i}{i}  \mbox{ as } i \to \infty.
\end{align}
 However, there seems no intuitive reason to think there should be some simple formula for the limits $a(i)$.
 So the following result was surprising to us.
  \begin{Theorem} [Occupation measure]
  \label{T:alimit}
 For each $i = 2,3,\ldots$,
 \begin{align}
a(i) : = \lim_{n \to \infty} a(n,i) &= \frac{6 h_{i-1}}{\pi^2 (i-1)} .\label{talimit}
 \end{align}
 And $a(1) = 1$.
 \end{Theorem}
 This is the starting point for our analysis of the  {\em fringe distribution} in Section \ref{sec:OPfringe}.
 We currently know 3 quite different proofs\footnote{And the existence of the limits (without the explicit formula) can be proved by analysis of recursions: implicit in \cite[Theorem 2.16]{beta1}.} 
 of Theorem \ref{T:alimit}.

\smallskip \noindent
{\bf 1.} One method \cite{HDchain} (straightforward in outline, though somewhat tedious in detail)\footnote{A simplification of that proof has been found by  Luca Pratelli and Pietro Rigo (personal communication).}   
is to first prove by coupling that the limits $a(i)$ exist. 
The limits must satisfy a certain infinite set of equations; the one solution $\frac{6 h_{i-1}}{\pi^2 (i-1)}$ was found by inspired guesswork.
  Then check that the solution is unique.

\smallskip \noindent
  {\bf 2.} Iksanov  \cite{iksanovHD} repeats his method for proving the CLT \cite{iksanovCLT} by exploiting the exact relationship with 
  regenerative composition structures, enabling  a shorter derivation of Theorem \ref{T:alimit} from known results in that theory.
  This methodology is clearly worth further consideration.

\smallskip \noindent
{\bf 3.} In Section \ref{sec:surprise} we outline a third proof \cite{beta3-arxiv,beta4-arxiv}, illustrating how to exploit the exchangeable representation of $\CTCS(\infty)$.

\subsection{The (limit) fringe tree}
\label{sec:fringe}


To be consistent with the {\em cladogram} representation described below, we work here in the discrete time $\DTCS(n)$ setting:
the definition \eqref{def:ani} of $a(n,i)$ is of course unchanged in discrete time.

  The motivation for Theorem \ref{T:alimit} involves the (asymptotic) {\em fringe tree}
  for the random tree model $\DTCS(n)$, that is 
   the $n \to \infty$ local weak limit of the tree relative to a typical leaf.
(We talk rather casually about {\em fringe tree} or {\em fringe process} 
or {\em fringe distribution} -- see Section \ref{sec:fringeterm} for a more careful account of terminology and local weak convergence.)
  It is straightforward to see that  the fringe tree can be described in terms of the limits $(a(i), i \ge 1)$ as follows.
  
  \medskip \noindent
  ({\bf a}) 
 The sequence of clade sizes as one moves away from the distinguished leaf is the discrete time ``reverse HD" Markov chain started at state $1$,  whose ``upward" 
transition probabilities $  q^\uparrow(i,j) $ are derived 
by considering for $j >i$
\[\lim_n n^{-1} \Ex [\mbox{number of splits $j \to (i,j-i)$ or $(j-i,i)$ in $\DTCS(n)$} ].\]
Calculating this in both directions leads to the identity
\[ i^{-1} a(i)   q^\uparrow(i,j) = j^{-1} a(j) (q(j,i) + q(j,j-i)) \]
which, from the explicit formula \eqref{talimit} for $a(i)$, becomes
\begin{eqnarray}
q^\uparrow (1,j) &=& 6 \pi^{-2}     \sfrac{1}{(j-1)(j-1)}, \ j \ge 2 \nonumber \\
q^\uparrow (i,j) &=& \sfrac{i-1}{(j-1)(j-i)h_{i-1}}, \ 2 \le i < j . \label{q_up}
\end{eqnarray}
({\bf b})
 At each such step $i \to j$, there is the  sibling clade of size $j-i$, and
 this clade is distributed as $\DTCS(j-i)$,
 independently for each step.

\smallskip \noindent
One can check that (\ref{q_up})  is a probability distribution by observing
\[
\sum_{j>i} \sfrac{1}{(j-1)(j-i)} 
= \sum_{j>i}\sfrac{1}{i-1} (\sfrac{1}{j-i} - \sfrac{1}{j-1})
= \sfrac{ h_{i-1}}{i-1} .
\]

\subsection{Motivation as a phylogenetic tree model}
\label{sec:motive}


Some motivation for the random tree model came from noticing the shape of phylogenetic trees in evolutionary biology.
{\em Phylogenetic tree} is the general phrase for any tree-like graphical representation;
{\em cladogram} is more specifically a leaf-labeled binary tree, illustrated\footnote{In particular, a cladogram has no quantitative time-scale on the vertical axis.} by 
a real example in Figure \ref{Fig:cladogram} (bottom). 
Nowadays such trees are typically derived from DNA analysis of extant species\footnote{Published trees are essentially ``best fit" to noisy data, with  occasional more-than-binary
splits which cannot be resolved to successive binaries.}.
There is no biological significance to the positioning of left/right branches, though in our models it
is convenient to make the distinction.
Our random tree model $\DTCS(n)$ 
is one of many probability models that have been considered for cladograms.
The model was proposed in \cite{me_clad} in 1996 with some brief informal study then,
and with little further study until the current project. 
The motivation for this particular model came from an observation, in the small-scale study \cite{me_yule},
that in splits $m \to (i,m-i)$ in real-world phylogenetic trees, the median size of the smaller subtree scaled roughly as $m^{1/2}$. 
That data is not consistent with more classical random tree models, where the median size would be 
$O(\log m)$  or $\Theta(m)$, 
but this $m^{1/2}$ median property does hold for our particular model. 
Figure  \ref{Fig:cladogram} compares a simulation of $\DTCS(77)$ with a real cladogram on 77 species; 
these appear visually quite similar.

  Cladograms are drawn in a particular way, with the species labels on leaves in a (usually horizontal) line.
This differs from the typical visualization of (mathematical) random trees, such as 
 Galton-Watson trees, where one  starts from a root and then draws successive generations.
Figure \ref{Fig:2} illustrates how to re-draw such a tree as a cladogram, in a representation where the heights of branchpoints
are positioned at  integer heights $1,2,3, \ldots$.
Doing this in a natural way (as in Figures \ref{Fig:2} and \ref{Fig:cladogram}), 
the height of the cladogram is equal to the height (maximum leaf height) of the tree in the usual
successive-generations picture -- see Section \ref{sec:drawn}.
So in particular, the height of the cladogram representation of $\DTCS(n)$ 
is the tree-height studied in \cite{beta1}, known to be of order $\log^2 n$: this is our $L_n^*$ in Section \ref{sec:htDT}.

\begin{figure}[ht]
\setlength{\unitlength}{0.038in}
\begin{picture}(60,46)(-30,-5)
\multiput(0,0)(3,0){20}{\circle*{0.9}}
\put(30,0){\circle{1.8}}
\put(30,0){\line(0,1){8}}
\put(24,0){\line(0,1){4}}
\put(27,0){\line(0,1){4}}
\put(24,4){\line(1,0){3}}
\put(25.5,4){\line(0,1){4}}
\put(25.5,8){\line(1,0){4.5}}
\put(27,8){\line(0,1){4}}

\put(33,0){\line(0,1){4}}
\put(36,0){\line(0,1){4}}
\put(33,4){\line(1,0){3}}
\put(34.5,4){\line(0,1){8}}

\put(27,12){\line(1,0){7.5}}

\put(39,0){\line(0,1){12}}
\put(42,0){\line(0,1){8}}
\put(45,0){\line(0,1){4}}
\put(48,0){\line(0,1){4}}
\put(45,4){\line(1,0){3}}
\put(46.5,4){\line(0,1){4}}
\put(42,8){\line(1,0){4.5}}
\put(45,8){\line(0,1){4}}
\put(39,12){\line(1,0){6}}
\put(30,12){\line(0,1){4}}
\put(43.5,12){\line(0,1){4}}
\put(30,16){\line(1,0){13.5}}

\put(51,0){\line(0,1){4}}
\put(54,0){\line(0,1){4}}
\put(51,4){\line(1,0){3}}
\put(52.5,4){\line(0,1){16}}

\put(36,16){\line(0,1){4}}
\put(36,20){\line(1,0){16.5}}

\put(57,0){\line(0,1){24}}
\put(39,20){\line(0,1){4}}
\put(39,24){\line(1,0){18}}
\put(40.5,24){\line(0,1){4}}

\put(18,0){\line(0,1){8}}
\put(12,0){\line(0,1){4}}
\put(15,0){\line(0,1){4}}
\put(12,4){\line(1,0){3}}
\put(13.5,4){\line(0,1){4}}
\put(13.5,8){\line(1,0){4.5}}
\put(15,8){\line(0,1){4}}
\put(21,0){\line(0,1){12}}
\put(15,12){\line(1,0){6}}
\put(16.5,12){\line(0,1){16}}
\put(16.5,28){\line(1,0){23.5}}

\put(9,0){\line(0,1){32}}
\put(34.5,28){\line(0,1){4}}
\put(9,32){\line(1,0){25.5}}
\put(33,32){\line(0,1){4}}

\put(6,0){\line(0,1){8}}
\put(0,0){\line(0,1){4}}
\put(3,0){\line(0,1){4}}
\put(0,4){\line(1,0){3}}
\put(1.5,4){\line(0,1){4}}
\put(1.5,8){\line(1,0){4.5}}
\put(3,8){\line(0,1){28}}
\put(3,36){\line(1,0){30}}
\put(28.5,36){\line(0,1){4}}

  \end{picture}
\caption{Cladogram representation of the Figure \ref{Fig:1d}  realization of $\DTCS(20)$.}
\label{Fig:2}
\end{figure}
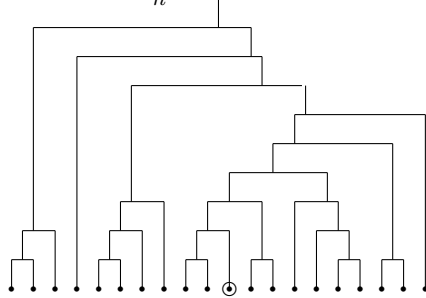

\begin{figure}
\setlength{\unitlength}{0.087in}
\begin{picture}(80,40)(8,-30)
\multiput(1,0)(1,0){77}{\circle*{0.3}}
\put(77,0){\line(0,1){1}}
\put(76,0){\line(0,1){1}}
\put(76,1){\line(1,0){1}}
\put(75,0){\line(0,1){2}}
\put(75,2){\line(1,0){1.5}}
\put(76.5,1){\line(0,1){1}}

\put(74,0){\line(0,1){1}}
\put(73,0){\line(0,1){1}}
\put(73,1){\line(1,0){1}}
\put(72,0){\line(0,1){2}}
\put(72,2){\line(1,0){1.5}}
\put(73.5,1){\line(0,1){1}}

\put(75.5,2){\line(0,1){1}}
\put(72.5,2){\line(0,1){1}}
\put(72.5,3){\line(1,0){3}}
\put(73.5,3){\line(0,1){1}}
\put(71,0){\line(0,1){4}}
\put(71,4){\line(1,0){2.5}}

\put(70,0){\line(0,1){1}}
\put(69,0){\line(0,1){1}}
\put(69,1){\line(1,0){1}}

\put(68,0){\line(0,1){1}}
\put(67,0){\line(0,1){1}}
\put(67,1){\line(1,0){1}}

\put(69.5,1){\line(0,1){1}}
\put(67.5,1){\line(0,1){1}}
\put(67.5,2){\line(1,0){2}}

\put(66,0){\line(0,1){3}}
\put(68.5,2){\line(0,1){1}}
\put(66,3){\line(1,0){2.5}}
\put(65,0){\line(0,1){4}}
\put(68,3){\line(0,1){1}}
\put(65,4){\line(1,0){3}}

\put(66.5,4){\line(0,1){1}}
\put(71.75,4){\line(0,1){1}}
\put(66.5,5){\line(1,0){5.25}}
\put(68.5,5){\line(0,1){1}}

\put(64,0){\line(0,1){6}}
\put(64.0,6){\line(1,0){4.5}}
\put(65.5,6){\line(0,1){1}}
\put(63,0){\line(0,1){7}}
\put(63,7){\line(1,0){2.5}}
\put(64.25,7){\line(0,1){1}}

\put(62,0){\line(0,1){1}}
\put(61,0){\line(0,1){1}}
\put(61,1){\line(1,0){1}}

\put(61.5,1){\line(0,1){1}}
\put(60,0){\line(0,1){2}}
\put(60,2){\line(1,0){1.5}}
\put(60.75,2){\line(0,1){1}}
\put(59,0){\line(0,1){3}}
\put(59,3){\line(1,0){1.75}}
\put(59.75,3){\line(0,1){1}}

\put(58,0){\line(0,1){1}}
\put(57,0){\line(0,1){1}}
\put(57,1){\line(1,0){1}}

\put(57.5,1){\line(0,1){3}}
\put(57.5,4){\line(1,0){2.25}}
\put(58.25,4){\line(0,1){1}}

\put(54,0){\line(0,1){2}}
\put(55.5,1){\line(0,1){1}}
\put(54,2){\line(1,0){1.5}}
\put(53,0){\line(0,1){3}}
\put(54.5,2){\line(0,1){1}}
\put(53,3){\line(1,0){1.5}}
\put(53.75,3){\line(0,1){2}}
\put(53.75,5){\line(1,0){4.5}}
\put(55.75,5){\line(0,1){3}}
\put(55.75,8){\line(1,0){8.5}}

\put(56,0){\line(0,1){1}}
\put(55,0){\line(0,1){1}}
\put(55,1){\line(1,0){1}}

\put(52,0){\line(0,1){1}}
\put(51,0){\line(0,1){1}}
\put(51,1){\line(1,0){1}}
\put(51.5,1){\line(0,1){8}}
\put(59.5,8){\line(0,1){1}}
\put(51.5,9){\line(1,0){8}}

\put(50,0){\line(0,1){1}}
\put(49,0){\line(0,1){1}}
\put(49,1){\line(1,0){1}}

\put(48,0){\line(0,1){2}}
\put(49.5,1){\line(0,1){1}}
\put(48,2){\line(1,0){1.5}}

\put(47,0){\line(0,1){1}}
\put(46,0){\line(0,1){1}}
\put(46,1){\line(1,0){1}}
\put(46.5,1){\line(0,1){2}}
\put(48.5,2){\line(0,1){1}}
\put(46.5,3){\line(1,0){2}}

\put(45,0){\line(0,1){4}}
\put(47.25,3){\line(0,1){1}}
\put(45,4){\line(1,0){2.25}}

\put(44,0){\line(0,1){1}}
\put(43,0){\line(0,1){1}}
\put(43,1){\line(1,0){1}}

\put(43.5,1){\line(0,1){4}}
\put(46,4){\line(0,1){1}}

\put(43.5,5){\line(1,0){2.5}}
\put(44.5,5){\line(0,1){5}}
\put(54.5,9){\line(0,1){1}}
\put(44.5,10){\line(1,0){10}}

\put(42,0){\line(0,1){1}}
\put(41,0){\line(0,1){1}}
\put(41,1){\line(1,0){1}}

\put(41.5,1){\line(0,1){1}}
\put(40,0){\line(0,1){2}}
\put(40,2){\line(1,0){1.5}}
\put(40.75,2){\line(0,1){1}}
\put(39,0){\line(0,1){3}}
\put(39,3){\line(1,0){1.75}}
\put(39.75,3){\line(0,1){1}}
\put(38,0){\line(0,1){4}}

\put(38,4){\line(1,0){1.75}}
\put(38.75,4){\line(0,1){7}}
\put(48.5,10){\line(0,1){1}}
\put(38.75,11){\line(1,0){9.75}}

\put(37,0){\line(0,1){12}}
\put(43,11){\line(0,1){1}}
\put(37,12){\line(1,0){6}}

\put(36,0){\line(0,1){1}}
\put(35,0){\line(0,1){1}}
\put(35,1){\line(1,0){1}}
\put(35.5,1){\line(0,1){1}}
\put(34,0){\line(0,1){2}}
\put(34,2){\line(1,0){1.5}}
\put(33,0){\line(0,1){1}}
\put(32,0){\line(0,1){1}}
\put(32,1){\line(1,0){1}}
\put(32.5,1){\line(0,1){2}}
\put(34.5,2){\line(0,1){1}}
\put(32.5,3){\line(1,0){2}}

\put(31,0){\line(0,1){1}}
\put(30,0){\line(0,1){1}}
\put(30,1){\line(1,0){1}}
\put(30.5,1){\line(0,1){3}}
\put(33.5,3){\line(0,1){1}}
\put(30.5,4){\line(1,0){3}}

\put(29,0){\line(0,1){5}}
\put(32,4){\line(0,1){1}}
\put(29,5){\line(1,0){3}}

\put(28,0){\line(0,1){6}}
\put(30.5,5){\line(0,1){1}}
\put(28,6){\line(1,0){2.5}}

\put(28.75,6){\line(0,1){7}}
\put(39.5,12){\line(0,1){1}}
\put(28.75,13){\line(1,0){10.75}}

\put(27,0){\line(0,1){1}}
\put(26,0){\line(0,1){1}}
\put(26,1){\line(1,0){1}}
\put(26.5,1){\line(0,1){13}}
\put(33,13){\line(0,1){1}}
\put(26.5,14){\line(1,0){6.5}}

\put(25,0){\line(0,1){1}}
\put(24,0){\line(0,1){1}}
\put(24,1){\line(1,0){1}}
\put(24.5,1){\line(0,1){1}}
\put(23,0){\line(0,1){2}}
\put(23,2){\line(1,0){1.5}}

\put(23.5,2){\line(0,1){13}}
\put(29,14){\line(0,1){1}}
\put(23.5,15){\line(1,0){5.5}}

\put(22,0){\line(0,1){1}}
\put(21,0){\line(0,1){1}}
\put(21,1){\line(1,0){1}}

\put(21.5,1){\line(0,1){1}}
\put(20,0){\line(0,1){2}}
\put(20,2){\line(1,0){1.5}}
\put(20.75,2){\line(0,1){1}}
\put(19,0){\line(0,1){3}}
\put(19,3){\line(1,0){1.75}}
\put(19.75,3){\line(0,1){1}}

\put(18,0){\line(0,1){1}}
\put(17,0){\line(0,1){1}}
\put(17,1){\line(1,0){1}}

\put(17.5,1){\line(0,1){1}}
\put(16,0){\line(0,1){2}}
\put(16,2){\line(1,0){1.5}}
\put(16.75,2){\line(0,1){1}}
\put(15,0){\line(0,1){3}}
\put(15,3){\line(1,0){1.75}}
\put(15.75,3){\line(0,1){1}}

\put(15.75,4){\line(1,0){4}}
\put(17.75,4){\line(0,1){1}}

\put(14,0){\line(0,1){5}}
\put(14,5){\line(1,0){3.75}}
\put(15.5,5){\line(0,1){1}}

\put(13,0){\line(0,1){6}}
\put(13,6){\line(1,0){2.5}}
\put(14,6){\line(0,1){10}}
\put(26,15){\line(0,1){1}}
\put(14,16){\line(1,0){12}}

\put(12,0){\line(0,1){17}}
\put(19,16){\line(0,1){1}}
\put(12,17){\line(1,0){7}}

\put(11,0){\line(0,1){1}}
\put(10,0){\line(0,1){1}}
\put(10,1){\line(1,0){1}}

\put(9,0){\line(0,1){1}}
\put(8,0){\line(0,1){1}}
\put(8,1){\line(1,0){1}}

\put(10.5,1){\line(0,1){1}}
\put(8.5,1){\line(0,1){1}}
\put(8.5,2){\line(1,0){2}}

\put(9.5,2){\line(0,1){16}}
\put(15,17){\line(0,1){1}}
\put(9.5,18){\line(1,0){5.5}}

\put(7,0){\line(0,1){1}}
\put(6,0){\line(0,1){1}}
\put(6,1){\line(1,0){1}}

\put(6.5,1){\line(0,1){1}}
\put(5,0){\line(0,1){2}}
\put(5,2){\line(1,0){1.5}}
\put(5.75,2){\line(0,1){1}}
\put(4,0){\line(0,1){3}}
\put(4,3){\line(1,0){1.75}}
\put(4.75,3){\line(0,1){1}}
\put(3,0){\line(0,1){4}}

\put(3,4){\line(1,0){1.75}}
\put(3.75,4){\line(0,1){15}}
\put(12,18){\line(0,1){1}}
\put(3.75,19){\line(1,0){8.25}}

\put(2,0){\line(0,1){20}}
\put(7.5,19){\line(0,1){1}}
\put(2,20){\line(1,0){5.5}}

\put(1,0){\line(0,1){21}}
\put(4.5,20){\line(0,1){1}}
\put(1,21){\line(1,0){3.5}}
\put(2.5,21){\line(0,1){1}}
\end{picture}

\vspace*{-1.7in}

\includegraphics[width=5.1in]{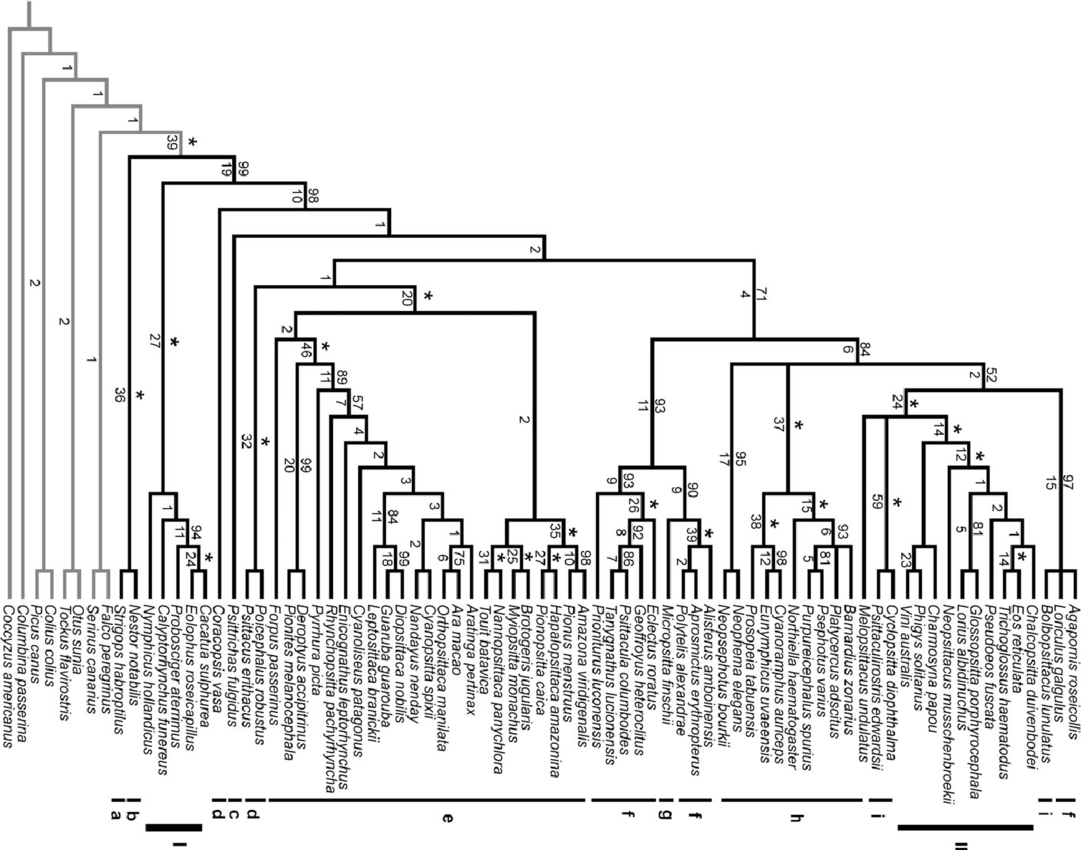}
 
 \caption{Bottom: cladogram showing phylogenetics of 77 parrot species, from \cite{parrot:phy}.  
 Top: simulation of $\DTCS(77)$, drawn as fringe distribution in the style of biological cladograms.
 }
\label{Fig:cladogram}
 \end{figure}

\begin{figure}
\setlength{\unitlength}{0.087in}
\begin{picture}(80,40)(8,0)
\multiput(1,0)(1,0){77}{\circle*{0.3}}

\put(73,0){\line(0,1){1}}
\put(74,0){\line(0,1){1}}
\put(73,1){\line(1,0){1}}
\put(73.5,1){\line(0,1){1}}
\put(72,0){\line(0,1){2}}
\put(72,2){\line(1,0){1.5}}

\put(76,0){\line(0,1){1}}
\put(77,0){\line(0,1){1}}
\put(76,1){\line(1,0){1}}
\put(76.5,1){\line(0,1){1}}
\put(75,0){\line(0,1){2}}
\put(75,2){\line(1,0){1.5}}

\put(73,2){\line(0,1){1}}
\put(76,2){\line(0,1){1}}
\put(73,3){\line(1,0){3}}

\put(70,0){\line(0,1){1}}
\put(71,0){\line(0,1){1}}
\put(70,1){\line(1,0){1}}
\put(70.5,1){\line(0,1){1}}
\put(69,0){\line(0,1){2}}
\put(69,2){\line(1,0){1.5}}

\put(68,0){\line(0,1){3}}
\put(70,2){\line(0,1){1}}
\put(68,3){\line(1,0){2}}

\put(69.5,3){\line(0,1){1}}
\put(74.5,3){\line(0,1){1}}
\put(69.5,4){\line(1,0){5}}

\put(67,0){\line(0,1){5}}
\put(72.5,4){\line(0,1){1}}
\put(67,5){\line(1,0){5.5}}

\put(65,0){\line(0,1){1}}
\put(66,0){\line(0,1){1}}
\put(65,1){\line(1,0){1}}
\put(65.5,1){\line(0,1){1}}
\put(64,0){\line(0,1){2}}
\put(64,2){\line(1,0){1.5}}
\put(65,2){\line(0,1){4}}
\put(72,5){\line(0,1){1}}
\put(65,6){\line(1,0){7}}

\put(63,0){\line(0,1){7}}
\put(70.5,6){\line(0,1){1}}
\put(63,7){\line(1,0){7.5}}

\put(61,0){\line(0,1){1}}
\put(62,0){\line(0,1){1}}
\put(61,1){\line(1,0){1}}
\put(61.5,1){\line(0,1){1}}
\put(60,0){\line(0,1){2}}
\put(60,2){\line(1,0){1.5}}

\put(59,0){\line(0,1){3}}
\put(61,2){\line(0,1){1}}
\put(59,3){\line(1,0){2}}
\put(60.5,3){\line(0,1){5}}
\put(70,7){\line(0,1){1}}
\put(60.5,3){\line(0,1){5}}
\put(60.5,8){\line(1,0){9.5}}

\put(57,0){\line(0,1){1}}
\put(58,0){\line(0,1){1}}
\put(57,1){\line(1,0){1}}
\put(57.5,1){\line(0,1){8}}
\put(68,8){\line(0,1){1}}
\put(57.5,9){\line(1,0){10.5}}

\put(56,0){\line(0,1){10}}
\put(67,9){\line(0,1){1}}
\put(56,10){\line(1,0){11}}

\put(55,0){\line(0,1){11}}
\put(66.5,10){\line(0,1){1}}
\put(55,11){\line(1,0){11.5}}

\put(51,0){\line(0,1){1}}
\put(52,0){\line(0,1){1}}
\put(51,1){\line(1,0){1}}
\put(53,0){\line(0,1){1}}
\put(54,0){\line(0,1){1}}
\put(53,1){\line(1,0){1}}
\put(51.5,1){\line(0,1){1}}
\put(53.5,1){\line(0,1){1}}
\put(51.5,2){\line(1,0){2}}

\put(50,0){\line(0,1){3}}
\put(52.5,2){\line(0,1){1}}
\put(50,3){\line(1,0){2.5}}

\put(48,0){\line(0,1){1}}
\put(49,0){\line(0,1){1}}
\put(48,1){\line(1,0){1}}
\put(48.5,1){\line(0,1){3}}
\put(52,3){\line(0,1){1}}
\put(48.5,4){\line(1,0){3.5}}

\put(47,0){\line(0,1){5}}
\put(51,4){\line(0,1){1}}
\put(47,5){\line(1,0){4}}

\put(50.5,5){\line(0,1){7}}
\put(66,11){\line(0,1){1}} 
\put(50.5,12){\line(1,0){15.5}} 

\put(46,0){\line(0,1){13}}
\put(62,12){\line(0,1){1}} 
\put(46,13){\line(1,0){16}}

\put(44,0){\line(0,1){1}}
\put(45,0){\line(0,1){1}}
\put(44,1){\line(1,0){1}}
\put(44.5,1){\line(0,1){1}}
\put(43,0){\line(0,1){2}}
\put(43,2){\line(1,0){1.5}}

\put(44,2){\line(0,1){12}}
\put(61.5,13){\line(0,1){1}} 
\put(44,14){\line(1,0){17.5}}

\put(42,0){\line(0,1){15}}
\put(60,14){\line(0,1){1}} 
\put(42,15){\line(1,0){18}}

\put(40,0){\line(0,1){1}}
\put(41,0){\line(0,1){1}}
\put(40,1){\line(1,0){1}}
\put(40.5,1){\line(0,1){1}}
\put(39,0){\line(0,1){2}}
\put(39,2){\line(1,0){1.5}}

\put(40,2){\line(0,1){14}}
\put(59.5,15){\line(0,1){1}} 
\put(40,16){\line(1,0){19.5}}

\put(37,0){\line(0,1){1}}
\put(38,0){\line(0,1){1}}
\put(37,1){\line(1,0){1}}
\put(37.5,1){\line(0,1){1}}
\put(36,0){\line(0,1){2}}
\put(36,2){\line(1,0){1.5}}

\put(35,0){\line(0,1){3}}
\put(37,2){\line(0,1){1}}
\put(35,3){\line(1,0){2}}

\put(36.5,3){\line(0,1){14}}
\put(58,16){\line(0,1){1}} 
\put(36.5,17){\line(1,0){21.5}}

\put(34,0){\line(0,1){18}}
\put(56,17){\line(0,1){1}} 
\put(34,18){\line(1,0){22}}

\put(32,0){\line(0,1){1}}
\put(33,0){\line(0,1){1}}
\put(32,1){\line(1,0){1}}
\put(32.5,1){\line(0,1){1}}
\put(31,0){\line(0,1){2}}
\put(31,2){\line(1,0){1.5}}

\put(30,0){\line(0,1){3}}
\put(32,2){\line(0,1){1}}
\put(30,3){\line(1,0){2}}

\put(29,0){\line(0,1){4}}
\put(31.5,3){\line(0,1){1}}
\put(29,4){\line(1,0){2.5}}

\put(27,0){\line(0,1){1}}
\put(28,0){\line(0,1){1}}
\put(27,1){\line(1,0){1}}

\put(27.5,1){\line(0,1){4}}
\put(31,4){\line(0,1){1}}
\put(27.5,5){\line(1,0){3.5}}
\put(30,5){\line(0,1){1}}

\put(26,0){\line(0,1){6}}
\put(26,6){\line(1,0){4}}

\put(22,0){\line(0,1){1}}
\put(23,0){\line(0,1){1}}
\put(22,1){\line(1,0){1}}
\put(24,0){\line(0,1){1}}
\put(25,0){\line(0,1){1}}
\put(24,1){\line(1,0){1}}
\put(22.5,1){\line(0,1){1}}
\put(24.5,1){\line(0,1){1}}
\put(22.5,2){\line(1,0){2}}

\put(29.5,6){\line(0,1){1}}
\put(23.5,2){\line(0,1){5}}
\put(23.5,7){\line(1,0){6}}

\put(20,0){\line(0,1){1}}
\put(21,0){\line(0,1){1}}
\put(20,1){\line(1,0){1}}
\put(20.5,1){\line(0,1){7}}
\put(27.5,7){\line(0,1){1}}
\put(20.5,8){\line(1,0){7}}

\put(19,0){\line(0,1){9}}
\put(26.5,8){\line(0,1){1}}
\put(19,9){\line(1,0){7.5}}

\put(17,0){\line(0,1){1}}
\put(18,0){\line(0,1){1}}
\put(17,1){\line(1,0){1}}
\put(17.5,1){\line(0,1){9}}
\put(25.5,9){\line(0,1){1}}
\put(17.5,10){\line(1,0){8}}

\put(24.5,10){\line(0,1){9}}
\put(55.5,18){\line(0,1){1}} 
\put(24.5,19){\line(1,0){31}}

\put(47,19){\line(0,1){1}} 

\put(16,0){\line(0,1){20}}
\put(16,20){\line(1,0){31}}

\put(14,0){\line(0,1){1}}
\put(15,0){\line(0,1){1}}
\put(14,1){\line(1,0){1}}
\put(14.5,1){\line(0,1){1}}
\put(13,0){\line(0,1){2}}
\put(13,2){\line(1,0){1.5}}
\put(14,2){\line(0,1){19}}
\put(46.5,20){\line(0,1){1}} 
\put(14,21){\line(1,0){32.5}}

\put(12,0){\line(0,1){22}}
\put(45,21){\line(0,1){1}} 
\put(12,22){\line(1,0){33}}

\put(10,0){\line(0,1){1}}
\put(11,0){\line(0,1){1}}
\put(10,1){\line(1,0){1}}
\put(10.5,1){\line(0,1){1}}
\put(9,0){\line(0,1){2}}
\put(9,2){\line(1,0){1.5}}
\put(7,0){\line(0,1){1}}
\put(8,0){\line(0,1){1}}
\put(7,1){\line(1,0){1}}
\put(10,2){\line(0,1){1}}
\put(7.5,1){\line(0,1){2}}
\put(7.5,3){\line(1,0){2.5}}
\put(9,3){\line(0,1){20}}
\put(44.5,22){\line(0,1){1}} 
\put(9,23){\line(1,0){35.5}}

\put(5,0){\line(0,1){1}}
\put(6,0){\line(0,1){1}}
\put(5,1){\line(1,0){1}}
\put(5.5,1){\line(0,1){23}}
\put(42,23){\line(0,1){1}} 
\put(5.5,24){\line(1,0){36.5}}

\put(4,0){\line(0,1){25}}
\put(41,24){\line(0,1){1}} 
\put(4,25){\line(1,0){37}}

\put(2,0){\line(0,1){1}}
\put(3,0){\line(0,1){1}}
\put(2,1){\line(1,0){1}}
\put(2.5,1){\line(0,1){1}}
\put(1,0){\line(0,1){2}}
\put(1,2){\line(1,0){1.5}}
\put(2,2){\line(0,1){24}}
\put(40.5,25){\line(0,1){1}} 
\put(2,26){\line(1,0){38.5}}
\put(39,26){\line(0,1){1}} 
\end{picture}
\label{Fig:cladogram1.5}
 \end{figure}

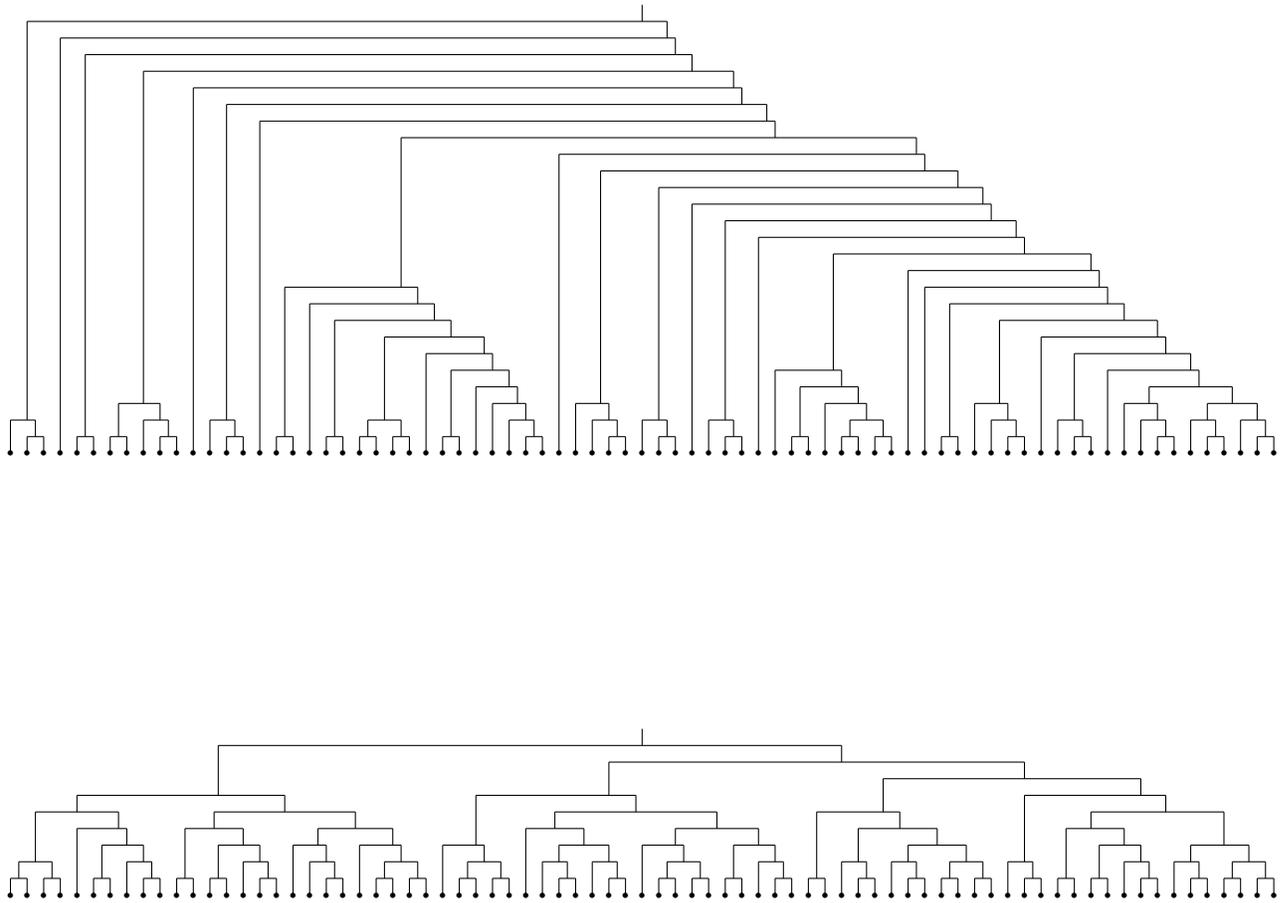
\begin{figure}
\setlength{\unitlength}{0.087in}
\begin{picture}(80,20)(8,-5)
\multiput(1,0)(1,0){77}{\circle*{0.3}}
\put(1,0){\line(0,1){1}}
\put(2,0){\line(0,1){1}}
\put(1,1){\line(1,0){1}}
\put(3,0){\line(0,1){1}}
\put(4,0){\line(0,1){1}}
\put(3,1){\line(1,0){1}}
\put(1.5,1){\line(0,1){1}}
\put(3.5,1){\line(0,1){1}}
\put(1.5,2){\line(1,0){2}}

\put(9,0){\line(0,1){1}}
\put(10,0){\line(0,1){1}}
\put(9,1){\line(1,0){1}}
\put(9.5,1){\line(0,1){1}}
\put(8,0){\line(0,1){2}}
\put(8,2){\line(1,0){1.5}}
\put(6,0){\line(0,1){1}}
\put(7,0){\line(0,1){1}}
\put(6,1){\line(1,0){1}}
\put(6.5,1){\line(0,1){2}}
\put(9,2){\line(0,1){1}}
\put(6.5,3){\line(1,0){2.5}}

\put(5,0){\line(0,1){4}}
\put(8,3){\line(0,1){1}}
\put(5,4){\line(1,0){3}}

\put(2.5,2){\line(0,1){3}}
\put(7.5,4){\line(0,1){1}}
\put(2.5,5){\line(1,0){5}}

\put(13,0){\line(0,1){1}}
\put(14,0){\line(0,1){1}}
\put(13,1){\line(1,0){1}}
\put(16,2){\line(0,1){1}}

\put(13.5,1){\line(0,1){2}}
\put(13.5,3){\line(1,0){2.5}}

\put(11,0){\line(0,1){1}}
\put(12,0){\line(0,1){1}}
\put(11,1){\line(1,0){1}}
\put(11.5,1){\line(0,1){3}}

\put(16,0){\line(0,1){1}}
\put(17,0){\line(0,1){1}}
\put(16,1){\line(1,0){1}}
\put(16.5,1){\line(0,1){1}}
\put(15,0){\line(0,1){2}}
\put(15,2){\line(1,0){1.5}}

\put(15,3){\line(0,1){1}}
\put(11.5,4){\line(1,0){3.5}}

\put(20,0){\line(0,1){1}}
\put(21,0){\line(0,1){1}}
\put(20,1){\line(1,0){1}}
\put(20.5,1){\line(0,1){1}}
\put(19,0){\line(0,1){2}}
\put(19,2){\line(1,0){1.5}}

\put(18,0){\line(0,1){3}}
\put(20,2){\line(0,1){1}}
\put(18,3){\line(1,0){2}}

\put(23,0){\line(0,1){1}}
\put(24,0){\line(0,1){1}}
\put(23,1){\line(1,0){1}}
\put(25,0){\line(0,1){1}}
\put(26,0){\line(0,1){1}}
\put(25,1){\line(1,0){1}}
\put(23.5,1){\line(0,1){1}}
\put(25.5,1){\line(0,1){1}}
\put(23.5,2){\line(1,0){2}}

\put(22,0){\line(0,1){3}}
\put(24.5,2){\line(0,1){1}}
\put(22,3){\line(1,0){2.5}}

\put(19.5,3){\line(0,1){1}}
\put(24,3){\line(0,1){1}}
\put(19.5,4){\line(1,0){4.5}}

\put(13.25,4){\line(0,1){1}}
\put(21.75,4){\line(0,1){1}}
\put(13.25,5){\line(1,0){8.5}}

\put(5,5){\line(0,1){1}}
\put(17.5,5){\line(0,1){1}}
\put(5,6){\line(1,0){12.5}}

\put(28,0){\line(0,1){1}}
\put(29,0){\line(0,1){1}}
\put(28,1){\line(1,0){1}}
\put(30,0){\line(0,1){1}}
\put(31,0){\line(0,1){1}}
\put(30,1){\line(1,0){1}}
\put(28.5,1){\line(0,1){1}}
\put(30.5,1){\line(0,1){1}}
\put(28.5,2){\line(1,0){2}}

\put(27,0){\line(0,1){3}}
\put(29.5,2){\line(0,1){1}}
\put(27,3){\line(1,0){2.5}}

\put(34,0){\line(0,1){1}}
\put(35,0){\line(0,1){1}}
\put(34,1){\line(1,0){1}}
\put(34.5,1){\line(0,1){1}}
\put(33,0){\line(0,1){2}}
\put(33,2){\line(1,0){1.5}}

\put(37,0){\line(0,1){1}}
\put(38,0){\line(0,1){1}}
\put(37,1){\line(1,0){1}}
\put(37.5,1){\line(0,1){1}}
\put(36,0){\line(0,1){2}}
\put(36,2){\line(1,0){1.5}}

\put(34,2){\line(0,1){1}}
\put(37,2){\line(0,1){1}}
\put(34,3){\line(1,0){3}}

\put(32,0){\line(0,1){4}}
\put(35.5,3){\line(0,1){1}}
\put(32,4){\line(1,0){3.5}}

\put(40,0){\line(0,1){1}}
\put(41,0){\line(0,1){1}}
\put(40,1){\line(1,0){1}}
\put(42,0){\line(0,1){1}}
\put(43,0){\line(0,1){1}}
\put(42,1){\line(1,0){1}}
\put(40.5,1){\line(0,1){1}}
\put(42.5,1){\line(0,1){1}}
\put(40.5,2){\line(1,0){2}}

\put(39,0){\line(0,1){3}}
\put(41.5,2){\line(0,1){1}}
\put(39,3){\line(1,0){2.5}}

\put(47,0){\line(0,1){1}}
\put(48,0){\line(0,1){1}}
\put(47,1){\line(1,0){1}}
\put(47.5,1){\line(0,1){1}}
\put(46,0){\line(0,1){2}}
\put(46,2){\line(1,0){1.5}}
\put(44,0){\line(0,1){1}}
\put(45,0){\line(0,1){1}}
\put(44,1){\line(1,0){1}}
\put(44.5,1){\line(0,1){2}}
\put(47,2){\line(0,1){1}}
\put(44.5,3){\line(1,0){2.5}}

\put(41,3){\line(0,1){1}}
\put(46,3){\line(0,1){1}}
\put(41,4){\line(1,0){5}}

\put(33.75,4){\line(0,1){1}}
\put(43.5,4){\line(0,1){1}}
\put(33.75,5){\line(1,0){9.75}}

\put(38.625,5){\line(0,1){1}}
\put(29,3){\line(0,1){3}}
\put(29,6){\line(1,0){9.625}}


\put(49,0){\line(0,1){1}}
\put(50,0){\line(0,1){1}}
\put(49,1){\line(1,0){1}}

\put(52,0){\line(0,1){1}}
\put(53,0){\line(0,1){1}}
\put(52,1){\line(1,0){1}}
\put(52.5,1){\line(0,1){1}}
\put(51,0){\line(0,1){2}}
\put(51,2){\line(1,0){1.5}}

\put(55,0){\line(0,1){1}}
\put(56,0){\line(0,1){1}}
\put(55,1){\line(1,0){1}}
\put(55.5,1){\line(0,1){1}}
\put(54,0){\line(0,1){2}}
\put(54,2){\line(1,0){1.5}}

\put(57,0){\line(0,1){1}}
\put(58,0){\line(0,1){1}}
\put(57,1){\line(1,0){1}}
\put(59,0){\line(0,1){1}}
\put(60,0){\line(0,1){1}}
\put(59,1){\line(1,0){1}}
\put(57.5,1){\line(0,1){1}}
\put(59.5,1){\line(0,1){1}}
\put(57.5,2){\line(1,0){2}}

\put(55,2){\line(0,1){1}}
\put(58.5,2){\line(0,1){1}}
\put(55,3){\line(1,0){3.5}}

\put(56.75,3){\line(0,1){1}}
\put(52,2){\line(0,1){2}}
\put(52,4){\line(1,0){4.75}}

\put(49.5,1){\line(0,1){4}}
\put(54.5,4){\line(0,1){1}}
\put(49.5,5){\line(1,0){5}}

\put(62,0){\line(0,1){1}}
\put(63,0){\line(0,1){1}}
\put(62,1){\line(1,0){1}}
\put(62.5,1){\line(0,1){1}}
\put(61,0){\line(0,1){2}}
\put(61,2){\line(1,0){1.5}}

\put(64,0){\line(0,1){1}}
\put(65,0){\line(0,1){1}}
\put(64,1){\line(1,0){1}}
\put(64.5,1){\line(0,1){3}}

\put(69,0){\line(0,1){1}}
\put(70,0){\line(0,1){1}}
\put(69,1){\line(1,0){1}}
\put(69.5,1){\line(0,1){1}}
\put(68,0){\line(0,1){2}}
\put(68,2){\line(1,0){1.5}}
\put(66,0){\line(0,1){1}}
\put(67,0){\line(0,1){1}}
\put(66,1){\line(1,0){1}}
\put(66.5,1){\line(0,1){2}}
\put(69,2){\line(0,1){1}}
\put(66.5,3){\line(1,0){2.5}}

\put(72,0){\line(0,1){1}}
\put(73,0){\line(0,1){1}}
\put(72,1){\line(1,0){1}}
\put(72.5,1){\line(0,1){1}}
\put(71,0){\line(0,1){2}}
\put(71,2){\line(1,0){1.5}}

\put(74,0){\line(0,1){1}}
\put(75,0){\line(0,1){1}}
\put(74,1){\line(1,0){1}}
\put(76,0){\line(0,1){1}}
\put(77,0){\line(0,1){1}}
\put(76,1){\line(1,0){1}}
\put(74.5,1){\line(0,1){1}}
\put(76.5,1){\line(0,1){1}}
\put(74.5,2){\line(1,0){2}}

\put(72,2){\line(0,1){1}}
\put(75.5,2){\line(0,1){1}}
\put(72,3){\line(1,0){3.5}}

\put(68,3){\line(0,1){1}}
\put(64.5,4){\line(1,0){3.5}}

\put(66,4){\line(0,1){1}}
\put(74,3){\line(0,1){2}}
\put(66,5){\line(1,0){8}}

\put(53.5,5){\line(0,1){2}}
\put(70.5,5){\line(0,1){1}}

\put(62,2){\line(0,1){4}}
\put(62,6){\line(1,0){8.5}}
\put(69,6){\line(0,1){1}}

\put(53.5,7){\line(1,0){15.5}}

\put(37,6){\line(0,1){2}}
\put(62,7){\line(0,1){1}}
\put(37,8){\line(1,0){25}}

\put(13.5,6){\line(0,1){3}}
\put(51,8){\line(0,1){1}}
\put(13.5,9){\line(1,0){37.5}}
\put(39,9){\line(0,1){1}}

\end{picture}
\caption{Simulations of the beta-splitting model on 77 species for other parameters: (top) $\beta = -1.5$;  (bottom) $\beta = 0$.}
\label{Fig:cladogram0}
 \end{figure}

 The mathematical theme of  \cite{me_clad} was to introduce the {\em beta-splitting model} with split probabilities
 \begin{equation}
 q(n,i) = \frac{1}{a_n(\beta)} \ \frac{\Gamma(\beta + i + 1) \Gamma(\beta +n - i + 1)} { \Gamma(i+1) \Gamma(n-i+1)} , \ 1 \le i \le n-1 
 \label{rule-0}
 \end{equation}
with a parameter $- 2 \le \beta \le \infty$ and normalizing constant $a_n(\beta)$.
 The qualitative behavior of the model is different for $\beta > -1$ than for $\beta < -1$;
in the former case the height (number of edges to the root) of a typical leaf grows as order $\log n$, 
and in the latter case as order $n^{-\beta - 1}$.
In this article we are studying the {\em critical} case $\beta = -1$, with two motivations.\footnote{Hence our terminology CS for {\em critical splitting}. But note that {\em critical} in our context is quite different from the usual {\em critical} in the context of branching processes or percolation.}

 (a) 
A stochastic model, at a critical  parameter value separating qualitatively different behaviors (loosely called a ``phase transition" by analogy with statistical physics), often has mathematically interesting special properties: we are seeing this in the current project.

(b)
Second, as mentioned above our small-scale study of real phylogenetic trees in \cite{me_yule} suggested that, amongst all splits of clades of size $m$,
the median size of the smaller subclade scales roughly as $m^{1/2}$.
The $\beta = -1$ case of our model  has this property, immediately from the definition. 
More broadly, the model does seem to match qualitative features of real large phylogenetic trees.
As mentioned before, 
Figure  \ref{Fig:cladogram} compares a simulation of our $\beta = -1$ model with a real cladogram on 77 species; 
these appear visually similar.  In contrast,  simulations of the familiar alternative models  look
substantially different -- see Figure \ref{Fig:cladogram0} for the Markov model ($\beta = 0$) and the PDA model ($\beta = -1.5$).

\subsubsection{More about the general beta case}
\label{sec:more_beta}
As noted earlier,
the general beta-splitting model is often\footnote{\cite{me_clad} has 334  citations on Google Scholar.}  mentioned in the mathematical biology literature 
on phylogenetics as one of several simple stochastic models.
  See \cite{lambert,steel} for recent overviews of that literature.
Obviously it is biologically unsatisfactory  by not being a forward-in-time model  of extinctions and speciations, and indeed the latter type of model
with age-dependent speciation rates is more plausible and can match the shapes of real trees quite well \cite{age},
though whether one can identify rates uniquely remains a contentious issue \cite{pennell}.
Is the qualitative similarity of these different models just a coincidence, or is there some mathematical connection between the models?

In other words, as stated succinctly in the 2006 survey \cite{blum_francois} of phylogenetic trees
\begin{quote}
Our main result says that the data generally agree with a very simple probabilistic model: [$\beta$-splitting with $\beta = -1$].
However, it leaves us with the issue of providing biological motivation for this.
\end{quote}
That survey suggests an alternative model with the desired ``forwards in time" biological interpretation, and with somewhat 
similar behavior. 
For subsequent work and variations of that {\em BB model} see \cite{sainudiin}.

Some typical uses of the $\beta$-splitting family are
\begin{itemize}
\item 
Comparing statistics of different tree models with data trees: \cite{jones2011}.
\item
Estimating the best-fit value of $\beta$ gives one way
of quantifying the balance of different data trees, and unlike most other balance indices this
allows a direct comparison between different-sized trees: \cite{soewongsono}. 
\item
Simulating trees from different parameters $\beta$ provides one way to see whether different indices of balance are substantially correlated 
or not: \cite{fuchs}. 
\item 
Studying how close (to the root of underlying tree)
one might expect the MRCA (most recent common ancestor)
of a sampled subtree to be. 
\cite{steel2025}.
\item
Methodology for describing tree shape:
\cite{liu_comparing}
\item As a basis for a model with extra parameters:
\cite{distance-metrics,ranked_shapes}.
\end{itemize}

 On the mathematical side, \cite{pitman_mccullagh} have shown that these are the only consistent binary
fragmentation models of a certain ``Gibbs" type. 

 
\subsubsection{On simulations and asymptotics}
The recent paper \cite{xue} studies  the ``balance" of a leaf-labelled binary tree using a statistic
which is equivalent to
\begin{quote}
$C(tree) = $ average over leaves, of number of edges from the root to the leaf.
\end{quote}
In our model this is $L_n$. 
Consider $C(n)$,  the empirical fit to $C(tree)$ for $n$-leaf trees.
 The paper \cite{xue} says, based on existing data
 \begin{quote}
 How do real data scale? Remarkably, it is found that, over
 three orders of magnitude of $n$, there is a power-law scaling $C(n) \sim n^\eta$, with the exponent 
 $\eta = 0.44 \pm 0.01$. 
 \end{quote}
 For a model that reproduces this power law,   \cite{xue} suggests a model 
based on ``niche construction".
 They also have extensive references to previous work.
 This power law of course seems quite different from our model prediction
 (Theorem \ref{TL1}) that
\[
C(n) \sim \tfrac{1}{2\zeta(2)}  \log^2 n.
\]
This perhaps illustrates the dangers of relating asymptotics to pre-asymptotics.
 Figure \ref{Fig:scaling} plots, over a range of $n$ from approximately 20
 to 1,000,
 the exact value  of $\Ex[L_n]$ in our model. 
We see that on the log-log plot it is almost linear in $n$ over this range: our model predicts $C(n) \approx n^{0.34}$ on this range.
This indicates the difficulty in interpreting power-law relationships as justifying any particular model.

 Note that our model, although designed to fit the observed balance of large clades,  also gives explicit predictions for small fringe trees, as discussed 
 in Section \ref{sec:propfringe} below.
 
 
 \begin{center}
 \begin{figure}
 \hspace*{0.7in}
 \includegraphics[width=4.3in]{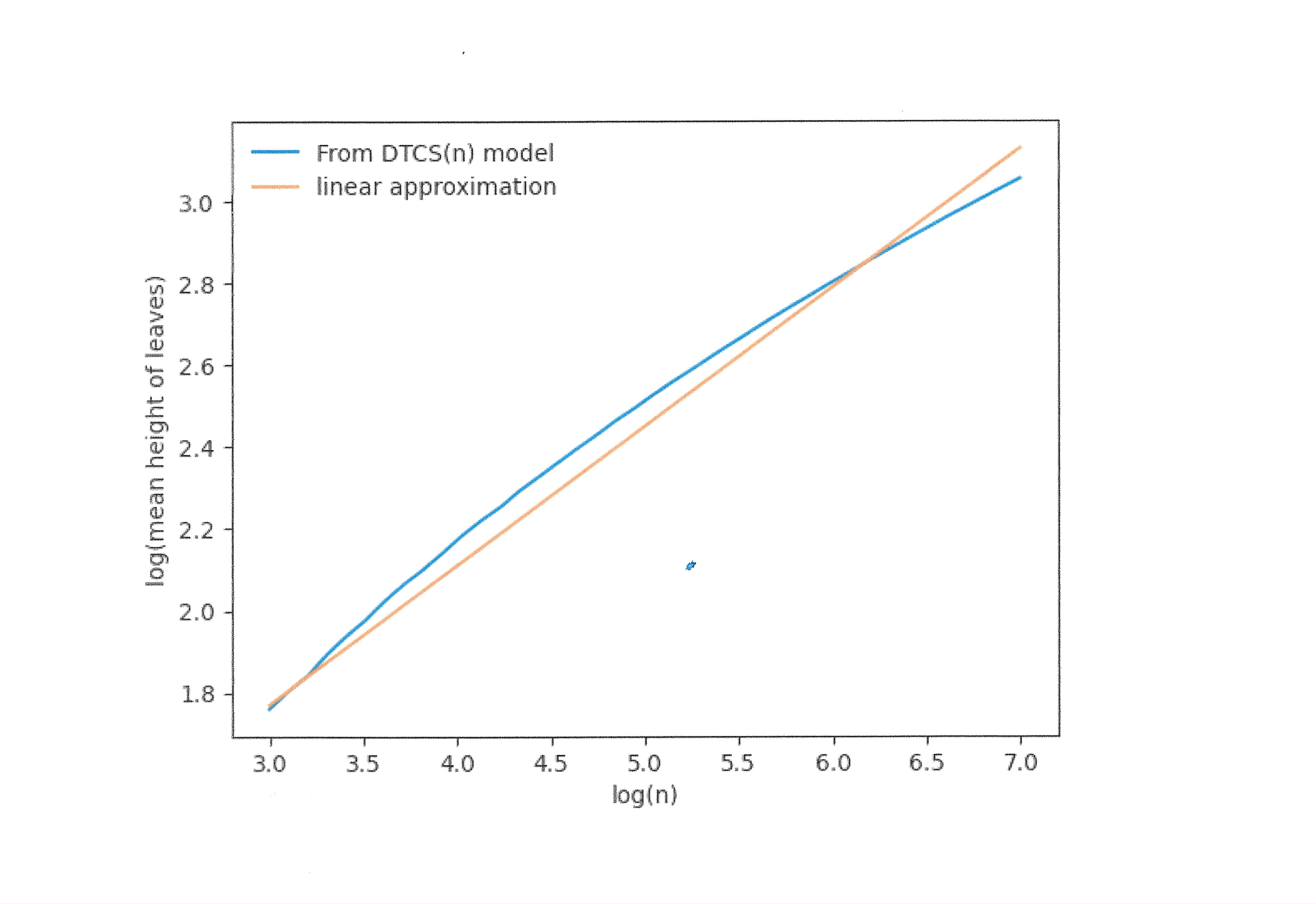}
 \caption{A log-log plot of $\Ex[L_n]$ over $20 \le n \le 1000$ is almost a straight line.}
 \label{Fig:scaling}
 \end{figure}
 \end{center}

\subsection{Properties of the fringe tree}
\label{sec:propfringe}


There are many aspects of the fringe tree that one could study.
One can study it as an interesting process in its own right -- loosely analogous to a stationary process indexed by $\Ints$, 
and in that analogy we could call it the {\em fringe process}
(see Section \ref{sec:fringeterm}).

Recall that, in the fringe tree, the probability that a leaf is in some clade of size $i$ equals $a(i)$.
Because a clade of size $i$ has the $\DTCS(i)$ distribution, we can then calculate the probability 
$p(\chi)$ that a leaf is in a specific clade $\chi$.
Some results are shown in Figure \ref{small_clades}. 
In that figure we have grouped clades with the same {\em shape}, meaning that 
(as in the biology use) we do not distinguish left and right branches. 
Figure \ref{small_clades} compares these model predictions with the data from a small set of real cladograms\footnote{Dragonflies \cite{dragonfly:phy}, eagles \cite{eagles:phy}, elms \cite{elms:phy}, gamebirds \cite{gamebirds:phy}, 
ladybirds \cite{ladybirds:phy}, parrots\cite{parrot:phy}, primates \cite{primates:phy}, sharks \cite{sharks:phy}, snakes \cite{snakes:phy}, swallows \cite{swallows:phy}}
 -- 10 cladograms with a total of 995 species.

These results can be compared with the corresponding results for some other
models of random cladograms in \cite[Appendix A]{SJ385},
see also \cite{Ischebeck}.  Note that the
models treated in \cite{SJ385} are precisely the cases $\beta=\infty, 0, -3/2$
of the beta-splitting tree \cite{me_clad}.

But also one can use the fringe tree to study asymptotics of statistics of $\DTCS(n)$ or $\CTCS(n)$, for statistics which depend only on the structure of the tree near the leaves. 
In particular, the number $N_n(\chi)$ of copies of a size-$i$ clade $\chi$ in $\DTCS(n)$ will satisfy
$n^{-1} \Ex[N_n(\chi)] \to p(\chi)/i$.
By analogy with results for other random tree models -- 
see \cite{fringe} sec.\ 14  -- and because occurrences of a given $\chi$ are only locally dependent,
it should not be difficult to resolve
\begin{OP}
 \label{OP:varNn}
Prove that $n^{-1} \var(N_n(\chi))$ converges to some limit $\sigma^2(\chi)$ and that the corresponding CLT holds.
\end{OP}
Another example is illustrated in the next section.

\begin{figure}[ht]
\setlength{\unitlength}{0.055in}
\begin{picture}(60,80)(0,-64)
\multiput(0,0)(3,0){2}{\circle*{0.9}}
\put(0,0){\line(0,1){4}}
\put(3,0){\line(0,1){4}}
\put(0,4){\line(1,0){3}}
\put(1.5,4){\line(0,1){2}}
\put(-6,-4){$\sfrac{6}{\pi^2} \doteq 0.6079$}
\put(-1,-7){[0.573]}

\multiput(15,0)(3,0){3}{\circle*{0.9}}
\put(15,0){\line(0,1){8}}
\put(18,0){\line(0,1){4}}
\put(21,0){\line(0,1){4}}
\put(18,4){\line(1,0){3}}
\put(19.5,4){\line(0,1){4}}
\put(15,8){\line(1,0){4.5}}
\put(18,8){\line(0,1){2}}
\put(10,-4){$\sfrac{9}{2\pi^2} \doteq 0.4559$}
\put(16,-7){[0.491]}

\multiput(30,0)(3,0){4}{\circle*{0.9}}
\put(30,0){\line(0,1){12}}
\put(33,0){\line(0,1){8}}
\put(36,0){\line(0,1){4}}
\put(39,0){\line(0,1){4}}
\put(36,4){\line(1,0){3}}
\put(37.5,4){\line(0,1){4}}
\put(33,8){\line(1,0){4.5}}
\put(36,8){\line(0,1){4}}
\put(30,12){\line(1,0){6}}
\put(34.5,12){\line(0,1){2}}
\put(26.5,-4){$\sfrac{8}{3\pi^2} \doteq 0.2702$}
\put(32.3,-7){[0.285]}

\multiput(51,0)(3,0){4}{\circle*{0.9}}
\put(51,0){\line(0,1){4}}
\put(54,0){\line(0,1){4}}
\put(51,4){\line(1,0){3}}
\put(52.5,4){\line(0,1){4}}
\put(57,0){\line(0,1){4}}
\put(60,0){\line(0,1){4}}
\put(57,4){\line(1,0){3}}
\put(58.5,4){\line(0,1){4}}
\put(52.5,8){\line(1,0){6}}
\put(55.5,8){\line(0,1){2}}
\put(46.5,-4){$\sfrac{1}{\pi^2} \doteq 0.1013$}
\put(52.5,-7){[0.120]}

\multiput(69,0)(3,0){5}{\circle*{0.9}}
\put(69,0){\line(0,1){16}}
\put(72,0){\line(0,1){12}}
\put(75,0){\line(0,1){8}}
\put(78,0){\line(0,1){4}}
\put(81,0){\line(0,1){4}}
\put(78,4){\line(1,0){3}}
\put(79.5,4){\line(0,1){4}}
\put(75,8){\line(1,0){4.5}}
\put(78,8){\line(0,1){4}}
\put(72,12){\line(1,0){6}}
\put(76.5,12){\line(0,1){4}}
\put(69,16){\line(1,0){7.5}}
\put(75,16){\line(0,1){2}}
\put(66.5,-4){$\sfrac{15}{11\pi^2} \doteq 0.1382$}
\put(72.5,-7){[0.125]}


\multiput(0,-30)(3,0){5}{\circle*{0.9}}
\put(0,-30){\line(0,1){12}}
\put(0,-18){\line(1,0){7.5}}
\put(6,-18){\line(0,1){2}}
\put(3,-30){\line(0,1){4}}
\put(6,-30){\line(0,1){4}}
\put(9,-30){\line(0,1){4}}
\put(12,-30){\line(0,1){4}}
\put(3,-26){\line(1,0){3}}
\put(9,-26){\line(1,0){3}}
\put(4.5,-26){\line(0,1){4}}
\put(10.5,-26){\line(0,1){4}}
\put(4.5,-22){\line(1,0){6}}
\put(7.5,-22){\line(0,1){4}}
\put(-3,-34){$\sfrac{45}{88\pi^2} \doteq 0.0518$}
\put(3,-37){[0.055]}

\multiput(21,-30)(3,0){5}{\circle*{0.9}}
\put(21,-30){\line(0,1){4}}
\put(24,-30){\line(0,1){4}}
\put(21,-26){\line(1,0){3}}
\put(27,-30){\line(0,1){8}}
\put(30,-30){\line(0,1){4}}
\put(33,-30){\line(0,1){4}}
\put(30,-26){\line(1,0){3}}
\put(31.5,-26){\line(0,1){4}}
\put(27,-22){\line(1,0){4.5}}
\put(30,-22){\line(0,1){4}}
\put(22.5,-26){\line(0,1){8}}
\put(22.5,-18){\line(1,0){7.5}}
\put(27,-18){\line(0,1){2}}
\put(18,-34){$\sfrac{5}{4\pi^2} \doteq 0.1267$}
\put(24,-37){[0.145]}

\multiput(42,-30)(3,0){6}{\circle*{0.9}}
\put(42,-30){\line(0,1){20}}
\put(45,-30){\line(0,1){16}}
\put(48,-30){\line(0,1){12}}
\put(51,-30){\line(0,1){8}}
\put(54,-30){\line(0,1){4}}
\put(57,-30){\line(0,1){4}}
\put(54,-26){\line(1,0){3}}
\put(55.5,-26){\line(0,1){4}}
\put(51,-22){\line(1,0){4.5}}
\put(54,-22){\line(0,1){4}}
\put(48,-18){\line(1,0){6}}
\put(52.5,-18){\line(0,1){4}}
\put(45,-14){\line(1,0){7.5}}
\put(51,-14){\line(0,1){4}}
\put(42.5,-10){\line(1,0){9}}
\put(47,-10){\line(0,1){2}}
\put(40,-34){$\sfrac{864}{1375\pi^2} \doteq 0.0637$}
\put(46,-37){[0.030]}

\multiput(66,-30)(3,0){6}{\circle*{0.9}}
\put(66,-30){\line(0,1){16}}
\put(66,-14){\line(1,0){9}}
\put(70.5,-14){\line(0,1){2}}
\put(69,-30){\line(0,1){12}}
\put(69,-18){\line(1,0){7.5}}
\put(75,-18){\line(0,1){4}}
\put(72,-30){\line(0,1){4}}
\put(75,-30){\line(0,1){4}}
\put(78,-30){\line(0,1){4}}
\put(81,-30){\line(0,1){4}}
\put(72,-26){\line(1,0){3}}
\put(78,-26){\line(1,0){3}}
\put(73.5,-26){\line(0,1){4}}
\put(79.5,-26){\line(0,1){4}}
\put(73.5,-22){\line(1,0){6}}
\put(76.5,-22){\line(0,1){4}}
\put(65,-34){$\sfrac{324}{1375\pi^2} \doteq 0.0239$}
\put(71,-37){[0.024]}


\multiput(0,-56)(3,0){6}{\circle*{0.9}}
\put(0,-56){\line(0,1){16}}
\put(0,-40){\line(1,0){9}}
\put(4.5,-40){\line(0,1){2}}
\put(3,-56){\line(0,1){4}}
\put(6,-56){\line(0,1){4}}
\put(3,-52){\line(1,0){3}}
\put(9,-56){\line(0,1){8}}
\put(12,-56){\line(0,1){4}}
\put(15,-56){\line(0,1){4}}
\put(12,-52){\line(1,0){3}}
\put(13.5,-52){\line(0,1){4}}
\put(9,-48){\line(1,0){4.5}}
\put(12,-48){\line(0,1){4}}
\put(4.5,-52){\line(0,1){8}}
\put(4.5,-44){\line(1,0){7.5}}
\put(9,-44){\line(0,1){4}}
\put(-1.5,-60){$\sfrac{72}{125\pi^2} \doteq 0.0584$}
\put(4.5,-63){[0.042]}

\multiput(24,-56)(3,0){6}{\circle*{0.9}}
\put(24,-56){\line(0,1){4}}
\put(27,-56){\line(0,1){4}}
\put(24,-52){\line(1,0){3}}
\put(25.5,-52){\line(0,1){12}}
\put(25.5,-40){\line(1,0){9}}
\put(30,-40){\line(0,1){2}}
\put(30,-56){\line(0,1){12}}
\put(33,-56){\line(0,1){8}}
\put(36,-56){\line(0,1){4}}
\put(39,-56){\line(0,1){4}}
\put(36,-52){\line(1,0){3}}
\put(37.5,-52){\line(0,1){4}}
\put(33,-48){\line(1,0){4.5}}
\put(36,-48){\line(0,1){4}}
\put(30,-44){\line(1,0){6}}
\put(34.5,-44){\line(0,1){4}}
\put(22.5,-60){$\sfrac{36}{55\pi^2} \doteq 0.0663$}
\put(28.5,-63){[0.054]}

\multiput(48,-56)(3,0){6}{\circle*{0.9}}
\put(48,-56){\line(0,1){4}}
\put(51,-56){\line(0,1){4}}
\put(48,-52){\line(1,0){3}}
\put(49.5,-52){\line(0,1){8}}
\put(49.5,-44){\line(1,0){9}}
\put(54,-44){\line(0,1){2}}
\put(54,-56){\line(0,1){4}}
\put(57,-56){\line(0,1){4}}
\put(54,-52){\line(1,0){3}}
\put(55.5,-52){\line(0,1){4}}
\put(60,-56){\line(0,1){4}}
\put(63,-56){\line(0,1){4}}
\put(60,-52){\line(1,0){3}}
\put(61.5,-52){\line(0,1){4}}
\put(55.5,-48){\line(1,0){6}}
\put(58.5,-48){\line(0,1){4}}
\put(46.5,-60){$\sfrac{27}{110\pi^2} \doteq 0.0249$}
\put(52.5,-63){[0.030]}

\multiput(72,-56)(3,0){6}{\circle*{0.9}}
\put(72,-56){\line(0,1){8}}
\put(75,-56){\line(0,1){4}}
\put(78,-56){\line(0,1){4}}
\put(75,-52){\line(1,0){3}}
\put(76.5,-52){\line(0,1){4}}
\put(72,-48){\line(1,0){4.5}}
\put(75,-48){\line(0,1){4}}
\put(81,-56){\line(0,1){8}}
\put(84,-56){\line(0,1){4}}
\put(87,-56){\line(0,1){4}}
\put(84,-52){\line(1,0){3}}
\put(85.5,-52){\line(0,1){4}}
\put(81,-48){\line(1,0){4.5}}
\put(84,-48){\line(0,1){4}}
\put(75,-44){\line(1,0){9}}
\put(79.5,-44){\line(0,1){2}}
\put(70.5,-60){$\sfrac{2}{5\pi^2} \doteq 0.0405$}
\put(76.5,-63){[0.024]}

 \end{picture}
\caption{Proportions of leaves in clades of a given shape, for each shape with $2 - 6$ leaves in the fringe tree.
The top number is from our model, the bottom number $[ \cdots ]$ from our small data set.
}
\label{small_clades}
\end{figure}
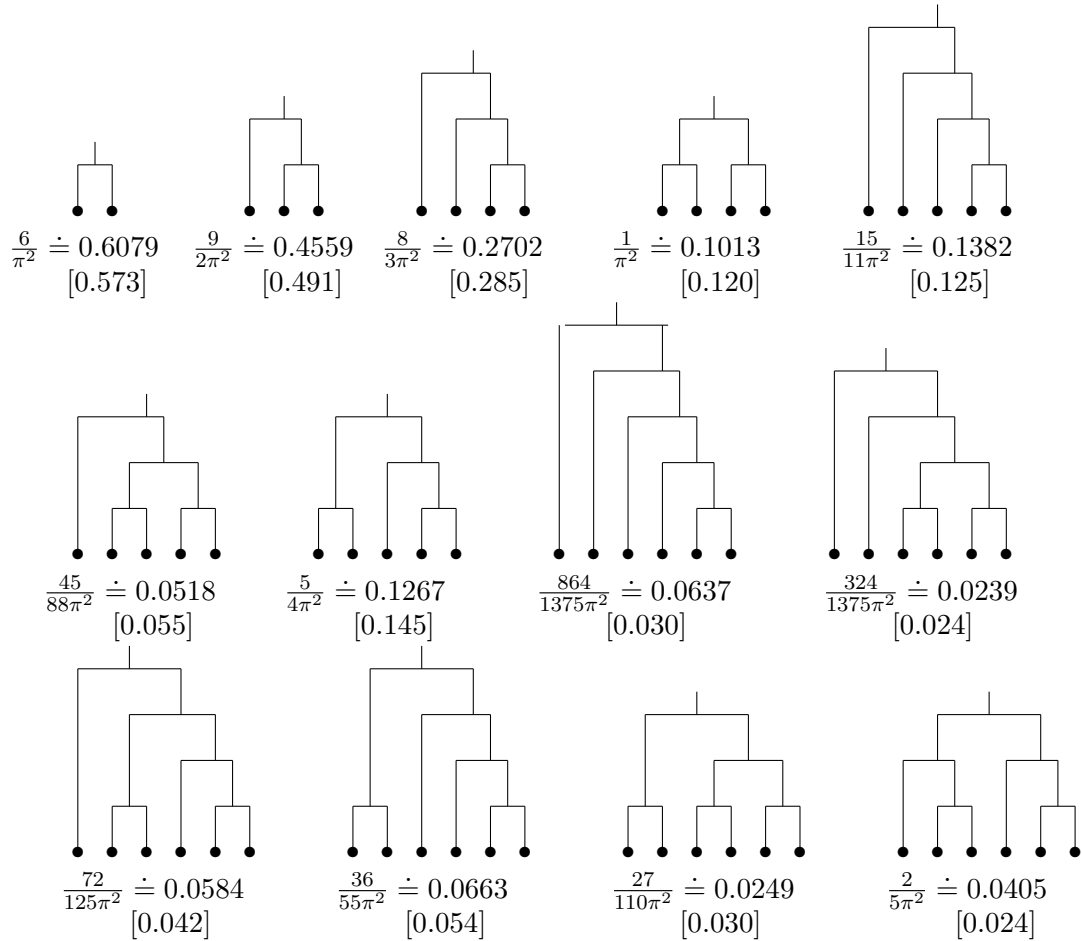

 \subsection{The length of $\CTCS(n)$}
\label{sec:Clength1}


The number of edges of $\CTCS(n)$ equals $n-1$.
Identifying {\em length} of an edge with {\em duration of time}, one can consider the length $\Lambda_n$ of 
$\CTCS(n)$, that is the sum of edge-lengths.
The expectation of the number of size-$i$ clades 
in $\CTCS(n)$ equals $\sfrac{n}{i} a(n,i)$, so we immediately have
\begin{equation}
\Ex [\Lambda_n] =  n \sum_{i=2}^n \frac{ a(n,i)}{i h_{i-1}} .
\label{ELn}
\end{equation}
Because $\lim_n a(n,i) = a(i)$ and
\begin{equation}
\sum_{i=2}^\infty \frac{ a(i)}{i h_{i-1}} = \frac{6}{\pi^2} \sum_{i=2}^\infty \frac{1}{i(i-1)} = \frac{6}{\pi^2} 
\label{iii}
\end{equation}
we naturally expect
\begin{Proposition}
\label{P:length}
$\lim_n n^{-1}  \Ex [\Lambda_n] = \frac{6}{\pi^2}$.
\end{Proposition}
This is proved in \cite[Theorem 9.1]{beta4-arxiv} as a consequence of the Mellin transform analysis.
An alternative ``probabilistic" proof is given in Appendix \ref{sec:Clength2}.

The fact that the limit equals $a(2)$ has an intriguing consequence -- see Appendix \ref{sec:coincide}.
As with the subtree counts, mentioned in Section \ref{sec:propfringe},
 it should not be difficult to resolve
\begin{OP}
\label{OP:varL}
Prove that $n^{-1} \var(\Lambda_n)$ converges to some limit $\sigma^2$ and that the corresponding CLT holds.
\end{OP}

\subsection{Combinatorial questions}
\label{sec:combin}


There are a range of what one might call ``combinatorial" questions related to the fringe tree.
What is the probability that two independent copies of $\DTCS(n)$ have the same shape?
Numerics for $n \le 200$ suggest a probability $\asymp 0.44^n$, but as in the discussion around Figure \ref{Fig:scaling} we have little confidence that this is the correct asymptotics.

Regarding  the number $N_n(\chi)$ of copies of a clade $\chi$ in $\DTCS(n)$,
one could study distributions of the following:
\begin{OP} 
\label{OP:Nnchi}. 
\mbox{ } 
\begin{itemize}
 \item The number $K_n := \sum_\chi 1_{(N_n(\chi) \ge 1)}$ of different-shape clades within (a realization of) $\DTCS(n)$.
\item The largest clade that appears more than once within $\DTCS(n)$.
\item The smallest clade that does not appear within $\DTCS(n)$.
  \end{itemize}
  \end{OP}

There are two issues here.
  First, the numerical values of $p(\chi)$ in Figure \ref{small_clades} are calculated recursively: 
  we don’t have a good intrinsic description of the set of probabilities
  $\{p(\chi) : |\chi| = m\}$ over all $m$-leaf trees.
  
  The second issue is more pervasive. 
  Theorem \ref{T:alimit} is a pointwise result: for fixed $m$ we have $a(n,m) \to a(m)$ as $n \to \infty$.
  In considering limits of statistics of $\DTCS(n)$, to go beyond the limit numbers of copies of a fixed fringe tree
  (Figure \ref{small_clades}), one needs sharper results, of the following type.

 \subsection{A technical obstacle}
\label{sec:obstacle}


As noted in the discussion in [\cite{beta1} Theorem 2.16], numerical evidence strongly supports
the conjecture that
$n \to a(n,i)$ is decreasing, for each $i \ge 2$.
This might be helpful for the following problem.

\begin{OP}
\label{OP:an}
Find explicit bounds for $|a(n,m) - a(m)|$.
In particular, prove the following Ansatz.
\end{OP}

\begin{Ansatz}
\label{Ansatz}
For a non-negative sequence $(f(j), j \ge 2)$ such that $f(j) = O(j^k)$ for some $k < \infty$:

\noindent
(i) If\/ $\sum_{i = 2}^\infty a(i)f(i) < \infty$ then $\sum_{i = 2}^n a(n,i)f(i) \to \sum_{i = 2}^\infty a(i)f(i)$.

\noindent
(ii) If\/ $\sum_{i = 2}^\infty a(i)f(i) = \infty$ then $\sum_{i = 2}^n a(n,i)f(i) \sim \sum_{i = 2}^n a(i)f(i)$.
\end{Ansatz}

In fact we have implicitly proved a number of special cases by ad hoc methods, for example 
the special case of the length of $\CTCS$ (section \ref{sec:Clength1}), where $f(j) = 1/(j h_{j-1})$.
Lack of a general proof is an annoying obstacle to further general rigorous progress. 
For instance, using the ansatz one can readily obtain the asymptotics of various {\em balance indices} (section \ref{sec:balance_ex}).
And one would want to use the ansatz also for the ``combinatorial problems"  in section \ref{sec:combin} and  the  ``powers of subtree sizes" problems in section \ref{sec:powers}.

\subsection{Tree balance indices}
\label{sec:tree_balance}


A longstanding topic in mathematical and statistical phylogenetics concerns  {\em tree balance indices}, that is statistics which measure ``balance" of a given tree
in some quantitative way. 
An authoritative comprehensive study of this topic  is given in the  monograph \cite{fischer2023tree},  
which (amongst much other material) records what is known about the mean and variance of each index under the two classical probability models, Yule and uniform.
It is natural to ask about the distribution of indices under our 
model.\footnote{These balance indices refer to cladograms, so we are concerned with the discrete-time form of our model.}.
Tables 1 and 2 of  \cite{fischer2023tree} list 19 examples of balance indices, 
 and our Table \ref{table:balance} lists those addressed in this article. 
 For an index $\mathbb{I}$ we write $\mathbb{I}(n)$ for the random variable obtained when $\mathbb{I}$ is applied to a realization of $\DTCS(n)$.
 We then seek to study the distribution of $\mathbb{I}(n)$.
 
 For this purpose, different aspects of our mathematical results can be helpful.
 In particular, the results surrounding the HD chain and the fringe distribution.
 Recall
 \begin{itemize}
 \item The mean number of subclades of size $m$ in $\DTCS(n)$ equals $\frac{n}{m} a(n,m)$.
 \item $a(n,m) \to a(m)$ as $n \to \infty$ for fixed $m \ge 1$.
 \item $a(m) = \frac{6h_{m-1}}{\pi^2(m-1)}$ for $m \ge 2$.
 \end{itemize}
 Some indices are determined in some ``additive" way based on the set of  splits $m \to (i,m-i)$,
 and in some such cases 
 one can immediately write down an expression for $\Ex[ \mathbb{I}(n)]$ in terms of the $a(n,m)$.
Then we seek to deduce the asymptotics of $\Ex[ \mathbb{I}(n)]$ from the limits  $a(m)$.
At this point we need to invoke Ansatz \ref{Ansatz}.
Implicitly this method can work only for indices that are not sensitive to the near-root structure, which fortunately is true for most indices. 
The same method for ``additive" indices has been used for the classical random tree models,
 but there one has more explicit formulas for the analog of our $a(n,i)$.
 
 In this article we will only give heuristics --  back-of-an-envelope calculations for the  asymptotics of  $\Ex[ \mathbb{I}(n)]$ -- in the next section.
 (Some use different methods, such as analysis of a recursion).
Of course  one would like to move beyond expectation, to study variance and limit distributions. 
This is another key open problem, with substantial scope for future systematic study.
The following open problem continues the themes of Open Problems \ref{OP:varNn} and  \ref{OP:varL}.

\begin{OP}
\label{OP:tree_bal} 
Write $N^{(n)}_m$ for the number of size-$m$ clades in $\DTCS(n)$.
Study the joint distribution of $(N^{(n)}_m, 2 \le m \le n)$ in such a way that one can calculate covariances and deduce CLTs.
\end{OP}
As noted in section \ref{sec:propfringe}, the local weak limit (from a random leaf) of $\DTCS(n)$ is a kind of stationary tree process on leaves $ \ldots, -2, -1, 0, 1, 2, \ldots$,
so one might start by proving CLTs within that structure, before seeking to transfer them to the asymptotics of $\DTCS(n)$. 
The large literature on CLTs for under mixing conditions is recounted in detail in \cite{bradley}.

\begin{table}
\caption{Some balance indices}
\centering
\begin{tabular}{lll}
\hline 
\hline
Name &       Notation in   \cite{fischer2023tree} &    our section \\
\hline 
Average leaf height & $\bar{N}(n)$  & \ref{sec:DTsetting}: our $L_n$\\
Colless index & $C(n)$ & \ref{sec:colless} \\
Quadratic Colless index & $QC(n)$& \ref{sec:colless2} \\
$B_1$ index & $B_1(n)$ & \ref{sec:B1} \\
$B_2$ index & $B_2(n)$ & \ref{sec:balance}  \\
$\hat{s}$-shape & $\hat{s}(n)$ & \ref{sec:shat} \\
Total cophenetic index& $\Phi(n)$ & \ref{sec:coph} \\
Variance of leaf heights & $\sigma^2_N(n)$ &\ref{sec:vard} \\
Rooted quartet index & $rQI(n)$ & \ref{sec:rQI} \\
\end{tabular}
\label{table:balance} 
\end{table}

\subsection{Examples and heuristics for balance indices}
\label{sec:balance_ex}


\subsubsection{Colless index}
\label{sec:colless}
The Colless index $C$ is the sum over all splits $m \to (i,m-i)$ of the size difference $| i - (m-i)|$.
So
\[ \Ex[ C(n)] = \sum_{m=3}^{n}  \frac{n}{m} a(n,m) c(m) \]
\[ c(m) = \sum_{i-1}^{m-1} q(m,i) | i - (m-i)|  .
\]
A brief calculation shows $c(m)  \sim m/2$.
Using the Ansatz (here with $f(j) \to 1/2$)
 we can approximate $a(n,m)$ as $a_m = \frac{6h_{m-1}}{\pi^2(m-1)}$, and find
\[ \Ex[ C(n)] \sim \frac{3n}{\pi^2} \sum_{m=3}^{n}  \frac{h_{m-1}}{m-1}  \sim \frac{3 n \log^2 n}{2\pi^2} .
\]

\subsubsection{Quadratic Colless index} 
\label{sec:colless2}
The Quadratic Colless index $C$ is the sum over all splits $m \to (i,m-i)$ of the squared size difference $(i - (m-i))^2$.
So

\[ \Ex[ QC(n)] = \sum_{m=3}^{n}  \frac{n}{m} a(n,m) g(m) \]
\[ g(m) = \sum_{i-1}^{m-1} q(m,i) (i - (m-i))^2. 
\]
A brief calculation shows $g(m)  \sim m^2/4$.
Again using the Ansatz (here with $f(j) \sim j/4$) to approximate  $a(n,m)$ as $a_m = \frac{6h_{m-1}}{\pi^2(m-1)}$, we find
\[ \Ex[ QC(n)] \sim \frac{3n}{2\pi^2} \sum_{m=3}^{n} h_{m-1} \sim  \frac{3n^2 \log n}{2\pi^2} .
\]

\subsubsection{The $B_1$ index}
\label{sec:B1}
The $B_1$ index is the  sum of the reciprocal of the heights of the subclades.  So
\[\Ex[B_1(n)] = \sum_{m=2}^{n}  \frac{n}{m} a(n,m) u(m) \]
\[ u(m) = \Ex[1/(\mathrm{height \ of } \DTCS(m)) ]. \]
We have not studied $u(m)$, but because $u(m) \le 1$ the sum $\sum_m
\frac{1}{m} a(m) u(m)$ is convergent and
so we expect\SJ
\[ \lim_n n^{-1} [\Ex[B_1(n)] = \sum_{m=2}^{\infty}  \frac{a(m)}{m} u(m) . \]

\subsubsection{The $\hat{s}$-shape index}
\label{sec:shat}
The $\hat{s}$-shape index is the
sum of $\log(m-1)$  over all splits $m\to (i,m-i)$.
So
\[ \Ex [\hat{s}(n)] = \sum_{m=3}^{n}  \frac{n}{m} a(n,m) \log(m-1) . \]
Again the sum is convergent, so we expect
\[ \lim_n  n^{-1}  \Ex [\hat{s}(n)] = \sum_{m=3}^{n}  \frac{1}{m} a(m) \log(m-1) < \infty .
\]

\subsubsection{Total cophenetic index}
\label{sec:coph} 
The ``cophenetic values" of a pair of leaves is the discrete height of their branchpoint.
The {\em total cophenetic index} 
$\Phi$ is the sum of the cophenetic values over all different pairs of leaves.
So
\[ \Ex[\Phi(n)] = {n \choose 2} \times \Ex[D_{n,2}]  \]
for $D_{n,2}$ as in Proposition \ref{P:roots} later.
There it is shown that $\Ex [D_{n,2}] \sim \log n$, and so
\[ \Ex[\Phi(n)] \sim \sfrac{1}{2} n^2 \log n .
\]

\subsubsection{Variance of leaf heights}
\label{sec:vard}

Our Theorem \ref{T:varL} (\cite{beta1} Theorem 1.2) shows that the discrete leaf height $L_n$
has unconditional variance 
$\sim \tfrac{2\zeta(3)}{3\zeta^3(2)}\log^3 n$.
In the present context we are concerned with the conditional variance given the realization of the tree.
However the ``asymptotic uncorrelation" result 
\eqref{rhon} for $\alpha = 1$ \cite[Theorem 2.6]{beta1}
implies that the 
expectation of the conditional variance has the same first-order asymptotics.

\subsubsection{Rooted quartet index}
\label{sec:rQI}
A version of the rooted quartet index $rQI$  counts the number of 4-leaf-sets whose induced subtree is the completely-balanced 
tree $\chi$ of size 4.
By the consistency property, this is just ${n \choose 4}$ times the probability $q(4,2) = 3/11$ that $\DTCS(4)$ is completely balanced:
\[ \Ex [rQI(n)] = \frac{3}{11} {n \choose 4} .\]


\begin{OP}
\label{OP:indices}
Study the distribution of these and other indices for $\DTCS(n)$ in more detail.
\end{OP}

Because our data studies of splitting and of the fringe distributions  (Figure \ref{small_clades})
were small-scale:
\begin{OP}
\label{OP:more_data}
Repeat these data studies on a larger scale and for other indices.
\end{OP}

Though we do not  expect the model to provide quantitatively accurate matches to real data, 
the point is that  more elaborate biologically-motivated models of the kind described in  \cite{lambert,steel} 
typically have real-valued parameters fitted to the individual tree data;
how much do they improve on our zero-parameter model?
In this context our asymptotics are irrelevant -- one can just simulate $\DTCS(n)$ numerically.

\subsection{Notes on fringe terminology}
\label{sec:fringeterm}
In the context of a large tree, the word {\em fringe} is informally used to mean the part of the tree near the leaves, rather than near the root. 
In this article, the specific {\em fringe distribution} is defined as the local weak limit of $\DTCS(n)$ relative to a uniform random leaf.  
 See \cite{me-fringe,bhamidi2025,fringe} for general accounts of local weak convergence.
The fringe distribution is
formally a probability distribution on the space of binary trees with a countable number of leaves (one leaf a distinguished ``root") and with a unique ``end", that is an asymptotic path to infinity.
This is essentially a special case of the notion of local weak convergence 
for sparse random graphs, modified because there one uses a uniform random vertex.\footnote{There is a straightforward connection between these two conventions.}
Such limits of graphs are examples of {\em unimodular} graphs, 
so our fringe distribution is a small modification of a unimodular tree.
However, in our model the left-right ordering of leaves provides extra structure, and 
we can label the leaves as $\{\ldots, -2, -1, 0 , 1, 2,  \ldots \}$
with the root as $0$.

One should visualize a realization of such a tree as in Figure \ref{Fig:cladogram}, but with leaves labelled as 
$\ldots, -2, -1, -0, 1, 2, \ldots$,
and with branches randomly positioned left/right instead of the biology convention 
of usually positioning the larger clade to the right.  
So a {\em fringe tree} $\TT$ is a 
random infinite tree whose distribution is the specific fringe distribution.
By re-rooting a realization of the fringe tree  at leaf $i$, 
we get a stationary random process $(\TT_i, - \infty < i < \infty)$ taking values in the tree-space.
Viewed this way one could call it a {\em fringe  process}.
Arguing as in \cite{me-fringe} one could show that this fringe process is ergodic, thereby 
obtaining a.s. limit theorems for averages of suitable functionals $m^{-1} \sum_{i = 0}^{m-1} \Phi(\TT_i) $
from the classical ergodic theorem.
The balance indices  in section \ref{sec:balance_ex} would be examples of such functionals.

For this article, we are interested in the fringe distribution in the context of
limits of finite $\DTCS(n)$ trees.
One can regard the ``fringe distribution" as determined by the collection of distributions over all finite trees, as 
indicated in Figure \ref{small_clades}.
By definition, the fringe distribution gives first order limits, in the sense that, 
writing $(\TT^{(n)}_i, 1 \le i \le n)$ for a realization of $\DTCS(n)$ rooted at each leaf $i$,
we have 
\[ \Ex[ n^{-1} \sum_{i = 1}^{n} \Phi(\TT^{(n)}_i) ] \to \Ex[ \Phi(\TT)] \]
for suitable functionals $\Phi$.
But can we get second-order limits also?
\begin{OP}
\label{OP:secondorder}
If we know that, for a given functional $\Phi$, a CLT holds for the fringe process
$ \sum_{i = 0}^{m-1} \Phi(\TT_i) $, does the same CLT necessarily also hold for $\DTCS(n)$, that is for
$\sum_{i = 1}^{n} \Phi(\TT^{(n)}_i)$? 
\end{OP}

 \section{The consistency property and the exchangeability representation}
\label{sec:consistent-exch}
\subsection{The consistency property}\label{sec:consistency}


The interval-splitting construction of $\CTCS(n)$ does  implicitly assign leaf-labels $\{1,2,\ldots,n\}$ 
but conceptually we are thinking of recursively splitting a set of objects which have labels but 
without any prior structure on the label-set. 
As mentioned in section \ref{sec:unordered},
 it is convenient to re-define $\CTCS(n)$ by applying a uniform random permutation to these leaf-labels.\footnote{This of course yields a certain type of (finite) exchangeability, suggesting a limit structure involving infinite exchangeability, described in Section \ref{sec:paintbox}.}
This does not affect earlier results, except for the ``correlation" feature in Section \ref{sec:depends}.
So our ``path to a uniform random leaf" is equivalent (in distribution) to ``path to leaf $1$".
And ``delete a uniform random leaf" is equivalent to ``delete leaf $n$".
Now we can define a ``delete a leaf, and prune" operation, illustrated in Figure \ref{Fig:dpo}.

 Note that the length of horizontal edges in the figure has no significance; these edges serve only to indicate 
 which are the left and right branches.

\begin{figure}
\setlength{\unitlength}{0.033in}
\begin{picture}(60,94)(0,-5)

\put(3,3){\line(1,0){27}}
\put(28.5,3){\line(0,-1){3}}

\put(1.5,14){\line(1,0){4.5}}
\put(3,3){\line(0,1){11}}
\put(6,14){\circle*{0.9}}

\put(0,27){\line(1,0){3}}
\put(1.5,14){\line(0,1){13}}
\put(0,27){\circle*{0.9}}
\put(3,27){\circle*{0.9}}

\put(30,3){\line(0,1){3}}
\put(9,6){\line(1,0){25.5}}
\put(9,6){\circle*{0.9}}

\put(34.5,6){\line(0,1){5}}
\put(16.5,11){\line(1,0){24}}
\put(16.5,11){\line(0,1){9}}

\put(16.5,20){\line(1,0){4.5}}
\put(21,20){\circle*{0.9}}
\put(16.5,20){\line(-1,0){1.5}}

\put(15,20){\line(0,1){5}}
\put(15,25){\line(-1,0){1.5}}
\put(15,25){\line(1,0){3}}
\put(18,25){\circle*{0.9}}
\put(13.5,25){\circle*{0.9}}


\put(40.5,11){\line(0,1){5}}
\put(40.5,16){\line(1,0){16.5}}
\put(57,16){\circle*{0.9}}
\put(40.5,16){\line(-1,0){1.5}}

\put(39,16){\line(0,1){5}}
\put(39,21){\line(1,0){13.5}}
\put(39,21){\line(-1,0){3}}

\put(36,21){\line(0,1){5}}
\put(36,26){\line(1,0){7.5}}
\put(36,26){\line(-1,0){6}}

\put(30,26){\line(0,1){7}}
\put(30,33){\line(1,0){4.5}}
\put(30,33){\line(-1,0){3}}

\put(27,33){\line(0,1){7}}
\put(27,40){\line(1,0){3}}
\put(27,40){\line(-1,0){1.5}}
\put(30,40){\circle*{0.9}}

\put(25.5,40){\line(0,1){15}}
\put(24,55){\line(1,0){3}}
\multiput(24,55)(3,0){2}{\circle*{0.9}}

\put(34.5,33){\line(0,1){11}}
\put(33,44){\line(1,0){3}}
\multiput(33,44)(3,0){2}{\circle*{0.9}}

\put(43.5,26){\line(0,1){9}}
\put(42,35){\line(1,0){3}}
\put(45,35){\line(0,1){5}}

\put(42,40){\line(1,0){4.5}}
\put(42,40){\circle*{0.9}}
\put(46.5,40){\line(0,1){9}}
\multiput(45,49)(3,0){2}{\circle*{0.9}}
\put(45,49){\line(1,0){3}}

\put(52.5,21){\line(0,1){12}}
\multiput(51,33)(3,0){2}{\circle*{0.9}}
\put(51,33){\line(1,0){3}}

\put(39,35){\circle*{0.9}}
\put(42,35){\line(-1,0){3}}

\put(26,-3){leaf $b$ deleted}


\put(68,3){\line(1,0){27}}
\put(93.5,3){\line(0,-1){3}}

\put(66.5,14){\line(1,0){4.5}}
\put(68,3){\line(0,1){11}}
\put(71,14){\circle*{0.9}}

\put(65,27){\line(1,0){3}}
\put(66.5,14){\line(0,1){13}}
\put(65,27){\circle*{0.9}}
\put(68,27){\circle*{0.9}}

\put(95,3){\line(0,1){3}}
\put(74,6){\line(1,0){25.5}}
\put(74,6){\circle*{0.9}}

\put(99.5,6){\line(0,1){5}}
\put(81.5,11){\line(1,0){24}}
\put(81.5,11){\line(0,1){9}}

\put(81.5,20){\line(1,0){4.5}}
\put(86,20){\circle*{0.9}}
\put(81.5,20){\line(-1,0){1.5}}

\put(80,20){\line(0,1){5}}
\put(80,25){\line(-1,0){1.5}}
\put(80,25){\line(1,0){3}}
\put(83,25){\circle*{0.9}}

\put(78.5,25){\line(0,1){15}}
\put(77,40){\line(1,0){3}}
\put(77,40){\circle*{0.9}}
\put(80,40){\circle*{0.9}}

\put(105.5,11){\line(0,1){5}}
\put(105.5,16){\line(1,0){16.5}}
\put(122,16){\circle*{0.9}}
\put(105.5,16){\line(-1,0){1.5}}

\put(104,16){\line(0,1){5}}
\put(104,21){\line(1,0){13.5}}
\put(104,21){\line(-1,0){3}}

\put(101,21){\line(0,1){5}}
\put(101,26){\line(1,0){7.5}}
\put(101,26){\line(-1,0){6}}

\put(95,26){\line(0,1){7}}
\put(95,33){\line(1,0){4.5}}
\put(95,33){\line(-1,0){3}}

\put(92,33){\line(0,1){7}}
\put(92,40){\line(1,0){3}}
\put(92,40){\line(-1,0){1.5}}
\put(95,40){\circle*{0.9}}

\put(90.5,40){\line(0,1){15}}
\put(89,55){\line(1,0){3}}
\multiput(89,55)(3,0){2}{\circle*{0.9}}

\put(99.5,33){\line(0,1){11}}
\put(98,44){\line(1,0){3}}
\multiput(98,44)(3,0){2}{\circle*{0.9}}

\put(108.5,26){\line(0,1){9}}
\put(107,35){\line(1,0){3}}
\put(110,35){\line(0,1){5}}

\put(107,40){\line(1,0){4.5}}
\put(107,40){\circle*{0.9}}
\put(111.5,40){\line(0,1){9}}
\multiput(110,49)(3,0){2}{\circle*{0.9}}
\put(110,49){\line(1,0){3}}

\put(117.5,21){\circle*{0.9}}

\put(104,35){\circle*{0.9}}
\put(107,35){\line(-1,0){3}}

\put(91,-3){leaf $c$ deleted}


\put(3,73){\line(1,0){27}}
\put(28.5,73){\line(0,-1){3}}

\put(1.5,84){\line(1,0){4.5}}
\put(3,73){\line(0,1){11}}
\put(6,84){\circle*{0.9}}

\put(0,97){\line(1,0){3}}
\put(1.5,84){\line(0,1){13}}
\put(0,97){\circle*{0.9}}
\put(3,97){\circle*{0.9}}

\put(30,73){\line(0,1){3}}
\put(9,76){\line(1,0){25.5}}
\put(9,76){\circle*{0.9}}

\put(34.5,76){\line(0,1){5}}
\put(16.5,81){\line(1,0){24}}
\put(16.5,81){\line(0,1){9}}

\put(16.5,90){\line(1,0){4.5}}
\put(21,90){\circle*{0.9}}
\put(16.5,90){\line(-1,0){1.5}}

\put(15,90){\line(0,1){5}}
\put(15,95){\line(-1,0){1.5}}
\put(15,95){\line(1,0){3}}
\put(18,95){\circle*{0.9}}

\put(13.5,95){\line(0,1){15}}
\put(12,110){\line(1,0){3}}
\put(12,110){\circle*{0.9}}
\put(15,110){\circle*{0.9}}
\put(11.0,110.9){b}

\put(40.5,81){\line(0,1){5}}
\put(40.5,86){\line(1,0){16.5}}
\put(57,86){\circle*{0.9}}
\put(40.5,86){\line(-1,0){1.5}}

\put(39,86){\line(0,1){5}}
\put(39,91){\line(1,0){13.5}}
\put(39,91){\line(-1,0){3}}

\put(36,91){\line(0,1){5}}
\put(36,96){\line(1,0){7.5}}
\put(36,96){\line(-1,0){6}}

\put(30,96){\line(0,1){7}}
\put(30,103){\line(1,0){4.5}}
\put(30,103){\line(-1,0){3}}

\put(27,103){\line(0,1){7}}
\put(27,110){\line(1,0){3}}
\put(27,110){\line(-1,0){1.5}}
\put(30,110){\circle*{0.9}}
\put(29.2,110.8){a}

\put(25.5,110){\line(0,1){15}}
\put(24,125){\line(1,0){3}}
\multiput(24,125)(3,0){2}{\circle*{0.9}}

\put(34.5,103){\line(0,1){11}}
\put(33,114){\line(1,0){3}}
\multiput(33,114)(3,0){2}{\circle*{0.9}}

\put(43.5,96){\line(0,1){9}}
\put(42,105){\line(1,0){3}}
\put(45,105){\line(0,1){5}}

\put(42,110){\line(1,0){4.5}}
\put(42,110){\circle*{0.9}}
\put(46.5,110){\line(0,1){9}}
\multiput(45,119)(3,0){2}{\circle*{0.9}}
\put(45,119){\line(1,0){3}}

\put(52.5,91){\line(0,1){12}}
\multiput(51,103)(3,0){2}{\circle*{0.9}}
\put(51,103){\line(1,0){3}}
\put(53.5,103.8){c}

\put(39,105){\circle*{0.9}}
\put(42,105){\line(-1,0){3}}

\put(26,67){root}


\put(68,73){\line(1,0){27}}
\put(93.5,73){\line(0,-1){3}}

\put(66.5,84){\line(1,0){4.5}}
\put(68,73){\line(0,1){11}}
\put(71,84){\circle*{0.9}}

\put(65,97){\line(1,0){3}}
\put(66.5,84){\line(0,1){13}}
\put(65,97){\circle*{0.9}}
\put(68,97){\circle*{0.9}}

\put(95,73){\line(0,1){3}}
\put(74,76){\line(1,0){25.5}}
\put(74,76){\circle*{0.9}}

\put(99.5,76){\line(0,1){5}}
\put(81.5,81){\line(1,0){24}}
\put(81.5,81){\line(0,1){9}}

\put(81.5,90){\line(1,0){4.5}}
\put(86,90){\circle*{0.9}}
\put(81.5,90){\line(-1,0){1.5}}

\put(80,90){\line(0,1){5}}
\put(80,95){\line(-1,0){1.5}}
\put(80,95){\line(1,0){3}}
\put(83,95){\circle*{0.9}}

\put(78.5,95){\line(0,1){15}}
\put(77,110){\line(1,0){3}}
\put(77,110){\circle*{0.9}}
\put(80,110){\circle*{0.9}}

\put(105.5,81){\line(0,1){5}}
\put(105.5,86){\line(1,0){16.5}}
\put(122,86){\circle*{0.9}}
\put(105.5,86){\line(-1,0){1.5}}

\put(104,86){\line(0,1){5}}
\put(104,91){\line(1,0){13.5}}
\put(104,91){\line(-1,0){3}}

\put(101,91){\line(0,1){5}}
\put(101,96){\line(1,0){7.5}}
\put(101,96){\line(-1,0){6}}

\put(95,96){\line(0,1){7}}
\put(95,103){\line(1,0){4.5}}
\put(95,103){\line(-1,0){3}}

\put(92,103){\line(0,1){7}}

\put(92,110){\line(0,1){15}}
\put(90.5,125){\line(1,0){3}}
\multiput(90.5,125)(3,0){2}{\circle*{0.9}}

\put(99.5,103){\line(0,1){11}}
\put(98,114){\line(1,0){3}}
\multiput(98,114)(3,0){2}{\circle*{0.9}}

\put(108.5,96){\line(0,1){9}}
\put(107,105){\line(1,0){3}}
\put(110,105){\line(0,1){5}}

\put(107,110){\line(1,0){4.5}}
\put(107,110){\circle*{0.9}}
\put(111.5,110){\line(0,1){9}}
\multiput(110,119)(3,0){2}{\circle*{0.9}}
\put(110,119){\line(1,0){3}}

\put(117.5,91){\line(0,1){12}}
\multiput(116,103)(3,0){2}{\circle*{0.9}}
\put(116,103){\line(1,0){3}}

\put(104,105){\circle*{0.9}}
\put(107,105){\line(-1,0){3}}

\put(91,67){leaf $a$ deleted}

 \end{picture}
\caption{The delete and prune operation: effect of deleting leaf $a$ or $b$ or $c$ from the top left tree.
}
\label{Fig:dpo}
\end{figure}
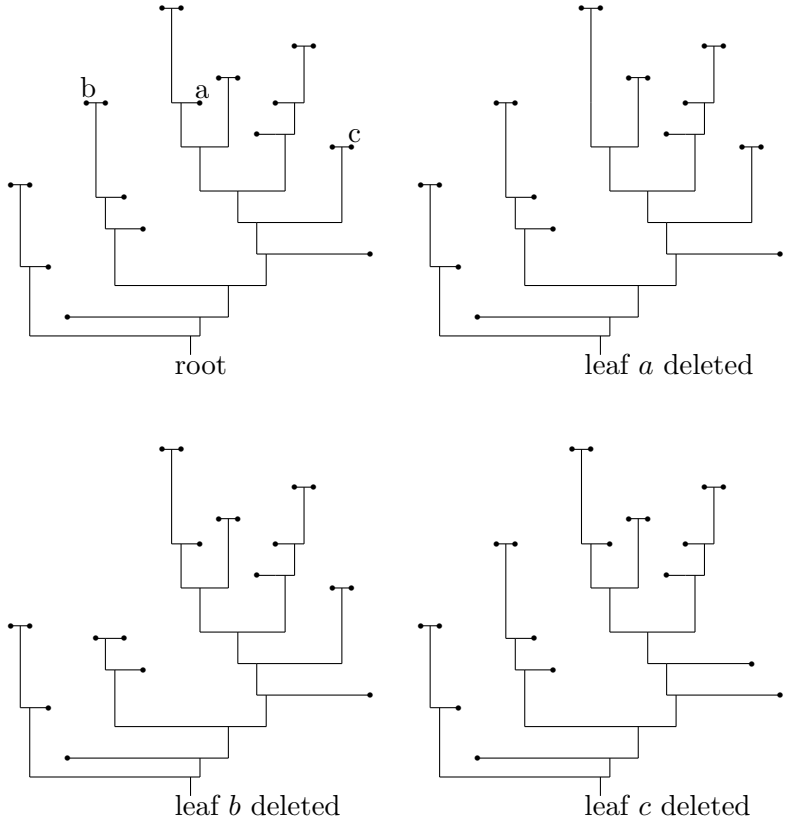

 We can now state the {\em consistency property of CTCS}.

 \begin{Theorem} 
 \label{T:consistent}
 The operation ``delete and prune leaf $n+1$ from $\CTCS(n+1)$" gives a tree distributed as $\CTCS(n)$.
 \end{Theorem} 
 So we can construct an infinite {\em consistent growth process} 
 $(\CTCS(n), n = 1,2,3,\ldots)$ such that, for each $n$,  
 ``delete and prune leaf $n+1$ from a realization of $\CTCS(n+1)$" gives exactly a realization of $\CTCS(n)$.
 In particular, the joint distribution $(\CTCS(n+1), \CTCS(n))$
will determine the associated conditional distribution of $\CTCS(n+1)$ given  $\CTCS(n)$, 
which turns out to be described by an explicit {\em growth algorithm}, stated below.

We know two proofs of Theorem \ref{T:consistent}. A proof via explicit formulas for the distributions, which will immediately provide
the required conditional distributions and growth algorithm, is in \cite[Appendix A]{beta3-arxiv}.
An alternative, more conceptual  proof with fewer calculations, is in \cite[section 2.4]{beta3-arxiv}.
Because $\DTCS(n)$ is embedded in $\CTCS(n)$, we automatically see\footnote{In previous versions it was incorrectly stated that the consistency property does not hold in discrete time.  Sorry.} 
that the consistency property holds in discrete time.  
However, the growth algorithm involves the real edge-lengths, and does not have any simple analog in discrete time.

In the context of {\em growth} of trees, it is more evocative to use the word {\em buds} instead of {\em leaves}, which we use in the following.
In Figure \ref{Fig:dpo} we see {\em side-buds} such as $a$, and end {\em
  bud-pairs} such as $b,c$.

We start with $\CTCS(1)$, which has a single bud at the root.

 \paragraph{The growth algorithm.}
   Given a realization of $\CTCS(n)$ for some $n \ge 1$:
 \begin{itemize}
 \item Pick a uniform random bud; move up the path from the root toward that bud.
 A ``stop" event occurs at rate = 1/(size of sub-clade from current position).
 \item If ``stop" before reaching the target bud, make a side-bud at that point, random on left or right.
 \item Otherwise, extend the target bud into a branch of Exponential(1) length to make a bud-pair.
\end{itemize}
  
\begin{Remark}\label{R:infty}
In our standard representations of the trees, we stop at each leaf.
In what follows, it is sometimes advantageous to consider an {\em extended} representation where
we add a vertical line to infinity from each leaf;
then each leaf lies on a unique path from the root continuing up to $\infty$.
Using that representation, the growth
algorithm has an even simpler description, where the two alternative cases above are
merged into one.
\end{Remark}
 
We note also that the consistency property in Theorem \ref{T:consistent}
implies that the subtree spanned by two random leaves in $\CTCS(n)$
(together with the root) has the same distribution as $\CTCS(2)$;
this gives another proof of  Proposition \ref{P:split}.
  
 \subsection{Exploiting the  growth algorithm}
 \label{sec:exploitgr}
One might expect to be able to exploit this inductive construction to prove asymptotic results, but we have been unable to do so, yet.
One possibility is outlined in Appendix \ref{sec:coincide}. 
Another possibility: the construction is reminiscent of other structures where martingales play a useful role, for instance  urn models \cite{mahmoud}
and branching process and branching random walk  \cite{shi}, so
\begin{OP}
\label{OP:martingale}
Is there a useful martingale associated with the inductive construction?
\end{OP}

\subsection{The exchangeable partitions representation}
\label{sec:paintbox}


As mentioned above, a consequence of Theorem \ref{T:consistent}  is that we can construct a canonical {\em consistent growth process} 
$(\CTCS(n), n = 1,2,3,4,\ldots)$ of random trees in which, for each $n$, the realization of $\CTCS(n)$ is precisely the 
realization obtained from $\CTCS(n+1)$ via the ``delete leaf $n+1$ and prune" operation.
Intuitively, there must be some kind of limit object  $\CTCS(\infty) :=\bigcup_{n=1}^\infty\CTCS(n)$.
An insight is provided by Proposition \ref{P:split} 
that, in $\CTCS(n)$, the height of the branchpoint
between two distinct random leaves has exactly Exponential(1) distribution, for each $n \ge 2$.
 As $n \to \infty$ these branchpoints persist, and (by the exchangeability argument for \eqref{d2} below: analogous to the P\'{o}lya urn scheme) the proportion
 of leaves in each branch converges to a random non-zero limit.
 Here, as in Remark \ref{R:infty}, we are imagining the line representing a leaf as continuing up to $\infty$.
 So one could define the limit object  $\CTCS(\infty)$
 as a kind of tree in which the leaves have gone off 
 to $\infty$ and in which there is a unit flow from the root to infinity along the branches.
  However this is not  the usual kind of ``locally finite" infinite tree\footnote{Such as a supercritical Galton-Watson tree.},
 because a realization has a countable infinite dense set of branchpoints.\footnote{This follows from, for instance, Lemma \ref{L:speed} below.}

 Instead of seeking to formalize $\CTCS(\infty)$ as a random tree, 
 we will use an existing formalism via  Kingman's theory of exchangeable partitions.
 A standard reference is  \cite[Section 2.3]{bertoin}  -- see also \cite{Bertoin-PTRF} and  \cite[Chapter 2]{Pitman}.
Applications to tree models somewhat similar to ours, though emphasizing characterizations rather than our explicit calculations, 
have been given in \cite{haas-pitman} (see Section \ref{sec:CRT} for further discussion).
The key feature of this approach is the {\em paintbox theorem}, developed below.

Fix a level (time) $t\ge0$.
Cutting the tree $\CTCS(n)$ at level $t$ yields a partition $\Pi\nnn(t)$ of
$[n]:=\setn$ into the clades at time $t$;
that is, 
$i$ and $j$ are in the same part if and only if
the branchpoint separating the paths to leaves $i$ and $j$ has height $>t$. 
The consistent growth process serves to define
a partition $\Pi(t)$ of $\bbN:=\set{1,2,\dots}$ into  clades at time $t$;
explicitly,
$i$ and $j$ (with $i,j\in\bbN$)
are in the same part if and only if
the branchpoint separating the paths to leaves $i$ and $j$ has height $>t$,
in $\CTCS(n)$ for any $n\ge\max(i,j)$.
In other words, $\Pi(t)$ is the partition of $\bbN$ into
the clades defined by the infinite tree
$\CTCS(\infty)$.

Because each $\CTCS(n)$ is exchangeable,  $\Pi(t)$ is an exchangeable random
partition of $\bbN$, so we can exploit the theory of exchangeable
partitions. 
Denote the clades at time $t$, that is the parts of $\Pi(t)$, by
$\Pi(t)_1,\Pi(t)_2,\dots$, enumerated in order of the least elements.
In particular, the clade of leaf 1 is $\Pi(t)_1$. 
The clades $\Pi(t)_\ell$ are thus subsets of $\bbN$, and the clades of 
$\CTCS(n)$ are the sets $\Pi(t)_\ell\cap[n]$ that are non-empty.

Writing $\left| \ \cdot \ \right|$ for cardinality, it is easy to show 
\begin{Lemma}
\label{Lpe1a}
A.s., all clades $\Pi(t)_\ell$ are infinite, that is  
$|\Pi(t)_\ell|=\infty$ for every $\ell\ge1$.
\end{Lemma}
Write, for $\ell,n\ge1$,
\begin{align}\label{d1}
  \KX^{(n)}_{t,\ell}:=\bigabs{\Pi(t)_\ell\cap[n]};
\end{align}
the sequence $\KX^{(n)}_{t,1},\KX^{(n)}_{t,2},\dots$ is thus the  sequence of sizes of
the clades in $\CTCS(n)$, extended by 0's to an infinite sequence.
Lemma \ref{Lpe1a} shows that
$\KX^{(n)}_{t,\ell}\to\infty$ as \ntoo{} for every $\ell$.
By Kingman's fundamental result \cite[Theorem 2.1]{bertoin},
the asymptotic proportionate clade sizes,  
that is the limits
\begin{align}\label{d2}
  \XP_{t,\ell}:=\lim_\ntoo\frac{\KX^{(n)}_{t,\ell}}{n},
\end{align}
exist a.s.\ for every $\ell\ge1$.  Then
the random partition $\Pi(t)$ may be constructed (in distribution)
from the limits $(\XP_{t,\ell})_\ell$ by Kingman's paintbox construction, which we state as the
following theorem.
Note that
obviously 
$\XP_{t,\ell}\in\oi$, and $\sum_\ell \XP_{t,\ell}\le1$ (by Fatou's lemma);
part \ref{T:paintbox1} of the theorem below
follows since otherwise a more general version
of the paintbox construction would imply that
$|\Pi(t)_\ell|=1$ for some $\ell$ 
\cite[Proposition 2.8(ii,iii)]{bertoin},
which is ruled out by Lemma \ref{Lpe1a}.

\begin{Theorem}\label{T:paintbox}
  \begin{romenumerate}
  \item \label{T:paintbox1}
A.s.\ each $\XP_{t,\ell}\in(0,1)$, and $\sum_\ell \XP_{t,\ell}=1$.
\item\label{T:paintbox2}     
Given a realization of $(\XP_{t,\ell})_\ell$,
give each integer $i\in \bbN $ a random color $\ell$, 
with probability distribution $(\XP_{t,\ell})_\ell$,
independently for different $i$. These colors define a random partition of
$\bbN$, which has the same distribution as $\Pi(t)$.
  \end{romenumerate}
\end{Theorem}
Note that the paintbox construction in Theorem \ref{T:paintbox} starts with
the limits $\XP_{t,\ell}$, but gives as the result $\Pi(t)$ and thus also 
the partition $\Pi\nnn(t)$ for every finite $n$.

\subsection{The subordinator within $\CTCS(\infty)$}
\label{sec:sub}

The conclusion of the discussion above is that the intuitive idea of a limit  continuum tree $\CTCS(\infty)$ 
can be formalized as the process $(\Pi(t), t \ge 0)$ of partitions of $\bbN$, in the spirit of
the formalization of  {\em fragmentation processes} in \cite{bertoin}. 
As with the Brownian continuum random tree context
(Section \ref{sec:CRT}) one can study $\CTCS(\infty)$ 
as an object in itself, or as a way to prove $n \to \infty$ limits of aspects of $\CTCS(n)$.

For given $n$ the process
$(  \KX^{(n)}_{t,1} ,  t \ge 0)$ at \eqref{d1} 
is the harmonic descent chain (Section \ref{sec:HD}) $(X^{[n]}_t, t \ge 0)$ 
started at state $n$.
We have several times exploited the approximation \eqref{approx} of this $(  \KX^{(n)}_{t,1} ,  t \ge 0)$ 
in terms of the subordinator 
 $(Y_t, 0 \le t < \infty)$ with  L\'{e}vy measure
  $\psi_\infty$ and corresponding $\sigma$-finite density $f_\infty$ on
  $(0,\infty)$ defined in \eqref{muinf}, which we for convenience repeat:
\begin{equation}
 \psi_\infty[a, \infty) :=  - \log (1 - e^{-a}); 
\quad  f_\infty(a) :=  \sfrac{e^{-a}}{1 - e^{-a}}, \quad
  \ 0 < a < \infty .
  \label{muinf3}
  \end{equation} 
As suggested by \eqref{d2}, this becomes exact in the $n \to \infty$ limit.
\begin{Theorem}[\cite{beta3-arxiv} Theorem 4.5]
\label{T:exact}
Define $Y_t := - \log  \XP_{t,1}$.
Then $(Y_t, 0 \le t < \infty)$ is the subordinator 
given by \eqref{muinf3}.
Moreover, for $t\ge0$ and complex $s$ with $\Re s>-1$,
\begin{align}\label{b5}
\Ex [\cXt^s] =\Ex [ e^{-s Y_t}]  
=e^{-t(\psi(s+1)-\psi(1))}
\end{align}
where $\psi(z):=\gG'(z)/\gG(z)$ 
is the digamma function. 
\end{Theorem}
This is proved in \cite{beta3-arxiv} by calculating moments.

Regarding $\CTCS(\infty)$ as a tree, the process $(\XP_{t,1}, t \ge 0)$ is the proportionate size of the subclade at time $t$, as one moves a speed $1$
down the path to a uniform random leaf on the infinite boundary.

\subsubsection{The subordinator and $\CTCS(n)$ for finite $n$}\label{sec:subn}
The subordinator above leads also to the following description of the
subclades along the path
to a given leaf, say leaf 1, in a finite tree $\CTCS(n)$, 
and thus of the HD chain and the hop-height $L_n$, and also of the height $D_n$.
(This is implicit in \cite{iksanovHD,iksanovCLT}, which has inspired the
description below.)
If we condition on the process $(P_{t,1}, t\ge0)$, then, by
the paintbox construction, any other leaf $j$, belongs to the same
clade as 1 at time $t$ with probability $P_{t,1}$.  
Define 
\begin{align}
T_{[1,j]}:=\inf\set{t:\text{1 and $j$ are in different parts in }\Pi(t)},
\qquad j\ge2,
\end{align}
 i.e., the time that the paths to 1 and $j$ in a finite tree $\CTCS(n)$
 (with $n\ge j)$ diverge.
Then it follows that conditioned on $(P_{t,1},t\ge0)$,
\begin{align}\label{tp1}
  \Pr(T_{[1,j]} >t)= P_{t,1},
\end{align}
so the conditional distribution function of $T_{[1,j]}$ is $1-P_{t,1}$.
Furthermore, still conditioned on $(P_{t,1},t\ge0)$, the variables $T_{[1,j]}$,
$j\ge2$, are i.i.d.
Hence, there exist random Uniform$(0,1)$ variables $U_j$,
 independent of $(P_{t,1}, t\ge0)$ and of each other,
such that
\begin{align}\label{tp2}
  T_{[1,j]}:=\inf\set{t:P_{t,1}\le U_j},
\qquad j\ge2.
\end{align}
The closed range of the subordinator $(Y_t,t\ge0)$ has a.s.\
Lebesgue measure 0, and thus
the same holds for the closed range of $P_{t,1}=e^{-Y_t}$.
The complement of the closed range is  the union of 
an infinite set of disjoint open intervals which we call gaps;
thus a.s.\ every $U_j$ falls in one of the gaps.
It follows from \eqref{tp2} that  $U_i$ and $U_j$
fall in the same gap if and only if $T_{[1,i]}=T_{[1,j]}$ and thus $i$ and $j$ leave the
path to 1 at the same time. In other words, for any finite $n$,
the branchpoints on the path to leaf 1 in $\CTCS(n)$
correspond to the gaps in the closed range of
$P_{t,1}$ that are occupied by at least one of $U_2,\dots,U_{n}$.
In particular, $L_n$, the hop-height of leaf 1, equals the number of
occupied gaps, and $D_n$, the height of leaf 1,
equals the time that $P_{t,1}$ jumps across the leftmost occupied gap.
Furthermore, for $k=1,\dots,L_n$,
the size of the $k$th clade containing leaf 1 equals 
1 + the total number of points $U_2,\dots,U_j$ in the leftmost $L_n+1-k$ gaps.

Since $P_{t,1}=e^{-Y_t}$,
this description of $L_n$, $D_n$, and the HD chain 
can  by a change of variables
be given in an equivalent form with i.i.d.\
Exponential(1) points $E_2,\dots,E_n$
thrown into the gaps of the closed range of the
subordinator $(Y_t,t\ge0)$.
This description
was found by
Iksanov  \cite{iksanovHD,iksanovCLT} in a different way,
using results from the theory of regenerative compositions,
and used by him to give proofs of 
Theorem \ref{T:alimit} and the CLT parts of Theorems
\ref{TNorm} and \ref{Lnorm}.

\subsection{Roots of subtrees}
\label{steel}

The following remarkable result was discovered in \cite[Theorem 1(i) and Theorem 2(ii)]{steel2025}.

\begin{Proposition}
\label{P:roots}
For $2 \le k < n$, consider $\DTCS(k)$ as the subtree of the tree $\DTCS(n)$ induced by $k$ random leaves.
Let $D_{n,k}$ be the discrete height, within $\DTCS(n)$, of the first split of $\DTCS(k)$.
So the event $\{D_{n,k} = 0\}$ is the event that the first split in the subtree occurs at the first split of the tree.
Then\\
(i) $\Pr(D_{n,k} = 0) = h_{k-1}/h_{n-1}$.\\
(ii) $ \frac{h_{k-1} D_{n,k}}{\log n} \to_d \Exp(1) \mbox{ as $n \to \infty$ for fixed $k \ge 2$.}$
\end{Proposition}
\begin{proof}
Assertion (i) is conceptually straightforward: 
by conditioning on the first split of the tree
\[ \Pr(D_{n,k} = 0) = \sum_{i=1}^{n-1} q(n,i) (1 - b(i,k) - b(n-i,k))
\]
for $b(i,k) = \frac{i(i-1) \cdot \cdot (i-k+1)}{n(n-1) \cdot \cdot (i-n+1)}$.
But it is not so simple to evaluate the sum. 
To illustrate use of the structures described above,
we give a quick proof based on the process $(P_{t,1}, t \ge 0)$.

For both cases we consider the discrete trees as embedded in $\CTCS(\infty)$ as the induced subtrees on leaves $1,\ldots,k$ or $1,\ldots,n$.
In $\CTCS(n)$, the initial rate of splits, until the first split occurs, is
$h_{n-1}$. When the first split occurs, it will with probability
$p:=\Pr(D_{n,k}=0)$ split the subset
$\set{1,\dots,k}$ of vertices and thus be the first split in $\CTCS(k)$.
Hence, the initial rate of splits in $\CTCS(k)$ is $ph_{n-1}$.
But this rate is $h_{k-1}$, and thus $ph_{n-1}=h_{k-1}$.

For (ii), consider the path from the root to leaf 1
in $\CTCS(n)$,
and let $B_{n,t}$ be the number of branchpoints (i.e.\ splits) that have been
passed at time $t$.
(Thus $B_{n,t}$ is the discrete height, in $\DTCS(n)$ of the last branchpoint
before or at time $t$.)
Recall that the height $H_k$  of the first branchpoint in $\CTCS(k)$ has $\Exp(h_{k-1})$ distribution (by the continuising construction).
Now
\[ D_{n,k} = B_{n,H_k} \]
and so assertion (ii) follows using Lemma \ref{L:speed} below.
\end{proof}

\begin{Lemma}
\label{L:speed}
For any fixed $t_0$, as \ntoo,
\begin{align}
  \sup_{t\le t_0}\lrabs{\frac{B_{n ,t}}{\log n}-t}\pto0.
\end{align}
[In other words, $B_{n,t}/\log n \pto t$ in the space $D[0,\infty)$.]
\end{Lemma}
\begin{proof}
Note that if $N\nn_t$ is a Poisson counting process 
on $[0,\infty)$ with constant rate $\gl_n$,
for some sequence $\gl_n\to\infty$, then 
$N\nn_t\eqd N_{\gl_n t}$ (as processes), 
where $N_t $ is a rate 1 Poisson counting process,
and since $N_t/t\asto 1$ as $\ttoo$ by the law of large numbers, it follows
easily that
for every fixed $t_0>0$, 
\begin{align}\label{slmath}
\sup_{t\le t_0}|N\nn_t/\gl_n-t|\pto0.  
\end{align}
The jumps in $B_{n ,t}$ do not occur at a fixed rate, but for time $t\in[0,t_0]$,
the rate is between $h_{n-1}=\log n + O(1)$ and $h_{M(n,t_0) -1}$, where $M(n,t_0)  := X^{(n)}_{t_0}$ is the
size of the 
clade at $t_0$. 
Furthermore, by \eqref{d2}, $M(n,t_0) /n\asto P_{t_0,1}$, 
and since $P_{t_0,1}>0$ a.s., it follows
that $\log M(n,t_0)  / \log n \asto 1$ and thus $h_{M(n,t_0) -1}/\log n\asto 1$.
Hence the result follows by conditioning on $P_{t_0,1}$ and 
sandwiching $B_{n,t}$ between two Poisson counting
processes with rates $\log n$ and  $(1-\eps)\log n$ for a fixed $\eps>0$,
and then letting $\eps\to0$.
 \end{proof}

\begin{Remark}
This argument also shows that the branchpoints in $\CTCS(\infty)$ are dense, as mentioned earlier.
\end{Remark}

\subsection{Proving Theorem \ref{T:alimit} via study of $\CTCS(\infty)$}
\label{sec:surprise}


Having the exchangeable formalization of $\CTCS(\infty)$ does not help with our first foundational result (the CLT for leaf-heights: Theorem \ref{TNorm}),
but (somewhat surprisingly) it does lead to an alternate proof of the second (the occupation measure: Theorem \ref{T:alimit}).
This is surprising because convergence of $\CTCS(n)$ to  $\CTCS(\infty)$ seems a kind of ``global" convergence, whereas 
 the asymptotic fringe is a ``local" limit.
The central idea of the proof is to define an infinite measure $\gU$ on $\oi$ by
\begin{align}\label{e7}
  \gU :=\intoo \cL(\XP_{t,1})\dd t.
\end{align}
 Formula \eqref{b5} tells us the moments of the measure $\gU$: 
\begin{align}
\label{m1}
  \intoi x^{s-1}\dd\gU(x)
=\intoo \Ex \XP_{t,1}^{s-1}
= \frac{1}{\psi(s)-\psi(1)},
\qquad \Re s>1.
\end{align}
So this is the Mellin transform of $\gU$.
We do not know how to invert the transform to obtain an explicit formula for $\gU$, but what is relevant to us here is the behavior of $\gU$ near $0$, as follows.
\begin{Lemma}[\cite{beta3-arxiv} Lemma 6.1 and \cite{beta4-arxiv} Lemma 6.1]
\label{LM}
  Let $\gU$ be the infinite
measure on $\oio$ having the Mellin transform
\eqref{m1}.
Then $\gU$ is absolutely continuous, with a continuous density $f(x)$ on $(0,1)$
that satisfies
\begin{align}\label{lm}
f(x)
= \frac{6}{\pi^2x}+O\bigpar{x^{-s_1}+x^{-s_1}|\log x|\qw},
\end{align}
uniformly for $x\in(0,1)$, 
where $s_1\doteq -0.567$ is the largest negative root of $\psi(s)=\psi(1)$.
In particular, 
for $x\in(0,\frac12)$ say,
\begin{align}\label{lmo}
f(x)
= \frac{6}{\pi^2x}+O\bigpar{x^{-s_1}} \mbox{ as } x \downarrow 0
.\end{align}
\end{Lemma}
Using the consistent sampling property of $\CTCS(n)$ and size-biasing, one
can derive%
\footnote{The details are in {\cite[Section 6]{beta3-arxiv}} except for
  this formula for $a(n,j)$ which is only implicit in
\cite{beta3-arxiv} but stated in \cite[Section 5]{beta4-arxiv} (which refers to
\cite{beta3-arxiv} for proof).}
an exact formula: for  $j \ge 2$
\[
a(n,j) = \sfrac{j}{n} h_{j-1} {n \choose j} 
\int_0^1   x^{j-1}(1-x)^{n-j}  d\gU(x) .
\]
 Combining with \eqref{lmo}, some calculus is sufficient  to prove $a(n,j) \to a(j)$.

 \section{Further aspects and open problems}
 \label{sec:misc}
 
 
There is an extensive literature (see e.g.\ \cite{drmota,janson21,legall,lyons}) 
on many different aspects of many different models of random trees.
In addition to the specific Open Problems mentioned already,
there are many further aspects of our model that could be studied.
We outline a few in this section.

\subsection{Inspiration from the drawn cladogram representation}
\label{sec:drawn}
A perhaps novel aspect of random trees arises from considering how cladograms are actually drawn on paper,
as illustrated in Figure \ref{Fig:cladogram}.
In the familiar models of random trees starting with the Galton-Watson tree, it is natural to study
the {\em width profile process}, the number of vertices at each height from the root  \cite{drmota97}.
In contrast a cladogram is drawn with all the leaves at the same ``level zero".  
So one could measure ``height" with reference to that level, but this depends on precisely how one draws
the cladogram.

There is in fact a convention implicit in Figure  \ref{Fig:cladogram}.
Each clade-split $\chi \to (\chi_1,\chi_2)$,
is represented by a horizontal line at some {\em draw-height} $\dh(\chi)$.
The draw-height depends on the shape of the subtree at $\chi$, not merely on its size $|\chi|$. 
For the usual convention, 
setting $\dh(\chi) = 0$ for a leaf (where $|\chi| = 1$), we define $\dh$
inductively%
\footnote{The ``maximum" in this rule is somewhat reminiscent of the classical 
{\em Horton–Strahler} statistic \cite{devroye21} in river networks, though we do not see any precise relation.  See \cite{horton1} for a recent connection 
with the Brownian CRT  in the context of asymptotics of uniform binary trees.}
for $|\chi| >1$:%
\begin{equation} 
\chi \to (\chi_1,\chi_2): \quad \dh(\chi) = 1 + \max(\dh(\chi_1), \dh(\chi_2))   . 
\label{htrec}
\end{equation}   
In particular, if $|\chi| = 2$ then $\dh(\chi) = 1$, and
if $|\chi| = 3$ then $ \dh(\chi) = 2$, but for larger clades,
$\dh(\chi)$ is not determined by the size:
a clade of size $4$ may have draw-height = 2 or 3, 
and a clade of size $8$ may have draw-height = 3 or 4 or 5 or 6 or 7.

The description above is clearly the minimal way to draw a cladogram such that each vertical edge length is a natural number.
It is easy to see that  the draw-height $\dh(\bt)$ of a finite clade tree $\bt$, that is the draw-height of the root split,
equals the height of $\bt$ in its discrete  representation, 
that is the largest number of edges in the  path from the root to a leaf. 
(Indeed the recursion for tree height is exactly \eqref{htrec}.)
For a leaf at this maximal height, the draw-heights upwards from the leaf 
take successive integer values $0, 1, 2, \ldots,  \dh(\bt)$.
For a leaf at lesser height, the difference of its height from the maximal height equals the number of missing
integers in the draw-heights along the path from that leaf.

\subsubsection{Heuristics: Drawn length and drawn width}
\label{sec:drawn2}

This conventional way of drawing a cladogram suggests other questions that apparently have not been studied.
One feature of interest is the {\em drawn length} $\dl(\chi)$ of the cladogram representation $\chi$ of a tree,
that is the sum of lengths of the vertical lines in the cladogram.
This satisfies a recursion: for a split $\chi_m \to (\chi_i,\chi_{m-i})$,
\begin{equation}
 \dl(\chi_m) = \dl(\chi_i) + \dl(\chi_{m-i}) + 2 + |\dh(\chi_{m-i}) - \dh(\chi_{i})| .
\label{rec:dl}
\end{equation}
What is the drawn length $\dl(\DTCS(n))$ in our model?

Here is a heuristic analysis of  the expectation $\overline{L}_n := \Ex [\dl(\DTCS(n))]$.  
Write  
\begin{align}
\overline{H}_n := \Ex [\dh(\DTCS(n))] \sim c \log^2 n
\end{align}
as in Open Problem \ref{OPmax2}, because $\overline{H}_n$ here is a re-naming of $\Ex[L^*_n]$ there.
In our model the increment \eqref{rec:dl} is dominated by the contribution from uneven splits, so for $i < m/2$ we 
 approximate the last term in  \eqref{rec:dl} as $ |\dh(\chi_{m-i}) - \dh(\chi_{i})| \approx  \dh(\chi_{m-i}) - \dh(\chi_{i})$.
Taking expectations and using \eqref{rec:dl} recursively leads roughly to
\begin{align}
  n^{-1} \overline{L}_n 
&\approx 2 \sum_{m=2}^n a(n,m)  
 \sum_{i=1}^{m/2} q(m,i) (2 + \overline{H}_{m-i} - \overline{H}_i  )
\notag\\&
 \approx 4c \sum_{m=2}^n m^{-1} \log m \ 
  \sum_{i=1}^{m/2} q(m,i) (2 + c (\log^2 (m-i) - \log^2 i))
\notag\\&
  \approx 2c \sum_{m=2}^n   \sum_{i=1}^{m/2}  \sfrac{1}{i(m-i)} (2 + c (\log^2 (m-i) - \log^2 i))
\notag\\&
\approx c^* \sum_{m=2}^n   m^{-1} \log^3 m
.\end{align}
This leads to
\begin{OP}
\label{OP:overL}
Prove that $\overline{L}_n$ grows roughly like $n \log^4 n$.
\end{OP}
In more detail, one could consider the analog of the width process mentioned earlier, illustrated in Figure \ref{Fig:2d}.
\begin{OP}
\label{OP:drawn_width}
What can we say about the {\bf drawn width profile process} $(W(h), h \ge 0)$ for $\DTCS(n)$, for the number $W(h)$ of vertical lines that cross an interval $(h,h+1)$, 
that is  the number of clades with height $\le h$ that arise as a split of a clade with height $\ge h+1$?
\end{OP}

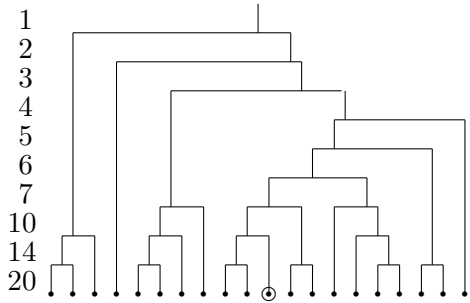
\begin{figure}[ht]
\setlength{\unitlength}{0.038in}
\begin{picture}(60,50)(-40,-5)
\multiput(0,0)(3,0){20}{\circle*{0.9}}
\put(30,0){\circle{1.8}}
\put(30,0){\line(0,1){8}}
\put(24,0){\line(0,1){4}}
\put(27,0){\line(0,1){4}}
\put(24,4){\line(1,0){3}}
\put(25.5,4){\line(0,1){4}}
\put(25.5,8){\line(1,0){4.5}}
\put(27,8){\line(0,1){4}}

\put(33,0){\line(0,1){4}}
\put(36,0){\line(0,1){4}}
\put(33,4){\line(1,0){3}}
\put(34.5,4){\line(0,1){8}}

\put(27,12){\line(1,0){7.5}}

\put(39,0){\line(0,1){12}}
\put(42,0){\line(0,1){8}}
\put(45,0){\line(0,1){4}}
\put(48,0){\line(0,1){4}}
\put(45,4){\line(1,0){3}}
\put(46.5,4){\line(0,1){4}}
\put(42,8){\line(1,0){4.5}}
\put(45,8){\line(0,1){4}}
\put(39,12){\line(1,0){6}}
\put(30,12){\line(0,1){4}}
\put(43.5,12){\line(0,1){4}}
\put(30,16){\line(1,0){13.5}}

\put(51,0){\line(0,1){4}}
\put(54,0){\line(0,1){4}}
\put(51,4){\line(1,0){3}}
\put(52.5,4){\line(0,1){16}}

\put(36,16){\line(0,1){4}}
\put(36,20){\line(1,0){16.5}}

\put(57,0){\line(0,1){24}}
\put(39,20){\line(0,1){4}}
\put(39,24){\line(1,0){18}}
\put(40.5,24){\line(0,1){4}}

\put(18,0){\line(0,1){8}}
\put(12,0){\line(0,1){4}}
\put(15,0){\line(0,1){4}}
\put(12,4){\line(1,0){3}}
\put(13.5,4){\line(0,1){4}}
\put(13.5,8){\line(1,0){4.5}}
\put(15,8){\line(0,1){4}}
\put(21,0){\line(0,1){12}}
\put(15,12){\line(1,0){6}}
\put(16.5,12){\line(0,1){16}}
\put(16.5,28){\line(1,0){23.5}}

\put(9,0){\line(0,1){32}}
\put(34.5,28){\line(0,1){4}}
\put(9,32){\line(1,0){25.5}}
\put(33,32){\line(0,1){4}}

\put(6,0){\line(0,1){8}}
\put(0,0){\line(0,1){4}}
\put(3,0){\line(0,1){4}}
\put(0,4){\line(1,0){3}}
\put(1.5,4){\line(0,1){4}}
\put(1.5,8){\line(1,0){4.5}}
\put(3,8){\line(0,1){28}}
\put(3,36){\line(1,0){30}}
\put(28.5,36){\line(0,1){4}}

\put(-6,0.5){20}
\put(-6,4.5){14}
\put(-6,8.5){10}
\put(-4.5,12.5){7}
\put(-4.5,16.5){6}
\put(-4.5,20.5){5}
\put(-4.5,24.5){4}
\put(-4.5,28.5){3}
\put(-4.5,32.5){2}
\put(-4.5,36.5){1}

  \end{picture}
\caption{Drawn width profile for the cladogram in Figure \ref{Fig:2}.}
\label{Fig:2d}
\end{figure}

\subsection{Powers of subtree sizes}
\label{sec:powers}
Another aspect of random trees that has been studied in other models
(for instance in \cite{fill-janson,SJ371}
for the model of conditioned Galton–Watson trees conditioned on total size $n$)
is the sum of $p$-powers of subtree sizes. 
Our work provides some results and conjectures for that quantity $\SS_n^{(p)}$ in our model,
that is 
\[  \SS_n^{(p)} := \sum_{j=2}^n N_n(j) j^p     \]
where $N_n(j)$ is the number of size-$j$ clades that ever arise in our model.
By \eqref{jun12} and Theorem \ref{T:alimit} we have 
\[ 
\Ex[ N_n(j)] 
= n a(n,j)/j
\sim n a(j)/j .
\]
So for $- \infty < p < 1$ we expect that
\begin{equation}  
n^{-1} \Ex[\SS_n^{(p)}]  
= \sum_{j \ge 2} a(n,j) j^{p-1}
\to \sum_{j \ge 2} a(j) j^{p-1}  < \infty.
\label{power1}
\end{equation}
For $p = 1$ we have the identity, conditioning on the random tree $\tree_n$,
\[ \SS_n^{(1)} = n \cdot \Ex [L_n \mid \tree_n] \]
and so by Theorem \ref{TL1}
\begin{equation}
\Ex[\SS_n^{(1)}]     \sim \tfrac{n}{2\zeta(2)}  \log^2 n 
\label{power2}
\end{equation}
as observed\footnote{So $n^{-1} \Ex[\SS_n^{(1)}]$ is the average size of a random subtree; it is noteworthy that, in any tree model with a fringe distribution limit, 
this average size $\to \infty$ as $n \to \infty$ \cite{me-fringe}.}
 in  \cite[Theorem 2.16]{beta1}.
For $p = 2$ we are dealing with the discrete time analog ($Q^{hop}_n(t), t = 0,1,2,\ldots)$
of the sum of squares of clade sizes in Section \ref{sec:Exp1}.
Instead of the exact formulas there, we have an approximation
\[
\Ex [Q_n(t) - Q_n(t+1) \vert \FF_t] \approx Q_n(t)/\log n, \quad t = O(\log n)
\]
leading to
\[
\Ex [Q_n(t)] \approx n^2 \exp(- t/\log n),  \quad t = O(\log n) .
\]
So heuristically
\begin{equation}   
\Ex[\SS_n^{(2)}] \approx \sum_t (\Ex [Q_n(t)] - n ) \approx n^2 \log n .
\label{power3}
\end{equation}
In fact \cite{fill-janson,SJ371} study also complex powers $\alpha$, so we note
\begin{OP}
\label{OP:SS}
Give a detailed analysis of $\SS_n^{(\alpha)}$ in our model.
\end{OP}

 \subsection{Analogies with and differences from the Brownian CRT}
 \label{sec:CRT}
 
 
 The best known continuous limit of finite random tree models is the Brownian 
continuum random tree (CRT) \cite{crt2,crt3,evans,goldschmidt2}, 
which is a scaling limit of conditioned Galton-Watson trees and other ``uniform random tree" models.
How does that compare with our $\CTCS(\infty)$ model?
 
 \noindent (a)  The most convenient formalization of the Brownian CRT is as a random {\em measured metric space}, with the
Gromov-Hausdorff-Prokhorov topology \cite{ADH} on the set of all such spaces.  So one automatically has a notion of convergence in distribution.
Our formalization of $\CTCS(\infty)$ via exchangeable partitions is less amenable to rephrasing as a random element of some metric space.
 
 \noindent  (b) Our consistency  result, that $\CTCS(n)$ is consistent as $n$ increases, and exchangeable over the random leaves, 
constitutes one general approach to the construction of continuum random trees (CRTs) \cite{crt3,evans}.

 \noindent (c) Our explicit inductive construction is analogous to the line-breaking construction of the Brownian CRT \cite{crt2} and stable trees \cite{linebreak}.

 \noindent (d)  Haas et al \cite{haas-pitman} and subsequent work such as \cite{haas-miermont}
have given a detailed general treatment of self-similar fragmentations via  exchangeable partitions, though the focus there
is on characterizations and on models like the 
$-2 < \beta < -1$ case of the beta-splitting model \eqref{rule-0}.
On the range $-2 < \beta < -1$ , such models have limits which are qualitatively analogous to the Brownian continuum random tree, 
which is the case $\beta = -2$.
But how this general abstract theory applies to explicit quantitative aspects of our specific $\beta = -1$ tree model seems a little hard to extract.

 \noindent (e) We do not know if there is any relation between  $\CTCS(\infty)$ and the {\em stable trees} whose construction is studied in \cite{stable1,linebreak}, or between the class of self-similar trees studied in \cite{selfsimilarmarkov}.

 \noindent (f) The Brownian CRT has a certain ``local and global limits are consistent" property, as follows.
That CRT is the scaling limit of certain discrete random tree models, and is encoded by Brownian excursion, 
and the local weak limit of those discrete models  is a discrete infinite tree encoded by random walk-like processes.
However these two limit processes are consistent in the following sense: the local behavior of the CRT around a typical point is another continuum tree encoded by 
the two-sided Bessel(3) process  on $\Reals$, and this process is also the scaling limit of the discrete infinite tree arising as the local weak limit.
In our CTCS model, the relationship between $\CTCS(\infty)$ as a scaling limit, and the fringe distribution as a local limit, is rather harder to describe
(cf.\ the Section \ref{sec:surprise} comment that one can derive the latter from the former).
It is intuitively clear that there is a
scaling limit of the discrete fringe process itself, the limit being
representable as a point process of branchpoint positions.
\begin{OP}
\label{OP:rescale_fringe}
Study that rescaled process.
\end{OP}


 \noindent (g) It is implausible that $\CTCS(\infty)$ is as ``universal" a limit as the Brownian CRT has proved to be, but nevertheless one can ask
\begin{OP}
\label{OP:universal}
Are there superficially different discrete models whose limit is the same $\CTCS(\infty)$?
\end{OP}
The key feature of our model seems to be subordinator approximation \eqref{approx}: can this arise in some other discrete tree model?

\subsection{Relations to the $beta(2,1)$ coalescent}
\label{sec:coalescent}
There is in fact\footnote{Thanks to Jim Pitman and an anonymous referee for pointing out this relation.} another closely related continuous model, the $beta(2,1)$
coalescent \cite{pitman1999}. 
In that model, as in $\CTCS(\infty)$, we have for every $t\ge0$ an
exchangeable random partition of $\bbN$, but 
the process 'runs in the opposite direction' and
the partitions get coarser as time increases, with blocks merging.
The $beta(2,1)$ coalescent is defined by a particular rate 
for merging of different blocks.
We see two connections to $\CTCS(\infty)$, indicating a close relation
in spite of the fact that $\CTCS(\infty)$ is a fragmentation process
while the $beta(2,1)$ coalescent is a coalescent process.

First, \cite[Proposition 26]{pitman1999} shows that
in the beta(2,1) coalescent, the frequency of singletons at time $t$
is $e^{-Y_t}$ (as a process in $t\ge0$), 
where $Y_t$ is the same subordinator as in 
Section \ref{sec:sub}; hence
by Theorem \ref{T:exact},
this frequency equals (in distribution)
$P_{t,1}$, which is the frequency of integers that
belong to the same clade in $\CTCS(\infty)$ as leaf 1.
Note, however, that this exact correspondence does not carry over to finite
$n$: The singletons at time $t$ in the beta(2,1) coalescent restricted to $[n]$
are not just the elements of $[n]$ that are singletons in the entire
infinite partition, but also any other elements of $[n]$ that lie in blocks
with no other element in common with $[n]$;
on the other hand, as we have seen in 
Section \ref{sec:consistent-exch}, the clade of leaf 1 in $\CTCS(n)$
at time $t$ is exactly the intersection of $[n]$ and the corresponding
clade in $\CTCS(\infty)$.

Secondly, 
the number of collisions in that coalescent started with $n$ singletons 
obeys the same CLT  \cite{iksanov2009} 
as does our $L_n$ in \cite[Theorem 1.7]{beta1}.
Again, this relation is not exact for finite $n$;
as discussed in \cite[Remark 1.2]{iksanov2009},
the asymptotics of the first and second moments of
this number of collisions (there denoted $X_n$) 
has the same leading term as the variable $Y_n$ there, 
which as discussed in Section \ref{sec:subn}
has the same distribution as $L_n$,
but the second-order terms differ.

This connection clearly warrants further study.

\begin{OP}
\label{OP:coalesce}
Elucidate the precise connections between $\CTCS(\infty)$ and the $beta(2,1)$ coalescent.
\end{OP}

\subsection{Methodology comments}
\label{sec:method}
We have described  methods used in this project, so what about methods not used?

\smallskip \noindent
(i) Stochastic calculus is used only briefly 
(Section  \ref{sec:bdheight} and
Appendices  \ref{sec:Exp1} and  \ref{sec:Proof1}), and only in elementary ways.
And as stated in Open Problem \ref{OP:martingale},
is there a useful martingale associated with the inductive construction?

\smallskip \noindent
(ii) We believe there should be some ``soft" proof of consistency and the exchangeable representation based on the fact
 \cite{me_clad} that the distribution $q(n,\cdot)$ arises via the first split of $n$ i.i.d.\ Uniform$(0,1)$
 points when the interval is split according to the (improper) density $1/(x(1-x))$.
But we have been unable to produce a satisfactory argument along those lines.  
See \cite{hollering} for a recent discussion of consistency for random tree models.

\paragraph{Acknowledgments.}
Thanks to Boris Pittel for extensive interactions regarding this project.
Thanks to Serte Donderwinkel for pointing out a gap in an early version,
 and to Jim Pitman and David Clancy and Prabhanka Deka for helpful comments on  early versions.
For recent (May 2024) alternative proofs mentioned in the text we thank  Brett Kolesnik,   Luca Pratelli and Pietro Rigo,
and in particular Alexander Iksanov, whose observation of the connection with  regenerative composition structures 
may lead to interesting further results.
 Thanks especially to B\'{e}n\'{e}dicte Haas for her careful explanation of how our setting fits into the general theory of exchangeable random partitions,
 which is the basis of our Section \ref{sec:paintbox}.

\begin{appendices}



 \section{Stochastic analysis proof of branchpoint height}
 \label{sec:Exp1}
 
 
 Here is a direct proof of Proposition \ref{P:split}.
 In $\CTCS(n)$ write 
$(X_n(i,t), i \ge 1)$ for the clade sizes at time $t$ and consider
\[
Q_n(t) = \sum_i X_n^2(i,t).
\]
Note that, when a size-$m$ clade is split, the effect on sum-of-squares of clade sizes has expectation
\begin{equation}
\sum_{i=1}^{m-1}  (m^2 - i^2 - (m-i)^2) \ q(m,i) 
= \frac{m}{2h_{m-1}}  \sum_{i=1}^{m-1}  2 = \frac{m(m-1)}{h_{m-1}}  .
\label{id1}
\end{equation}
If we chose some arbitrary rates $r(m,n)$ for splitting a size-$m$ clade, then
\[ \Ex [Q_n(t) - Q_n(t+dt) \vert \FF_t] = \sum_i r(X_n(i,t),n) \  \frac{X_n(i,t) (X_n(i,t)-1)}{h_{X_n(i,t)-1}}  \ dt .
\]
So by choosing $r(m,n) = h_{m-1}$ we obtain
\[ 
\Ex [Q_n(t) - Q_n(t+dt) \vert \FF_t] = (Q_n(t) - n) \ dt . 
\]
Because $Q_n(0) = n^2$ we obtain the exact formula
\begin{equation}
 \Ex [Q_n(t)] = n + (n^2  - n)e^{-t}, \ 0 \le t < \infty . 
 \label{EQn}
  \end{equation}
Now we are studying the height $B_n$ of the branchpoint between the paths to two uniform random distinct leaves of $\CTCS(n)$.
The conditional probability that both sampled leaves are in clade $i$ at time $t$ equals
$\sfrac{1}{n(n-1)} X_n(i,t) (X_n(i,t)-1) $.  So
\begin{eqnarray*}
\Pr(B_n > t) &=&
\sfrac{1}{n(n-1)} \Ex[\sum_i  X_n(i,t) (X_n(i,t)-1) ]\\
&=& \sfrac{1}{n(n-1)} \Ex[ Q_n(t) - n] \\
&=& e^{-t} \mbox{ by } \eqref{EQn} .
\end{eqnarray*}

 \section{Proof of CLT for $D_n$ via weak convergence and the martingale CLT}
 \label{sec:Proof1}
 
 
\subsection{The weak law of large numbers}
\label{sec:WLLN}
Assume we know the result $ \Ex [D_n] \sim \sfrac{6}{\pi^2}\log n $, which can be proved by a simple recurrence argument as in Section \ref{sec:recurse}.
Next we need the ``weak law"
\begin{Lemma}
\label{L:WLLN}
$\frac{D_n}{ \log n} \to_p 6/\pi^2 \mbox{ as } n \to \infty $.
\end{Lemma}
This follows from the variance estimate in Theorem \ref{TvarD1}, or the weaker version found in \cite[Theorem 1.1]{beta1} by the recursion method.
Here now is our ``probability" proof of the CLT.

\begin{proof}
First we repeat and expand upon the earlier discussion of ``the approximation calculation" in Section \ref{sec:approxc}.
The process  $\log \bX $ is itself Markov with transition rates described below.
A jump\footnote{Note these are {\em downward}  jumps, so take negative values.} of $\bX$ from $j$ to $j - i$ has height $-i$, which corresponds to a
jump of $\log \bX$ from $\log j$ having height  $\log (j-i)  - \log j =  \log (1 - i/j)$.
Define the measure $\chpsi_j$ on $(-\infty, 0)$ as the measure assigning weight $1/i$ to point $\log (1 - i/j)$, for each $1 \le i \le j-1$.
So this measure $\chpsi_j$ specifies the  heights and rates of the downward jumps of $\log \bX$ from $\log j$. 
Writing
\begin{align}
 \chpsi_j(-\infty, a] = \sum_{i = j (1-e^{a})}^{j-1} 1/i  
\end{align}
shows that there is a $j \to \infty$ limit measure in the sense
\begin{equation}
  \chpsi_j(-\infty, a]  \to  \chpsi_\infty(-\infty, a] \mbox{ as } j \to \infty, \  -\infty < a < 0  
  \label{mumu2}
  \end{equation}
where the limit $\sigma$-finite measure $\chpsi_\infty$ on $(-\infty,0)$  is the ``reflected" version of the measure $\psi_\infty$ on $(0,\infty)$ at (\ref{muinf}):
\begin{equation}
 \chpsi_\infty(-\infty, a] := - \log (1 - e^a), \quad
\chf_\infty(a) := \sfrac{e^a}{1 - e^a}, \ -\infty < a < 0  .
\label{phiphi2}
\end{equation}
In fact we use only a one-sided bound in (\ref{mumu2}), which we will get by coupling, in two stages.
We first define, for $j \ge 2$, a measure $\chnu_j$ on $(-\infty,0)$, whose total mass $h_{j-1}$ is the same as the total mass of $\chpsi_j$,
and where the reflected measures on $(0,\infty)$ satisfy the usual stochastic ordering 
$\psi_j \preceq \nu
_j$ on the line, that is to say
\[ \psi_j [0,b] \ge  \nu
_j[0,b], \ 0 < b < \infty , \quad j \ge 2. \]
To define $\chnu_j$ we simply take the mass $1/i$ of $\chpsi_j$ at point $\log (1 - i/j)$, for each $1 \le i \le j-1$,  and spread the mass over the interval $[\log (1 - (i+1)/j),  \ \log (1 - i/j)]$ with density proportional
to $\chf_\infty$.
This procedure gives a measure $\chnu_j$ with density
\[ \chg_j(u) = b_i
\chf_\infty (u), \quad u \in [ \log (1 - (i+1)/j),  \ \log (1 - i/j)], 
\quad 1\le i \le j-1 \]
on $-\infty < u < \log(1 - 1/j)$, 
and $\chg_j(u)=0$ on $\log(1 - 1/j)<u<0$,
where
\[ b_i := \frac{1}{i (\log (i+1) - \log i)}, \ i \ge 1 .\]
Clearly we have the stochastic ordering $\psi_j \preceq \nu
_j$ of the reflected measures.
Define a kernel density, for $a > 0$ and $u < 0$,
\[ \kappa(a,u) := \chg_j(u) \mbox{ on } \log (j -1) < a \le \log j,
\qquad j\ge2 ;\]
let also $\kappa(a,u)=0$ for $a\le 0$.
Now write $(Z^{(n)}_t, t \ge 0)$ for the decreasing Markov process on $(0,\infty)$, starting at $Z^{(n)}_0 =  \log n$, for which 
the  heights $u$ and rates $\kappa$ of the downward jumps from $a$ are given by $\kappa(a,u)$.
The stochastic ordering relation  $\psi_j \preceq \nu_j$  between the driving measures of the processes $\log \bX^{(n)}$ and $\bZ^{(n)}$,
together with the fact that 
$\log \bX^{(n)}$  
is stochastically monotone, imply that
we can couple the two processes so that 
\begin{equation}
 \log X^{(n)}_t \ge Z^{(n)}_t 
. 
 \label{Xnt}
 \end{equation}
Now fix small $\eps > 0$ and define a density $\chf^\eps_\infty$ on $(-\infty, 0)$ by
\begin{eqnarray}
\chf^\eps_\infty(u) &=& 2 \chf_\infty(u),    \ - \eps < u < 0 \\
                          &=& (1 + \eps)  \chf_\infty(u), \ -\infty < u \le -\eps .
\end{eqnarray}
Because $2 > b_i \downarrow 1$ as $i \to \infty$, there exists $j(\eps)$ such that
\[ \chg_j \le \chf^\eps_\infty \mbox{ for all } j >j(\eps)  \]
and therefore 
\begin{equation}
 \kappa(a,u) \le \chf_\infty^\eps(u), \quad a \ge \log j(\eps) . 
 \label{kappabd}
 \end{equation}
 Now consider the subordinator $\bY^\eps$ with  L\'{e}vy density
 $f^\eps_\infty(u):=\chf^\eps_\infty(-u)$ on $(0,\infty)$.
 The inequality (\ref{kappabd}) implies that we can couple
 $\bZ^{(n)}$ and $\bY^\eps$ as
 \begin{equation} 
  Z^{(n)}_t \ge \log n - Y^\eps_t \mbox{ while }  \log n - Y^\eps_t \ge j(\eps) .
  \label{Znt}
  \end{equation}
  Now the strong law of large numbers for $\bY^\eps$ is
  \[ t^{-1} Y^\eps_t \to \rho^\eps := \int_0^\infty u f^\eps_\infty(u)  du .
  \]
 Combining this with (\ref{Xnt}, \ref{Znt}) and noting that
 $\rho^\eps \downarrow \rho = \pi^2/6$ by \eqref{psirho}, 
 it is straightforward to deduce
that, for the coupling used here,
   \[ 
 \liminf_n D_n/\log n \ge 6/\pi^2 \mbox{ a.s. }
 \]
 Together with the upper bound on $\Ex[D_n]$ from Theorem \ref{thm1},
 this implies that $D_n/\log n \to_p 6/\pi^2$. 
 \qed
\end{proof}

\subsection{The Gaussian limit}
\label{sec:Gaussian}
Recall that $D_n$ is the time that our size-biased chain $(X^{(n)}_t)$ is
absorbed at 1.
Recalling (\ref{approx}) and (\ref{LLN}), the first-order approximation for
$(X^{(n)}_t)$ is 
\[
\log X^{(n)}_t \approx  \log n -   \rho t, \quad 0 \le t \le \rho^{-1} \log n
\]
where $\rho =  \mu^{-1} = \pi^2/6$.
To study the second-order structure, we standardize as follows. 
Subtract the first order approximation, divide by $\sqrt{\log n}$  
(the desired order of the s.d.)\ and speed up by $\log n$ (the order of $\Ex[D_n]$).
So the standardized process is
\begin{equation}
\widetilde{S}^{(n)}_s := \frac{ \log X^{(n)}_{s \log n}  - \log n + \rho  s \log n}{\sqrt{\log n}}, \quad 0 \le s \le \rho^{-1}
\label{tildeS}
\end{equation}
and essentially we want to show this converges in distribution to Brownian motion.

The first step is that the rates (\ref{lambda-rates}) determine the
infinitesimal drift rate $a(j)$ and the variance rate $b(j)$ of $\log X_t$
when $X_t = j$, as follows.
\begin{equation}
a(j) := \sum_{1 \le i \le j-1} \frac{\log i - \log j}{j-i} ; \quad 
b(j) := \sum_{1 \le i \le j-1} \frac{(\log i - \log j)^2}{j-i}  .
\label{def:ab0}
\end{equation}
Approximating the sums by integrals, 
\begin{equation}
a(j) \to -  \rho \mbox{ and }  b(j) \to \int_0^1 \frac{\log^2 y}{1- y } \ dy 
=: 
\sigma^2 = 2 \zeta(3)
\mbox{ as } j \to \infty .
\label{def:ab}
\end{equation}
We will need a bound on the former rate of convergence, but we do not need a bound for $b(j)$.
Applying  Euler's summation formula\footnote{Variants of this formula play a central role in the precise estimates in \cite{beta1}.} (Graham, Knuth, and Patashnik \cite{GKP}, (9.78)) for a smooth function $f$,
\[
\sum_{a\le i<b}\!f(i)\!=\!\int_a^b\!f(x)\,dx-\tfrac{1}{2}f(x)\Big|_{a}^{b}+\tfrac{1}{12}f'(x)\Big|_a^b+O\biggl(\int_a^b |f^{''}(x)|\,dx\!\biggr),
\]
to $f_j(x) = \frac{\log x - \log j}{j - x}$, 
one can show
\begin{equation}
|a(j) + \rho | = O(j^{-1} \log j) .
\label{aj}
\end{equation}

To start a proof of convergence, we need to stop the process before $X_t = O(1)$, so take the stopping time
\[
T_n : = \min\{t: \log X^{(n)}_t \le \log^{1/3} n \} 
\]
and replace \eqref{tildeS} by the stopped process
\[
S^{(n)}_s := \frac{ \log X^{(n)}_{\min(s \log n, T_n)}  - \log n + \rho  \min(s \log n,T_n)}{\sqrt{\log n}}, \quad 0 \le s < \infty .
\]
The central issue is to prove the following.
Write $(B_s, 0 \le s < \infty)$ for standard Brownian motion.
Recall $\mu = \rho^{-1} = 6/\pi^2$ and $\sigma^2 = 2 \zeta(3)$.

\begin{Proposition}
\label{PB}
$(S^{(n)}_s, 0 \le s <\infty) \to_d (\sigma B_{\min(s,\mu)}, 0 \le s < \infty)$
in the usual Skorokhod topology.
\end{Proposition}
\noindent
Granted Proposition \ref{PB}, we proceed as follows.  Clearly 
$T_n \le D_n$
and from $\Ex [D_m] \sim \mu \log m$ we have
\begin{equation}
\Ex[D_n - T_n] = O(\log^{1/3} n).
\label{EDT}
\end{equation}
Combining this with Lemma \ref{L:WLLN}, that $D_n/\log n \to_p \mu$,
we have 
\begin{equation}
  \label{Tnlim}
T_n/\log n \to_p \mu .  
\end{equation}
From Proposition  \ref{PB} at $s = T_n/\log n$ we have
\[ S^{(n)}_{T_n/\log n} \to_d \sigma B_\mu =_d \mathrm{Normal}(0,\mu \sigma^2) \]
and then from the definition of $S^{(n)}_s$
\[
\frac{T_n - \mu \log n}{\mu \sqrt{\log n}} \to_d  \mathrm{Normal}(0,\mu \sigma^2).
 \] 
Using (\ref{EDT})  again, we can replace $T_n$ by $D_n$, and then rewrite as
\[ 
\frac{D_n - \mu \log n}{\sqrt{\log n}} \to_d  \mathrm{Normal}(0,\mu^3 \sigma^2)
\]
as in  Theorem \ref{TNorm}.

\bigskip
{\bf Proof of Proposition  \ref{PB}.} 
Recall the infinitesimal rates $a(j)$ and $b(j)$ at \eqref{def:ab}.
Consider the Doob-Meyer decomposition 
$S^{(n)} = A^{(n)} + M^{(n)}$
in which $A^{(n)}$ is a continuous process and $M^{(n)}$ is a martingale.
In this decomposition 
$S^{(n)}_0 = A^{(n)}_0 =  M^{(n)}_0 = 0$ and $A^{(n)}_t = \int_0^t dA^{(n)}_s$ and
one readily sees that
\[
dA^{(n)}_s = ( \log^{1/2} n) \ (a(X^{(n)}_{s \log n}) + \mu^{-1}) \ ds .
\]
Here and in what follows we need only consider  $s < T_n/\log n$.

The increasing process
$<M^{(n)}>_t$ associated with $M^{(n)}$, that is the continuous component of the Doob-Meyer decomposition of $(M^{(n)})^2$,
is
\begin{equation}
d <M^{(n)}>_s = b(X^{(n)}_{s \log n}) \ ds .
\label{def:inc}
\end{equation}
To prove  Proposition \ref{PB}, it will suffice to prove \\
(i)  $A^{(n)}$ converges to the zero process \\
(ii) $M^{(n)}$ converges to the stopped Brownian motion process $(\sigma B_{\min(s,\mu)}, 0 \le s < \infty)$.

 \medskip
 \noindent
For (i) it is enough to show
\begin{equation}
(\log^{1/2} n)  \ \int_0^{T_n/\log n} | a(X^{(n)}_{s \log n}) + \mu^{-1} | \ ds \to_p  0 \mbox{ as } n \to \infty 
\label{TaX}
\end{equation}
and (because 
$X^{(n)}_{s \log n} \ge \exp(\log^{1/3} n)$ on the interval of integration)
the bound 
$|a(j) + \mu^{-1}| = O(j^{-1} \log j) $
from (\ref{aj}) is,
together with \eqref{Tnlim},
more than sufficient to prove (\ref{TaX}).

By one version of the classical martingale CLT
(Helland \cite{helland} Theorem 5.1(a)),
to prove (ii) it suffices to show that for each $t < \mu$
\begin{equation}
<M^{(n)} >_t \to_p \sigma^2 t
\label{Mnt}
\end{equation}
\begin{equation}
\rho^\eps[M^{(n)}]_t := \sum_{u \le t}  |\Delta M^{(n)}(u)|^2 1_{\{  |\Delta M^{(n)}(u)| > \eps \}} \to_{L^1} 0
\label{Lind}
\end{equation}
where the sum is over jumps $\Delta M^{(n)}(u) := M^{(n)}(u) - M^{(n)}(u-)$.
In fact,
\cite[Theorem 5.1(a)]{helland} uses instead of \eqref{Lind} the assumption
that the compensator of $\rho^\eps[M^{(n)}]_t$ tends to 0 in probability for
each $t$; this is a weaker assumption, since an increasing process and its
compensator have the same expectation, and thus \eqref{Lind} implies
convergence of the compensator to 0 in $L^1$ and thus in probability.

Now (\ref{Mnt}) is immediate from (\ref{def:ab}) and (\ref{def:inc}).
To prove (\ref{Lind}), we require only very crude bounds.
The jumps of $M^{(n)}$ are the jumps of $S^{(n)}$ which are the jumps of $(\log^{-1/2} n) \log X^{(n)}$.
So $0 > \Delta M^{(n)}(u) \ge  - \log^{1/2} n$, and it suffices to show that for fixed $\eps > 0$, 
the number of large jumps satisfies
\[
(\log n ) \ \Ex [  | \{u \le T_n/\log n: \ \Delta M^{(n)}(u) \le - \eps \}| ] \to 0 .
\]
In other words, it suffices to show
\begin{equation}
(\log n ) \ \Ex [  | \{u\le T_n:  \log X^{(n)}_{u-} - \log X^{(n)}_{u} \ge  \eps  \log^{1/2} n\}| ] \to 0 .
\label{jumps}
\end{equation}
Now from the transition rates (\ref{lambda-rates}) for $X_t$, we have
\begin{quote}
for $1 \le i \le j/2$, the rate of jumps from $j$ to some $k \le i$  \\
equals $\sum_{k=1}^i 1/(j-k) \le 2i/j$.
\end{quote}
Jumps in (\ref{jumps}) are from some state $j$ to a state below $i$ where $i/j = \exp(- \eps  \log^{1/2} n)$,
and so (for large $n$) occur at rate at most $\alpha_n:= 2  \exp(- \eps  \log^{1/2} n)$.
So the expectation in (\ref{jumps}) is at most
$\Ex [T_n] \alpha_n \sim (\mu \log n )\alpha_n$.
Now $(\mu \log^2 n) \alpha_n \to 0$ as required to establish (\ref{jumps}).

\section{Length of $\CTCS(n)$: probability proof}
\label{sec:Clength2}
We re-state Proposition \ref{P:length}:
\[ \lim_n n^{-1}  \Ex [\Lambda_n] = \sfrac{6}{\pi^2}\]
where $\Lambda_n$ is the length of $\CTCS(n)$.

\begin{proof}
We need to justify the implicit interchange of limits in the argument in Section \ref{sec:Clength1}.
Of course Fatou's lemma 
and \eqref{ELn}--\eqref{iii}
tell us that
\begin{equation}
 \liminf n^{-1} \Ex[\Lambda_n] \ge 6/\pi^2.
 \label{fatou}
 \end{equation}
We will use several pieces of previous theory.
In the context of the consistency property, Figure  \ref{Fig:dpo} illustrated the ``delete and prune" 
operation.  Deletion of each possible type of leaf ($a,b,c$ in the Figure) decreases the number of edges by $1$,
but only (b) and (c) reduce the length of the tree.
In fact  in the inductive construction, essentially the inverse
of the ``cut and prune" operation,  at each step the total length is either unchanged 
or is increased by an Exponential(1) amount.  So in particular
\begin{equation}
 \Ex [\Lambda_n] \le \Ex [\Lambda_{n+1}] \le \Ex [\Lambda_n] + 1 . 
 \label{ELn2}
 \end{equation}
 We need a fact from the analysis of the HD chain in  \cite{HDchain}.
 For $n \ge 2$
\begin{equation}
a(i) =  \sum_{j=i}^n  \hat{b}_n(j)  a(j,i), \ i \le n 
\label{bdec3}
\end{equation}
where, with $q^*(m,j)$ from \eqref{qstar},
\[ \hat{b}_n(j) := \sum_{m>n} a(m)q^*(m,j), \ 1 \le j \le n, \]
is the overshoot distribution, that is the distribution of the state where the chain enters $[[1,n]]$.
Dividing \eqref{bdec3} by $i h_{i-1}$ and summing over $i$
\[
\sum_{i=2}^n  \sum_{j=i}^n  \hat{b}_n(j)  \frac{a(j,i)}{ih_{i-1}} 
= \sum_{i =2}^n  \frac{ a(i)}{i h_{i-1}} 
\]
and then from \eqref{ELn} for $\Ex [\Lambda_j]$ and the summation at \eqref{iii}
\begin{equation}
 \sum_{j=2}^n  \hat{b}_n(j) \frac{\Ex [\Lambda_j]}{j} = \frac{6}{\pi^2} (1 - \frac{1}{n}) .
\label{bnj}
\end{equation}
As a final ingredient, the overshoot distribution $\hat{b}_n :=
\mathrm{dist}(V_n)$ has a scaling limit%
\footnote{Explicitly, $V$ has density 
$f_V(v) = 6 \pi^{-2} \int_1^\infty  \frac{1}{x(x-v)} dx 
= 6\pi^{-2}\,\frac{-\log(1-v)}{v} $.}
$n^{-1}V_n \to_d V$ where $V$ has support $[0,1]$.

To complete a proof by contradiction, suppose
\[  \limsup n^{-1} \Ex[\Lambda_n] > 6/\pi^2.\]
Then, using \eqref{ELn2}, there exist $\eps > 0$ and infinitely many $n_0$ such that 
$j^{-1} \Ex[\Lambda_j] \ge 6/\pi^2 + \eps$
for all $n_0(1-\eps) \le j \le n_0$. 
But this and \eqref{fatou} and the scaling limit for $\hat{b}_n$ imply
\[ \limsup_{n}  \sum_{j=2}^n  \hat{b}_n(j) \frac{\Ex [\Lambda_j]}{j} \ge \frac{6}{\pi^2} + \eps \Pr(1-\eps < V < 1) > \frac{6}{\pi^2} \]
contradicting \eqref{bnj}.
\qed
\end{proof}

\section{A hidden symmetry?}
\label{sec:coincide}


From Proposition \ref{P:length} and \eqref{talimit} we see that
$\ell := \lim_n n^{-1} \Ex [\Lambda_n]$
and $a(2)$ are both equal to $6/\pi^2$.
There are two different implications of ``$a(2) = \ell$".
First, it implies that (asymptotically) exactly half of the total length is in the ``terminal" edges to a bud-pair.
Second,
in the inductive construction we expect that as $n \to \infty$ there are limit probabilities 
for the three types of placement of the new bud:
\begin{itemize}
\item $p^\uparrow$ is the probability of a  branch extension
\item $p^\to$ is the probability of a side-bud addition
\item $p^\nearrow$   is the probability of a side-bud extension.
\end{itemize}
Now observe
\[  p^\uparrow + p^\nearrow = \ell \]
because these are the cases where the tree length increases by a mean length $1$.
And
\[ 2p^\nearrow = a(2) \]
because this is the only case where the number of buds in pairs increases, by $2$. 

So the assertion $\ell = a(2)$ is equivalent to the assertion $p^\uparrow = p^\nearrow$.
\begin{OP}
\label{OP:symmetry}
Is the fact  $p^\uparrow = p^\nearrow$ a consequence of some kind of symmetry  for the shape of the tree?
\end{OP}

\paragraph{A variance heuristic.}
Assuming the limit probabilities above exist, then in the
 inductive construction we are adding an edge of Exponential (1) length in a proportion $\ell$ of the steps, which strongly 
suggests  $\var(\Lambda_n) \sim \ell n= \frac{6}{\pi^2} n$,
as mentioned in Open Problem \ref{OP:varL}.

\section{List of Open Problems}
\label{sec:OPlist}



\noindent
{\bf Open Problem}  \ref{OP:alpha}.
Prove that, for  $\frac{\log r_n}{\log n} \to \alpha \in [0,1]$, we have (for correlation between heights of leaves at distance $r_n$ apart)
\[
\rho(n,r_n) \to 1 - \alpha . 
\]
(Section \ref{sec:depends}).

\noindent 
{\bf Open Problem}  \ref{OP:original}.
In the original interval-splitting model, analyze the distribution of the height of the 
leaf $i(n)$ in $\DTCS(n)$ and $\CTCS(n)$.
(Section \ref{sec:depends}).

\noindent
{\bf Open Problem} \ref{OPmax}. 
Show that the height of $\CTCS(n)$ satisfies $D^*_n \sim c \log n$ in probability, and identify the constant $c$.
(Section \ref {sec:leaf-ht}).

\noindent
{\bf Open Problem} \ref{OPmax2}.
Show that the height of $\DTCS(n)$ satisfies $L^*_n \sim c \log^2 n$ in probability, and identify the constant $c$.
(Section \ref{sec:htDT}).


\noindent
{\bf Open Problem}  \ref{OP:varNn}.
Prove that $n^{-1} \var(N_n(\chi))$ converges to 
some limit $\sigma^2(\chi)$ and that the corresponding CLT holds.
(Section \ref{sec:propfringe}).

\noindent
{\bf Open Problem} \ref{OP:varL}. 
Prove that $n^{-1} \var(\Lambda_n)$ converges to $6/\pi^2$ and that the corresponding CLT holds.
(Section \ref {sec:Clength1}).

\noindent
{\bf Open Problem} \ref{OP:Nnchi}.   
Study combinatorial properties of fringe clades, for instance
\mbox{ } 
\begin{itemize}
 \item The number $K_n := \sum_\chi 1_{(N_n(\chi) \ge 1)}$ of different-shape clades within (a realization of) $\DTCS(n)$.
\item The largest clade that appears more than once within $\DTCS(n)$.
\item The smallest clade that does not appear within $\DTCS(n)$.
  \end{itemize}
(Section \ref{sec:combin}).
  
  \noindent
{\bf Open Problem}  \ref{OP:an}.   
Find explicit bounds for $|a(n,m) - a(m)|$.
In particular, prove the following Ansatz.
(Section \ref{sec:obstacle}).
\\
{\em For a non-negative sequence $(f(j), j \ge 2)$ such that $f(j) = O(j^k)$ for some $k < \infty$:

\noindent
(i) If $\sum_{i = 2}^\infty a(i)f(i) < \infty$ then $\sum_{i = 2}^n a(n,i)f(i) \to \sum_{i = 2}^\infty a(i)f(i)$.

\noindent
(ii) If $\sum_{i = 2}^\infty a(i)f(i) = \infty$ then $\sum_{i = 2}^n a(n,i)f(i) \sim \sum_{i = 2}^n a(i)f(i)$.
}

\noindent
{\bf Open Problem} \ref{OP:tree_bal}. 
Write $N^{(n)}_m$ for the number of size-$m$ clades in $\DTCS(n)$.
Study the joint distribution of $(N^{(n)}_m, 2 \le m \le n)$ in such a way that one can calculate covariances and deduce CLTs.
(Section \ref{sec:tree_balance}).

\noindent
{\bf Open Problem} \ref{OP:indices}. 
Study the distribution of these and other indices for $\DTCS(n)$ in more detail.
(Section \ref{sec:rQI}).

\noindent
{\bf Open Problem} \ref{OP:more_data}.  
Repeat the data studies of empirical fringe distributions on a larger scale.
(Section \ref{sec:rQI}).

\noindent
{\bf Open Problem} 
\ref{OP:secondorder}
If we know that, for a given functional $\Phi$, the CLT holds for the fringe process
$ \sum_{i = 0}^{m-1} \Phi(\TT_i) $, does the CLT necessarily also hold for $\DTCS(n)$, that is for
$\sum_{i = 1}^{n} \Phi(\TT^{n}_i)$? 
(Section \ref{sec:fringeterm}).

\noindent
{\bf Open Problem}  \ref{OP:martingale}.   
Is there a useful martingale associated with the inductive construction?
(Section \ref{sec:exploitgr}).


\noindent
{\bf Open Problem} \ref{OP:overL}. 
Prove that $\overline{L}_n$ grows roughly like $n \log^4 n$.
(Section \ref {sec:drawn2}).

\noindent
{\bf Open Problem} \ref{OP:drawn_width}.  
What can we say about the {\bf drawn width profile process} $(W(h), h \ge 0)$ for $\DTCS(n)$, for the number $W(h)$ of vertical lines that cross an interval $(h,h+1)$, 
that is  the number of clades with height $\le h$ that arise as a split of a clade with height $\ge h+1$?
(Section \ref{sec:drawn2}).

\noindent
{\bf Open Problem} \ref{OP:SS}.  
Give a detailed analysis of powers $\SS_n^{(\alpha)}$ in our model.
(Section \ref{sec:powers}).

\noindent
{\bf Open Problem} \ref{OP:rescale_fringe}.    
Study the (point process) scaling limit of branchpoints in the fringe  process.
(Section \ref{sec:CRT}).

\noindent
{\bf Open Problem} \ref{OP:universal}.  
Are there superficially different discrete models whose limit is the same $\CTCS(\infty)$?
(Section \ref{sec:CRT}).

\noindent
{\bf Open Problem} \ref{OP:coalesce}.
Elucidate the precise connections between $\CTCS(\infty)$ and the $beta(2,1)$ coalescent.
(Section \ref{sec:coalescent}).

\noindent
{\bf Open Problem} \ref{OP:symmetry}.  
Is the fact  $p^\uparrow = p^\nearrow$ a consequence of some kind of symmetry  for the shape of the tree?
(Appendix\SJ \ref{sec:coincide}).

  \end{appendices}


 \end{document}